\documentclass[12pt,twoside, a4paper]{article}
\usepackage{lmodern}
\usepackage{latexsym}
\usepackage[english]{babel}
\usepackage{mathrsfs}
\usepackage{ifthen}
\usepackage{url}
\usepackage{epsfig}
\usepackage{amsbsy}
\usepackage{extarrows}
\usepackage{amsfonts}
\usepackage{amsmath}
\usepackage{bbm}
\usepackage{amssymb}
\usepackage{amsthm}
\usepackage{amsxtra}
\usepackage{bbm}
\usepackage{color}
\usepackage{stmaryrd}
\usepackage{fonttable}
\usepackage[colorlinks=true]{hyperref}
\hypersetup{
    allcolors = blue
    }
\usepackage[T1]{fontenc}
\usepackage{enumerate}

\usepackage[margin=2cm]{geometry}

\usepackage{verbatim}

\usepackage{hyperref} % HAS TO BE THE LAST PACKAGE TO BE LOADED! - provides links inside the document!!!

\numberwithin{equation}{section}
\numberwithin{figure}{section}

\newtheoremstyle{thm-style-oskari}
{7pt}      % Space above
{7pt}      % Space below
{\itshape} % Body font
{}         % Indent amount (empty = no indent, \parindent = para indent)
{\scshape} % Thm head font
{.}        % Punctuation after thm head
{.5em}     % Space after thm head: " " = normal interword space; 
{}         % Thm head spec (can be left empty, meaning `normal')

\theoremstyle{thm-style-oskari}
    \newtheorem{theorem}{Theorem}[section]
    \newtheorem{proposition}[theorem]{Proposition}
    \newtheorem{corollary}[theorem]{Corollary}
    \newtheorem{lemma}[theorem]{Lemma}
     
    \newtheorem{definition}[theorem]{Definition}% howto make rm-style text inside definitions
    \newtheorem{example}[theorem]{Example}
    \newtheorem{convention}[theorem]{Convention}
    \newtheorem{remark}[theorem]{Remark}

\newenvironment{Proof}[1][Proof]{\begin{proof}[\sc{#1}]}{\end{proof}}

\newcommand{\bels}[2] {
        \begin{equation} \label{#1} \begin{split} 
                #2 
        \end{split} \end{equation}
        }
\newcommand{\bes}[1]{
        \begin{equation*}  \begin{split} 
                #1 
        \end{split} \end{equation*}
        }

\definecolor{olivegreen}{rgb}{0,0.6,0.1}

\newcommand{\bs}[1]{\boldsymbol{\mathrm{#1}}} %bold
\newcommand{\bbm}{\mathbbm} %blackboard bold
\renewcommand{\rm}{\mathrm} %upright
\renewcommand{\cal}{\mathcal} 
\newcommand{\scr}{\mathscr} 
\renewcommand{\frak}{\mathfrak} 
\newcommand{\ol}[1]{\overline{#1} \!\,} %overline
\newcommand{\wh}{\widehat}
\newcommand{\wt}{\widetilde}

\newcommand{\eps}{\epsilon}
\newcommand{\ord} {\mathcal{O}}

\renewcommand{\d}{\partial}

\renewcommand{\P}{\mathbb{P}}
\newcommand{\E}{\mathbb{E}}
\newcommand{\R}{\mathbb{R}}
\newcommand{\C}{\mathbb{C}}
\newcommand{\N}{\mathbb{N}}

\newcommand{\D}{\mathbb{D}}

\newcommand{\ee}{\mathrm{e}} %\newcommand{\me}{\mathrm{e}}
\newcommand{\ii}{\mathrm{i}} %\newcommand{\mi}{\mathrm{i}}
\newcommand{\dd}{\mathrm{d}}

\newcommand{\p}[1]{({#1})}
\newcommand{\pb}[1]{\bigl({#1}\bigr)}
\newcommand{\pB}[1]{\Bigl({#1}\Bigr)}
\newcommand{\pbb}[1]{\biggl({#1}\biggr)}

\newcommand{\cB}[1]{\Bigl\{{#1}\Bigr\}}

\newcommand{\abs}[1]{\lvert #1 \rvert}
\newcommand{\absb}[1]{\big\lvert #1 \big\rvert}

\newcommand{\norm}[1]{\lVert #1 \rVert}
\newcommand{\normb}[1]{\big\lVert #1 \big\rVert}

\newcommand{\avg}[1]{\langle #1 \rangle}

\newcommand{\scalar}[2]{\langle{#1} \mspace{2mu}, {#2}\rangle}

\DeclareMathOperator{\diag}{diag}
\DeclareMathOperator{\tr}{Tr}

\DeclareMathOperator{\supp}{supp}

\DeclareMathOperator{\re}{Re}
\DeclareMathOperator{\im}{Im}

\DeclareMathOperator{\dist} {dist}                
\DeclareMathOperator*{\spec}{Spec}						%Spectrum

\newcommand{\1} {\mspace{1 mu}}
\newcommand{\2} {\mspace{2 mu}}
\newcommand{\msp}[1] {\mspace{#1 mu}}

\newcommand{\mtwo}[2]
{
\left(
\begin{array}{cc}
#1 
\\
#2
\end{array}
\right)
}
\newcommand{\mfour}[4]
{
\left(
\begin{array}{cccc}
#1 
\\
#2
\\
#3
\\
#4
\end{array}
\right)
}

\begin{document}
\title{Randomly coupled differential equations with elliptic correlations}
\author{
\begin{minipage}{0.3\textwidth}
\begin{center}
L\'aszl\'o Erd{\H o}s\footnotemark[1]  \\
\footnotesize {IST Austria}\\
{\text{lerdos@ist.ac.at}}
\end{center}
\end{minipage}
\begin{minipage}{0.3\textwidth}
 \begin{center}
Torben Kr\"uger\footnotemark[2]  \\
\footnotesize 
{FAU Erlangen-N\"urnberg \& University of Copenhagen}\\
{\text{torben.krueger@fau.de}}
\end{center}
\end{minipage}
\begin{minipage}{0.3\textwidth}
 \begin{center}
David Renfrew\footnotemark[3] \\
\footnotesize 
{Binghamton University}\\   
{\text{renfrew@math.binghamton.edu}}  
\end{center}
\end{minipage}
}

\footnotetext[1]{Partially supported by ERC Advanced Grant RANMAT No.\ 338804}
\footnotetext[2]{Partially supported by VILLUM FONDEN research grant no. 29369}
\footnotetext[3]{Supported by Austrian Science Fund (FWF): M2080-N35}

\date{\today}
%\date{} % empty argument removes the date!

\maketitle
\thispagestyle{empty} %%% no page numbers!!!

\begin{abstract} 
We consider the long time asymptotic behavior of a large system of  $N$ linear differential equations
with random coefficients. We allow for general elliptic  correlation structures   among the coefficients,
thus we substantially generalize our  previous work \cite{EKR} that was  restricted to the 
independent case. In particular, we  analyze a recent model in the theory of neural networks  \cite{correlationsMarti}
that specifically focused on the effect  of the  distributional asymmetry in the  random connectivity matrix $X$. We rigorously  prove and 
 slightly correct the explicit formula from \cite{MC} on the time decay as a function of the asymmetry parameter.
 Our main tool is an asymptotically precise formula for the normalized trace of  $f(X) g(X^*)$,
 in the large $N$ limit,  where $f$ and $g$ are analytic functions. 
\end{abstract}

{\bf AMS Subject Classification (2010):} \texttt{60B20}, \texttt{15B52}.

\section{Introduction}
A basic model in theoretical neuroscience \cite{SCS} to describe the {evolution} of a network of $N$ fully 
connected neurons with activation variables $u=(u_i)_{i=1}^N\in \C^N$ is
the system of linear differential equations
\begin{equation}\label{ode}
   \partial_t u_t = -u_t+ gXu_t,
\end{equation}
where $X\in \C^{N\times N}$ is the connectivity matrix and $g>0$ is a coupling parameter.
The model \eqref{ode} already assumes that the input-output transfer function has been linearized;
a common mathematical simplification of the original nonlinear model, made to study its stability properties  
\cite{VS, correlationsMarti, PhysRevLett.97.188104}.  The matrix $X$ is random and in the simplest model it is drawn 
from the {\it Ginibre ensemble}, i.e. $x_{ij}$ are independent, identically distributed (i.i.d.)  centered
complex Gaussian random variables with the convenient normalization $\E |x_{ij}|^2=\frac{1}{N}$.
This normalization keeps the spectrum bounded uniformly in $N$.
Recent experimental data,  however, indicate that in reality reciprocal connections are overrepresented \cite{overrep1,overrep2}, i.e.  $x_{ij}$ and $x_{ji}$ cannot be modeled by  independent
variables. A natural way to incorporate correlations is to keep independence among the pairs
$(x_{ij}, x_{ji})$ for different index pairs $\{i,j\}$, but assume that 
\begin{equation}\label{asym}
   \E x_{ij}x_{ji} = \varrho\2 \E |x_{ij}|^2 = \frac{\varrho}{N}
\end{equation}
for every $i\ne j$. The {\it correlation coefficient} $\varrho$ is a complex parameter of the model; $\varrho=0$ corresponding to the {\it fully asymmetric} case, while $\varrho=1$ is the {\it fully symmetric} case.
For example, the Gaussian unitary ensemble (GUE), where $X=X^*$ is Hermitian, is fully symmetric, while
the Ginibre ensemble is fully asymmetric.  The intermediate case $0<|\varrho|<1$ is called 
the  {\it elliptic  ensemble}. 

Depending on the coupling $g$,  the solution to \eqref{ode} typically grows or decays exponentially
for large times. However, there is a critical value of $g$ 
where the solution has a power law decay. Critical tuning has been the main focus for this model in
 the neuroscience literature, see e.g. \cite{HENNEQUIN20141394, nonnormal, lim2013balanced, MacNeil}
as this case exhibits complex patterns.
The decay  exponent characteristically depends on
the symmetry properties of $X$. In fact, in  \cite{ CM, MC} Chalker and Mehlig showed that the expectation of the squared $\ell^2$-norm of the solution, $\|u(t) \|_2^2$ decays as $t^{-3/2}$ in
the $\varrho=1$ fully symmetric case, while a much slower decay of $t^{-1/2}$ occurs in the fully asymmetric and {\it partially symmetric}  (or elliptic)  cases, $|\varrho|<1$.
Their analysis was mathematically not rigorous, as they used uncontrolled Feynman diagrammatic perturbation theory.  The rigorous proof
in the  two extreme
$\varrho=0$ and $\varrho=1$ cases were given in \cite{EKR}.  Motivated by the recent more detailed but still non-rigorous  analysis of the partially symmetric cases
\cite{correlationsMarti}, in the current article we give the complete  mathematical proof of all remaining intermediate cases. We also use this
opportunity to correct an error in the final formula \cite[Eq. (116)]{MC}, see \eqref{eq:correct}.

Following the original insights of \cite{MC, CM}, we consider the quantity $\frac{1}{N}\tr f(X) g(X^*)$, where $f$ and $g$ are analytic
functions outside of the spectrum of $X$. For the solution to \eqref{ode} we will later choose $f(x) = g(x)= e^{t(gx-1)}$.
 By the {\it circular}   \cite{Girko1984,tao2010,bordenave2012} and {\it elliptic laws} \cite{Girkoelliptic,nguyen2015elliptic}  it
is well known that the spectrum of $X$ becomes approximately deterministic in the large $N$ limit. 

With our approach, we also treat models of {\it elliptic-type} that are 
much more general  than the i.i.d. case with asymmetric correlation \eqref{asym}, studied in \cite{MC, CM}.  Namely, we  allow the matrix element pairs $(x_{ij}, x_{ji})$ 
 to have different joint distributions  for different index pairs  $(i,j)$. 
In particular,  our methods are not restricted to the Gaussian case, the distribution of $x_{ij}$ can be arbitrary (with some finite moment conditions).
 Finally, we compute  $\frac{1}{N}\tr f(X) g(X^*)$ with high probability and not just its expectation, as done in  \cite{ CM, MC,correlationsMarti}.

Informally, our main theorem (Theorem~\ref{thm:mainthmgen})  states that, as $N \to \infty$ 
\begin{equation}
\label{contint}
\frac{1}{N}\tr f(X)g(X^*) \to  \left( \frac{1}{2\ii \pi} \right)^2 \oint_{\gamma} d \zeta_1  \oint_{\ol \gamma} d \ol{\zeta}_2 f(\zeta_1) g({\ol{\zeta_2}}) K(\zeta_1,\zeta_2), 
\end{equation}
where $K(\zeta_1,\zeta_2)$ is a deterministic function that depends on the covariance structure of $X$. The contour $\gamma$ is outside of the spectrum, which we characterize. This formula immediately
follows from  contour integration  once we show that the trace of the  product of resolvents $\frac{1}{N}\tr (X-\zeta_1)^{-1} (X^*-\ol{ \zeta}_2)^{-1}$
 converges to  $K(\zeta_1,\zeta_2)$.  Using Girko's Hermitization trick \cite{Girko1984} and a linearization, we reduce this problem to understanding the derivative
of a single resolvent of a larger Hermitian matrix with a specific block structure.
 Spectral analysis of large Hermitian matrices has been thoroughly developed in the last years; we use
 the most recent results  on the optimal {\it local laws} outside of the pseudospectrum
  as well as on the corresponding Dyson equation  \cite{AjankiCorrelated, EKS, AEKN}; see Section \ref{sec:Hermitian random matrices and the matrix Dyson equation} for more details on the Dyson equation and optimal local laws.  
  In particular, the function $K(\zeta_1, \zeta_2)$ can be expressed in terms
  of the solution to the  {\it extraspectral Dyson equation} (see \eqref{definition of b}) that provides a deterministic approximation to the resolvent $(X-\zeta)^{-1}$ when $\zeta$ is away from the  spectrum of $X$.
In some cases, e.g. in the case of identical variances, this solution can be computed explicitly, thus recovering
all regimes studied in \cite{MC, CM}.

 For general non-Hermitian matrices, product of resolvents involves the {\it overlap} of right and left eigenvectors of $X$,
 a basic concept in the works of Chalker and Mehlig. However, overlaps are poorly understood beyond
 the Gaussian case, \cite{BourgadeDubachoverlap, Fyodorov2018, WaltersShannon}. Our method is more robust, as it uses 
 Hermitized 
 resolvents to circumvent this problem.
We use a similar approach to \cite{EKR}, where random matrices with independent entries were studied. 
A major new obstacle in the analysis lies in the singularities of kernel $K(\zeta_1, \zeta_2)$, which give the leading contribution to the double contour integral
  \eqref{contint}.  In the fully independent setup, i.e. $\varrho= 0$ of \cite{EKR} we have the explicit formula $K(\zeta_1, \zeta_2) = N^{-2}
  \sum_{ij}  (\zeta_1\ol{\zeta}_2-S)^{-1}_{ij}$,   
  where $S=(s_{ij})$ is the matrix of variances, $s_{ij} = \E |x_{ij}|^2$. 
   The  genuinely   elliptic-type  models with nontrivial 
   correlation between $x_{ij}$ and $x_{ji}$   do not allow for such a simple expression for a structural reason: the Hermitized problem 
  does not factorize as in \cite{EKR}. In fact, the key a-priori  bound on the linear stability operator associated with
  the Dyson equation at energy zero in \cite{EKR} was a direct calculation using a symmetrization transformation from \cite{AjankiCorrelated}.
  The main novel analysis  in the current paper yields a replacement for this direct estimate via  studying 
  the newly introduced {\it extraspectral Dyson equation (EDE)}, see \eqref{definition of b} later. %

  {\it Notations.} The space of $N\times N$ matrices is equipped with the standard inner product, $\langle A, B \rangle = \tr_N A^*B$, where
  $\tr_N := \frac{1}{N}\tr $ is the normalized trace. On $N$-vectors, we use $\langle \cdot, \cdot \rangle$ to denote the normalized inner product, $\langle u, v \rangle = \frac{1}{N} \sum_{i=1}^N \overline{u}_i v_i$ and $\|\cdot\|_2$ to denote the normalized Euclidean norm,  $\| v\|_2^2=\frac{1}{N}\sum_j |v_j|^2$, and $\|\cdot\|_{\infty}$ to denote the max norm. We also use  $\langle u \rangle = \frac{1}{N} \sum_{i=1}^N u_i $. When multiplying a number by an identity matrix we will often drop the identity matrix from the notation. We use $[N]$ to denote the set $\{1, \ldots, N \}$. When it is clear from context, we use $a$ to denote the scalar diagonal matrix $aI$. On matrices, we use $\| \cdot \|$ denote the operator norm induced by the Euclidean vector norm. The use any other norm will be specified locally.  
  
We will consider $\C^N$ as an algebra with entry-wise multiplication, i.e. we write $uv:=(u_iv_i)_{i=1}^N$ and $f(u):=(f(u_i))_{i=1}^N$ for vectors $u,v \in \C^N$ and  functions $f: \C \to \C$. Consistent with this notation, for any $\alpha \in \C$, we write $\alpha=(\alpha, \dots,\alpha) \in \C^N$ for the constant vector, as we did with $1 = (1, \dots, 1)$ in \eqref{definition of b}.
Furthermore, we write $\frak{D}_{u}\in \C^{N \times N}$ for the diagonal matrix with the vector $u \in \C^N$ along its diagonal. For any square matrix $R$ we denote by $\frak{r}(R)$ its spectral radius. \\

 {\it Acknowledgments.} DR would like to thank Nicolas Brunel and Johnatan Aljadeff for fruitful discussions as well sharing unpublished notes. The authors would like to thank the anonymous referees for their helpful comments.

 \section{Setup and main results} Our main results concern the asymptotic behavior of the solution to the system of ordinary differential equations \eqref{ode}, coupled by the  $N \times N$-matrix $gX-1$, where $g>0$ is a coupling parameter, $-1$ introduces an exponential damping and the  random  
 connectivity matrix $X$  %
 couples the components of the activation vector  $u_t$. 
  In this work we consider the regime of decaying activation, $u_t \to 0$ as $t \to \infty$. Thus,  $g$ is chosen such that the spectrum of the non-normal matrix $gX-1$ lies to the left of the imaginary axis $\Re \zeta =0$ in the complex plane. 
  The first step of our analysis is therefore to determine the location of 
  the eigenvalues of $X$ in the $N \to \infty$ limit.  
  Let  $\{ \lambda_i \}_{i=1}^N$  denote these eigenvalues (counted with multiplicity). 
The  {\it empirical spectral measure} (ESM) associated to $X$ is 
defined by $\mu_N:= \frac{1}{N} \sum_{i=1}^N \delta_{\lambda_i}$. 

We now introduce two random matrix ensembles,  {\it the elliptic and the elliptic-type}, 
 and present the corresponding results on the decay of $u_t$.  The elliptic ensemble  is 
 a special case of the elliptic-type, where our results are more explicit. 
In both cases we make the following assumptions.

\begin{itemize}
\item[(A)] ({\it Centered entries}) All entries of $X$ are centered, i.e. 
\[ \E[x_{ij}] = 0 .\]

\item[(B)] ({\it Finite moments}) \label{assum:boundmoments}
For each $p \in \N$, there exists a $\varphi_p > 0$ such that 
\[ \E[|x_{ij}|^p] \leq \varphi_p N^{-p/2} \]
for all $1\leq i,j \leq N$.
\end{itemize}
The constants $\{\varphi_p\}$ and further constants appearing
in assumptions (1.C)-(1.D) and (2.C)--(2.F) later are called {\it model parameters}.
Now we present the simpler of the two random matrix ensembles,   the classical {\it elliptic ensemble}.

\subsection{Elliptic Ensemble}\label{sec:Ellipticassum}

The ensemble of elliptic random matrices was introduced by Girko 
\cite{Girkoelliptic} as an interpolation between Wigner random matrices and non-Hermitian random matrices with i.i.d. entries.
 In this model, each entry of $X$ has the same variance, every pair of entries $(x_{ij},x_{ji})$ is independent of all other entries, and within each of these pairs the entries are correlated.
In this case the relevant information about the deterministic measure that approximates the empirical eigenvalue distribution  is encoded in the scalar quantities $\E[|x_{ij}|^2]$ and $\E[x_{ij}x_{ji}]$; for convenience we choose $ \E[|x_{ij}|^2]=1/N $.    

Formally, we say $X$ is an {\it elliptic random matrix}  if the following hold:

\begin{itemize}
\item[(1.C)] ({\it Independent families}) \label{assum:ellipticcorrelations}
Let $(\xi_1,\xi_2)$ be a random vector in $\C^2$ and $\varsigma$ be a random variable in $\C$. 
The set $\{(x_{ij},x_{ji})\}_{1\leq i<j\leq N} \cup \{(x_{ii})\}_{1\leq i\leq N}$ is a collection of independent random elements, with  $\{(x_{ij},x_{ji})\}_{1\leq i<j\leq N}$ a family of i.i.d. copies of  $(\xi_1,\xi_2)$ and $\{(x_{ii})\}_{1\leq i\leq N}$ a family of i.i.d. copies of $\varsigma$.

\item[(1.D)] ({\it Normalization}) 
For $i\not=j$, the mixed second moments satisfy
\[  \E[|x_{ij}|^2]=\frac{1}{N}\quad  \text{ and } \quad \E[x_{ij}x_{ji}]=\frac{ \varrho}{N}, \]
for some complex parameter $\varrho$, with $|\varrho| \leq 1$.
\end{itemize}

In this case, the well known elliptic law \cite{Girkoelliptic,nguyen2015elliptic}
states that the ESM of $X$ converges  to the uniform measure on the closed domain 
\begin{equation} \label{def:ellipse} E_{\varrho} := \Big\{\zeta \in \C \; : \;  \frac{\left(\Re{\zeta}  \cos(\frac{\theta}{2}) - \Im{\zeta} \sin(\frac{\theta}{2})  \right)^2 }{(1+|\varrho|)^2} + \frac{\left(\Re{\zeta}  \sin(\frac{\theta}{2}) + \Im{\zeta} \cos(\frac{\theta}{2})  \right)^2 }{(1- |\varrho|)^2}  \leq 1 \Big\},  \end{equation}
enclosed by the ellipse $\partial E_{\varrho}$,
where $\theta \in [0,2 \pi]$ is such that $\varrho = |\varrho|e^{\ii \theta}$. If $|\varrho| =1$, then the support of the ESM degenerates to the line segment $\Re\zeta \sin(\theta/2) + \Im\zeta \cos(\theta/2) = 0$ with $|\zeta|\leq 2$. 
In what follows, it will be useful to note that the maximum value of $\Re\zeta$ on $E_{\varrho}$ is $\sqrt{1 + |\varrho|^2 + 2 \Re\varrho}.$

Our main result for the elliptic ensemble is the following theorem about the asymptotic decay of the solution $u_t$ of \eqref{ode}, where the full expansion in $\varrho$ can be explicitly computed. The  cases 
 $\varrho=0,1$ were already considered in \cite{EKR}. We now consider the remaining intermediate cases. Note that formally taking the limit as $\varrho \to 0$ in the following theorem recovers the result of \cite{EKR}.

\begin{theorem}[Asymptotics of ODE system with elliptic coupling] \label{thm:heat}
Let $X$ satisfy Assumptions (A), (B), and (1.C-D) with $\varrho $ such that $0 < |\varrho| < 1$ and let $u_t \in \C^N$ solve the linear ODE 
\eqref{ode}
with initial value $u_0$ distributed uniformly on the $N$ dimensional unit sphere, $\{u : \|u\|_2=1\} \subset \C^N $ and
coupling coefficient $0<g \leq (1+|\varrho|^2 + 2 \Re\varrho)^{-1/2}$. 

 Then there exists a constant $c_\varrho>0$ such that for any   $\epsilon\in (0,1)$ we have 
\begin{align}\label{besseldecay}
  \P\left( \left| \E_{u_0} \|u_t\|_2^2 - e^{-2t}  \sum_{j=1}^{\infty}  |\varrho|^{-j} \left| \frac{j}{tg} I_{j}(2\sqrt{\varrho} tg)  \right|^2  \right| \le N^{-1/2+\eps}  : \forall t \leq  N^{c_\varrho\epsilon}  \right) 
\geq 1 - \frac{C_{\epsilon, \nu}}{N^{\nu}}
\end{align}
for any 
$\nu \in \N$ and some constant $C_{\epsilon, \nu}$. The function $I_j$ is the $j^{th}$ modified Bessel function of the first kind. 
Here, $C_{\epsilon, \nu}$ depends  on the model parameters 
and $c_\varrho$ depends only on $\varrho$. $\E_{u_0}$ denotes expectation with respect to the initial condition
and $\P$ is the probability with respect to $X$. 
\end{theorem}

The series in \eqref{besseldecay} is convergent because $I_j(x) \sim \frac{1}{j!}(x/2)^n$ for fixed $x>0$ and  $j \to \infty$.  
We will show in Section \ref{sec:longellipse} that the infinite sum in \eqref{besseldecay} is, for large $t$, approximated by
\begin{equation} \label{eq:correct} (1+ |\varrho|^2) I_0\left(2t  g \sqrt{2\Re{\varrho} +|\varrho|^2 + 1 }  \right)  - 2\Re\left( \frac{\varrho + |\varrho|^2 }{ \varrho + 1 } \right) I_2\left(2tg \sqrt{2\Re{\varrho} +|\varrho|^2 + 1 }  \right). \end{equation}
The same asymptotics were computed in \cite[Eq. (116)]{MC} but with a slightly erroneous final formula. Our formula \eqref{eq:correct}, thus, corrects the corresponding formula in \cite{MC}.
Using the asymptotics of the Bessel functions in \eqref{eq:correct} we will then show that
\begin{equation}\label{asymp1}
e^{-2t}  \sum_{j=1}^{\infty}  |\varrho|^{-j} \left| \frac{j}{tg} I_{j}(2\sqrt{\varrho} tg)  \right|^2 \approx  
\frac{e^{2t \left( g \sqrt{2\Re{\varrho} +|\varrho|^2 + 1 } -1 \right) }}{{\sqrt{2\pi}\sqrt{ 2t  g \sqrt{2\Re{\varrho} +|\varrho|^2 + 1 }  }}}\left(1+|\varrho|^{2}  - 2\Re\left( \frac{\varrho + |\varrho|^2 }{ \varrho + 1 } \right) \right),
\end{equation}
asymptotically for large $t$. 
In particular, for the critically tuned case, $g=(1+|\varrho|^2 + 2 \Re\varrho)^{-1/2}$ we have from \eqref{besseldecay} and \eqref{asymp1}, after choosing $\epsilon=\frac{1}{4}$, say, that
\begin{equation}\label{2side1}
 \E_{u_0} \|u_t\|_2^2 \approx \frac{ 1}{2\sqrt{\pi t}}
  \left( (1+|\varrho|^{2})  - 2\Re\left( \frac{\varrho + |\varrho|^2 }{ \varrho + 1 } \right) \right) 
\end{equation}
with very high probability in the regime where $t, N\to \infty$ such that $t\le N^{c}$, for some positive constant $c $. 
\begin{remark} A similar result considering the system of ODEs in equilibrium, driven by white noise was considered by the authors of \cite{correlationsMarti}. Their  analysis is based on the formulas from \cite{MC}.
\end{remark}

\subsection{Elliptic-Type Ensemble} \label{sec:Elliptictypeassum}
In this model, we once again assume that each pair of entries $(x_{ij},x_{ji})$ is independent of all other entries, but do not assume the matrix entries are identically distributed. This model is a natural generalization of the elliptic ensemble and interpolates between Hermitian Wigner-type matrices with a variance profile; see, for instance \cite{Ajankirandommatrix} and references within, and non-Hermitian matrices with independent entries, also with a variance profile \cite{AEK-circ}, \cite{CHNR}. 
The latter case is considered in \cite{EKR}.

In this case the relevant information about the covariances of the matrix entries is encoded in the $N \times N$-matrices $S$ and $T$ defined by  their matrix elements as 
\bels{definition of S and T}{s_{ij}:=\E[|x_{ij}|^2] \quad\text{ and } \quad t_{ij}:=\E[x_{ij}x_{ji}].}
Formally, we say $X$ is an {\it elliptic-type random matrix} if in addition to (A) and (B) the following Assumption~(2.C) holds.
\begin{itemize}
\item[(2.C)] ({\it Independent families})
The set $\{(x_{ij},x_{ji})\}_{1\leq i<j\leq N} \cup \{(x_{ii})\}_{1\leq i\leq N}$ is a collection of independent random elements.
\end{itemize}
Throughout the entire paper we will assume that $X$ is of elliptic-type 
   and that the following more technical Assumptions~\mbox{(2.D-F)} are also satisfied.
\begin{itemize}
\item[(2.D)] 
 ({\it Genuinely non-Hermitian}) There is a positive constant\footnote{We use the notation $|\varrho|$ for consistency with elliptic random matrices, but the angle of $\varrho$ is irrelevant.} $|\varrho|<1$ such that
for every $1 \leq i < j \leq N$, 
\begin{equation} \label{assum:ellcorr} |\E[x_{ij} x_{ji} ]|^{2} \leq |\varrho|^2  \E[|x_{ij}|^2] \E[| x_{ji} |^2],
\quad  \mbox{i.e.} \quad  |t_{ij}|\le |\varrho| \sqrt{ s_{ij} s_{ji}}. 
 \end{equation}
Note that, this would be the Cauchy-Schwarz inequality if $|\varrho|$ were replaced by 1, therefore the condition $|\varrho|<1$ ensures we are in the genuinely non-Hermitian setting.  
\item[(2.E)] ({\it Uniform primitivity of $S$}) There is a constant $c_0>0$ and an integer $L$ such that
\begin{equation} \label{eq:prim} [S^L]_{ij} \geq \frac{c_0}{N}\,, \qquad
 [(S^*S)^L]_{ij} \geq \frac{c_0}{N}\end{equation} 
for all $1 \leq i,j \leq N$. \footnote{The uniform primitivity condition was stated incorrectly in the published version of \cite{EKR} but is correct in the arXiv version (arXiv:1708.01546v3). The first of the two formulas defining uniform primitivity in part (1) of Assumption 2.1 of \cite{EKR} should be the  formula \eqref{eq:prim}. }
\end{itemize}

\begin{itemize}
\item[(2.F)] ({\it H\"{o}lder continuity\footnote{ This condition can easily be relaxed to piecewise H\"{o}lder continuity requiring the existence of a partition $\{I_k\}_{k=1}^K$ of $[N]$  into discrete intervals such that $\min_k \abs{I_k} \ge c_1 N$ and  such that
\[ \| t_{\cdot,i}- t_{\cdot,j}\|_2  \le \frac{C_1}{N}  \Big| \frac{i-j}{N}\Big|^{1/2}  \]
whenever $i,j \in I_k$ for some ($N$-independent) constants $c_1,C_1>0$.  Here $t_{\cdot, j}$ denotes the vector $(t_{aj})_{a=1}^N$.} of $T$}) We assume that
\[
    T_{ij} = \frac{1}{N}\tau\big(\frac{i}{N}, \frac{j}{N}\big)
\]
with some $1/2$-H\"older continuous function $\tau: [0,1]^2\to \C$, i.e.
\[ 
    | \tau(x, y) - \tau(x',y')| \le C_1 [ |x-x'|+ |y-y'|]^{1/2}
\]
for any $x,x', y,y'\in [0,1]$
with some ($N$-independent) constant $C_1>0$.

\end{itemize}
 Assumption (2.F) is a regularity assumption on the input data that ensures regularity of the solution to the EDE, \eqref{definition of b} below, and certain properties of the
 self-consistent pseudo-resolvent set $\mathcal{R}$ given in 
 Definition \ref{def:Rb}, below. See \cite{AjankiQVE}, Section 11.2 for more details. 
 
Note that Assumptions (1.C-D) with $|\varrho|<1$ imply Assumptions (2.C-F)
 with $s_{ij} =\frac{1}{N}$, $t_{ij} = \frac{\varrho}{N}$ 
 and, thus, elliptic random matrices $X$ also satisfy the assumptions made for elliptic-type  ensemble.\\

The ESM of elliptic-type matrices  in this generality  has not  been considered
in the literature.  A few exceptions are in \cite{BSSbrown}, where the special case of triangular-elliptic operators are introduced and their ESM and Brown measure are computed. Additional special cases of this model were considered in \cite{girko2012theory}, where the canonical equation with name K23 is derived. In the physics literature, fixed point equations to derive the ESM are considered in \cite{KS-corrvar} and applications to the stability of ecosystems is considered in \cite{GrilliModularity}.

In Theorem~\ref{thm:heat} the coupling constant $g$ satisfies an explicit $\varrho$-dependent upper bound which coincides with the inverse of $\max_{\zeta \in E_\varrho} \Re\2\zeta $, i.e. the maximal real part among the spectral parameters in the asymptotic spectrum $E_\varrho$ of $X$. When $X$ is of elliptic-type, the location of the spectrum in the $N \to \infty$ limit is not as explicit. Before we can state the analog of Theorem~\ref{thm:heat} for the elliptic-type matrices, Theorem \ref{thm:heatelliptictype}, we need to determine the appropriate analog of the elliptic law. 
Therefore, a main technical task of the current paper is to find the appropriate generalization of the deterministic 
set $E_\varrho$ for elliptic-type ensembles since
its rightmost point determines the large time behavior of \eqref{ode} similarly to the elliptic case.

The key object in determining the generalization of the
set $E_\varrho$ is a new nonlinear equation
for the unknown vector $\frak{b}=\frak{b}(\zeta)\in \C^N$, depending on a 
complex spectral parameter $\zeta\in \C$,
that we coin the {\it extraspectral Dyson equation (EDE)}
\bels{definition of b}{
1+(\zeta + T \frak{b} (\zeta)) \frak{b}(\zeta) = 0\,,  \qquad (EDE)
}
 with the following
 important side condition on the solution:
\bels{sidecond}{
\frak{r}(\frak{D}_{\abs{ \frak{b}(\zeta)}^2} S)<1.
}
We remind the reader that vector multiplication was defined in the Notations section by considering $\C^N$ as an algebra with entry-wise multiplication. 

Equation \eqref{definition of b} is similar to the  {\it quadratic vector equation (QVE)}  that was extensively studied 
in~\cite{AjankiQVE}
 with some key differences. Unlike for the QVE in~\cite{AjankiQVE}, here 
 $T$ is not necessarily assumed to have non-negative entries  and we do not restrict 
 either $\zeta$  or the solution vector to the complex upper half plane. 
 In particular, the  standard arguments from  \cite{AjankiQVE}  ensuring 
 the existence and uniqueness of the 
 solution $ \frak{b}(\zeta)$ 
 to \eqref{definition of b} break down in the current setting.  Thus we need a new method.
Note that EDE itself depends only on the matrix $T$, but the side condition~\eqref{sidecond}
involves $S$ as well.

We now define the key concept of this paper, the  {\it self-consistent pseudo-resolvent set}
for elliptic-type ensembles:
\begin{definition}\label{def:Rb} Given an elliptic-type ensemble with (co)variance matrices $S$ and $T$
as in \eqref{definition of S and T}, we define the self-consistent pseudo-resolvent set $\cal{R}\subset \C$
as the set of all spectral parameters $\zeta$ for which 
the extraspectral Dyson equation \eqref{definition of b} with side condition~\eqref{sidecond}
has a solution $ \frak{b}= \frak{b}(\zeta)\in \C^N$.  
\end{definition}

We call the complement of $\cal{R}$, i.e. $\cal{R}^c=\C\setminus \cal{R}$,
  the  {\it  self-consistent pseudospectrum}. The following proposition summarizes properties of $\cal{R}$ and the solution to EDE.

\begin{proposition}[Solution of the extraspectral Dyson equation \eqref{definition of b}] 
\label{prp:Solution of vector equation}
  Let $X$ satisfy Assumptions (A), (B) and (2.C-F). Then 
the set $\cal{R}\subset \C$ and the vector $\frak{b}=\frak{b}(\zeta)$, solving \eqref{definition of b} and \eqref{sidecond},
from Definition~\ref{def:Rb} satisfy
 the following properties:
\begin{enumerate}
\item 
\label{prop:Uniqueness}
For any $\zeta\in \cal{R}$  the solution $\frak{b}(\zeta)$  to
 \eqref{definition of b} with \eqref{sidecond} is unique.
\item 
\label{prop:Openness}
$\cal{R}$ is open, its complement $\cal{R}^c$  is compact in $\C$ and $0\not\in \cal{R}$. 
\item 
\label{prop:Holomorphicity}
The function $\frak{b}:\cal{R}\to \C^N$ is holomorphic.
\item 
\label{prop:Asymptotics}
At infinity $\frak{b}$ has the asymptotic behavior $\zeta \2\frak{b}(\zeta) \to -1$ as $\abs{\zeta} \to \infty$.
\end{enumerate}
\end{proposition}
 The proof of Proposition~\ref{prp:Solution of vector equation} is presented at the end of Section~\ref{Subsec:Self consistent pseudospectrum via MDE}. Note that in contrast to the simpler case of elliptic ensembles, for elliptic-type matrices the self-consistent pseudo-resolvent set can have several connected components (see  Example~\ref{ex:1} in Section~\ref{sec:examples}). 

The solution $ \frak{b}(\zeta)$ of \eqref{definition of b} and \eqref{sidecond} has the interpretation of being the entrywise limit of the resolvent of $X$ away from its self-consistent pseudospectrum as $N \to \infty$, i.e. 
\[
(X-\zeta)^{-1} \approx \frak{D}_{ \frak{b}(\zeta)}
\]
with very high probability. 
We therefore call $ \frak{b}(\zeta)$  the diagonal part of the {\it self-consistent resolvent} of $X$. 
The precise statement 
is the following 
non-Hermitian analogue of the well-known {\it isotropic local law} for Wigner random matrices \cite{KnowlesYin}.

\begin{theorem}\label{prp:b as resolvent} For elliptic-type ensembles satisfying  
Assumptions (A), (B) and (2.C-F)
there is a small constant $c_*>0$, depending only on $\abs{\varrho}$ and $L$ in \eqref{eq:prim} such that the following hold:
\begin{itemize}
\item[(i)] (Concentration of spectrum) 
\bels{locspec}{
\P\pb{\exists \,\zeta \in \spec(X)\cap\cal{R} \; : \; \abs{\zeta}\ge N^{-c_*}, 
 \frak{r}(\frak{D}_{\abs{ \frak{b}(\zeta)}^2} S) \le 1- N^{-c_*}  }\le \frac{C_\nu}{N^\nu}\,
}
holds
for all $\nu\in \N$, where $C_\nu$ depends on model parameters in addition to $\nu$. 

\item[(ii)] (Isotropic local law)
\bels{isotrop}{
\P \pbb{\abs{\scalar{x}{((X-\zeta)^{-1}-\frak{D}_{ \frak{b}(\zeta)})y}}\ge \frac{N^\eps}{\sqrt{N }}\norm{x}_2\norm{y}_2}\le \frac{C_{\eps,\nu}}{N^\nu}
}
holds for all (small) $\eps>0$, $\nu \in \N$,  
deterministic vectors  $x,y \in \C^N$ and $\zeta \in\cal{R}$ with 
$\abs{\zeta}\ge N^{-c_*\epsilon}$ and 
$ \frak{r}(\frak{D}_{\abs{ \frak{b}(\zeta)}^2} S) \le 1- N^{-c_*\epsilon}$.
%$\min\{\abs{\zeta}, \Delta_\zeta\} \ge N^{-c_*\epsilon}$.
%\bels{control}{
%\frak{r}(\frak{D}_{\abs{b(\zeta)}^2} S)\le 1- N^{- c_*\eps},\quad \mbox{and}\quad |\zeta|\ge N^{-c_*}.
%}
%\bels{control}{ \min\{\abs{\zeta}, \Delta_\zeta\} \ge N^{-c_*\epsilon}
%%\Delta_\zeta\ge N^{- c_*\eps},\quad \mbox{and}\quad |\zeta|\ge N^{-c_*}.
%}
The constant  $C_{\eps,\nu}$ depends on the model parameters in addition to $\eps$ and $\nu$.
 \end{itemize}
\end{theorem}

The proof of Theorem~\ref{prp:b as resolvent} is presented at the end of Section~\ref{sec:Properties of R}. 
It is tempting to conclude from~\eqref{locspec} that a neighborhood of the self-consistent pseudospectrum $\cal{R}^c$ contains the (random) spectrum of $X$ with very high probability, and we provide numerical evidences for several concrete ensembles in Section \ref{sec:examples}.
However, for an arbitrary elliptic-type random matrix we do not have a good rigorous control on the derivative of the function 
$\zeta\mapsto \frak{r}(\frak{D}_{\abs{ \frak{b}(\zeta)}^2} S)$ near the boundary of $\cal{R}$, 
so it is difficult to ensure the effective lower bound on $1-\frak{r}(\frak{D}_{\abs{ \frak{b}(\zeta)}^2} S)$ in~\eqref{locspec} by choosing $\zeta$ a little bit away from $\cal{R}^c$. 
Fortunately, under the additional condition that $t_{ij}\ge 0$, we have such an effective control in the most important regime, namely around 
the rightmost point
of $\cal{R}^c$ that determines the long time asymptotics of~\eqref{ode}.

 The analogous concentration of the (random) spectrum to  a neighborhood of the self-consistent pseudospectrum was shown in the real elliptic case \cite{pertelliptic}, with a proof that readily generalizes to the complex case, see also Remark 2.3 (ii) of \cite{alt2021local}. In fact, when \eqref{locspec} is applied in the elliptic ensemble $\frak{r}(\frak{D}_{\abs{ \frak{b}(\zeta)}^2} S) $ is explicit; see Section \ref{sec:longellipse}, which also gives
 an alternative proof that the spectrum of $X$  concentrates for elliptic ensembles. We remark that similar
  concentration was also shown for Hermitian and non-Hermitian random matrices with very general decaying correlations among the matrix elements, Corollary 2.3 in \cite{EKS} and Theorem 2.2 in \cite{alt2020local}, respectively.

We also point out that apart from an effective lower bound on 
$1-\frak{r}(\frak{D}_{\abs{ \frak{b}(\zeta)}^2} S)$,  a lower bound separating $|\zeta|$ from 
zero is also needed in Theorem~\ref{prp:b as resolvent}. 
There is no relation between these two effective 
lower bounds.
In particular, the error term in~\eqref{isotrop} may  blow up as
$\zeta$ approaches  zero, especially since our conditions do not exclude
that $X$ has a large kernel (see Example~\ref{ex:1} below). The EDE correctly predicts
that $0\in \cal{R}^c$
(see Proposition~\ref{prp:Solution of vector equation}), but  the size of the neighborhood around the origin contained in $\cal{R}^c$ 
  is unstable for very  small $|\zeta|$.
 This technical issue is irrelevant for our application in this paper since the proofs of Theorems~\ref{thm:heat} and \ref{thm:heatelliptictype} only require evaluation of $\frak{b}$ away from $\zeta =0$.  
  \\

Now we state our main result on the asymptotic behavior of the solution to \eqref{ode} in the case when $X$ is of elliptic-type.  For this result, we additionally  assume that the matrix $T$ is entry-wise non-negative. 
 This assumption makes the analysis more tractable and  allows us to effectively bound the set $\cal{R}^c$.  Most importantly, in 
Proposition~\ref{prp:rightmost} 
 we will show that $t_{ij}\ge 0$ 
  ensures that $ \cal{R}$  is symmetric across the real axis, and that $\zeta^*:=\max_{\zeta \in \cal{R}^c} \Re \zeta$ is positive and lies in $ \cal{R}^c$.   In fact, it is the unique point  with maximal real part in $\cal{R}^c$. This property does not hold 
  for general $T$, as $\cal{R}$ can take a wide variety of shapes, see \cite[Figure 5]{GrilliModularity} for examples. Additionally, the case $T$ has non-negative entries, $t_{ij}= \E [x_{ij}x_{ji}] \ge 0$,
   is of particular interest in neuroscience, as it corresponds to an overrepresentation of reciprocal connections, often found in neural networks \cite{correlationsMarti}.

\begin{theorem}[Asymptotics of ODE system with general elliptic-type  matrix ] \label{thm:heatelliptictype}
Let $X$ satisfy Assumptions (A), (B) and (2.C-F) with $t_{ij} \geq 0$ for all $i,j$. Let $u_t \in \C^N$ solve the linear ODE 
\eqref{ode}
with initial value $u_0$ distributed uniformly on the $N$ dimensional unit sphere, in $ \C^N $. 
Set $\zeta^{*} =\max_{\zeta \in \cal{R}^c} \Re \zeta$ and assume the coupling coefficient $g$ satisfies $0<g \leq \frac{1}{\zeta^*}$.    Then we have
\begin{align}\label{besseldecaygenell}
 & \P\left( \left| \E_{u_0} \|u_t\|_2^2 - \frac{A(S,T)}{\sqrt{2 \pi g t}} e^{2t(g \zeta^* - 1 )} \right| \le N^{-1/2+\eps} + \Big[ N^{c_* \eps }  e^{-gtc} +  \frac{C}{ gt}\Big]e^{2t(g \zeta^* - 1 )}: \forall t \le N^{c_*\epsilon} \right) \\
&\geq 1 - \frac{C_{\epsilon, \nu}}{N^{\nu}} \notag
\end{align}
for any $\eps>0$, 
$\nu \in \N$ and some constants $c_*, c, C,$ and $C_{\epsilon, \nu}$. The positive constant $A(S,T)$  is explicitly  given in \eqref{eq:Defelltypecoeff} and \eqref{eq:elltypecoeff}  of Section \ref{sec:ellcorr}. 
We use $\E_{u_0}$ to denote the expectation with respect to the initial condition and $\P$ for the probability with respect to $X$. 
The constant $C_{\epsilon, \nu}$ depends on  the model parameters as well as on $ \eps, \nu$, the constants $c,C$ depend on model parameters and $c_*$ depends only on $\abs{\varrho}$ and $L$ in \eqref{eq:prim}.
\end{theorem}

In the critically coupled regime, $g = 1/\zeta^* $, this theorem recovers the $t^{-1/2}$ decay of the \mbox{$l^2$-norm} squared of the solution seen in \cite{EKR}, for the uncorrelated case. The proof is given in Section \ref{sec:ellcorr}.

 \subsection{Examples} \label{sec:examples}

 We now present some examples, which demonstrate possible behaviors of 
  the  self-consistent pseudo-resolvent set and the self-consistent resolvent. For more examples; see \cite{GrilliModularity} and \cite{KS-corrvar}.\\

\begin{example} \label{ex:1}
 \normalfont
This example shows the solution $\frak{b}(\zeta)$ to EDE~\eqref{definition of b} can blow-up in a neighborhood of $0$.
Let $N$ be divisible by 3, and consider the matrix 
\begin{equation}\label{Ematrix}
X  = \begin{pmatrix}  X_{11}  & Y +W  \\ Y^*+Z^*  & 0 \end{pmatrix},
\end{equation}
where  $X_{11}$  is an $N/3 \times N/3$ matrix and $W$, $Y$, and $Z$ are $N/3\times 2N/3$ matrices, 
each of the matrices are independent with centered, i.i.d. entries having variance $\frac{3}{N}\sigma_{11}^2$,  $\frac{3}{N}\sigma_W^2$, $\frac{3}{N}\sigma_Y^2$, and $\frac{3}{N}\sigma_Z^2$, respectively. 
We clearly have
\[
S = \frac{3}{N} \begin{pmatrix}  \sigma_{11}^2 J & (\sigma_Y^2 + \sigma_W^2) J  &  (\sigma_Y^2 + \sigma_W^2) J  \\ (\sigma_Y^2 + \sigma_Z^2) J & 0 & 0 \\  (\sigma_Y^2 + \sigma_Z^2) J & 0 & 0
\end{pmatrix}\qquad \text{ and } \qquad
T =  \frac{3}{N}\begin{pmatrix} 0 & \sigma_Y^2 J & \sigma_Y^2 J \\ \sigma_Y^2 J & 0 & 0 \\ \sigma_Y^2 J & 0 & 0\end{pmatrix},
 \] 
where $J$ is  the $N/3 \times N/3$ matrix of all ones. 

Considering the structure of $T$,
 the EDE \eqref{definition of b} simplifies to a system of two equations for two unknowns $\frak{b}^1=\frak{b}^1(\zeta), \frak{b}^2=\frak{b}^2(\zeta) \in \C$ with block-constant solution $\frak{b}(\zeta) = (\mathfrak{b}^1,\mathfrak{b}^2) \in \C^{N/3} \oplus  \C^{2N/3}$. In fact, it becomes the Dyson equation, which we will discuss in detail in Section~\ref{sec:Hermitian random matrices and the matrix Dyson equation} below, 
 \[
 -\frac{1}{\frak{b}^1} = \zeta + 2\sigma_Y^2 \frak{b}^2\,, \quad -\frac{1}{\frak{b}^2} = \zeta + \sigma_Y^2 \frak{b}^1 \qquad \text{ associated to } \quad H=\begin{pmatrix}  0& Y  \\ Y^*  & 0 \end{pmatrix}\,.
 \]
 
 An elementary calculation shows that $2\1\frak{b}^2(\zeta)\approx -1/\zeta$ for $\abs{\zeta}\ll 1$, i.e. it has a blow up singularity at $\zeta=0$.
In fact, this blow up is the signature of an atom at the origin in the asymptotic spectral density of $H$ due to
 ${\rm{rank}} \1H  < N$.

The side condition \eqref{sidecond} holds outside of a 
complex neighborhood of the support of this spectral density (consisting of two symmetrically positioned intervals on the real line, away from zero and the atom at the origin). 
 When $\sigma_W^2$, and $\sigma_Z^2$ are not too large, relative to $\sigma_Y^2$, the self-consistent pseudospectrum associated to $X$ is  given by two connected sets, separated from 0, and an  atom at 0. 

This is demonstrated in Figure \ref{fig:atom}, which shows the eigenvalues of an $1500 \times 1500$ random matrix $X$ with complex Gaussian entires, $ \sigma_{11}^2  = \sigma_Y^2 = 1$, and $\sigma_Z^2= \sigma_W^2=0.05$.

\begin{figure}[ht]
    \centering
    \includegraphics[width=.5\textwidth]{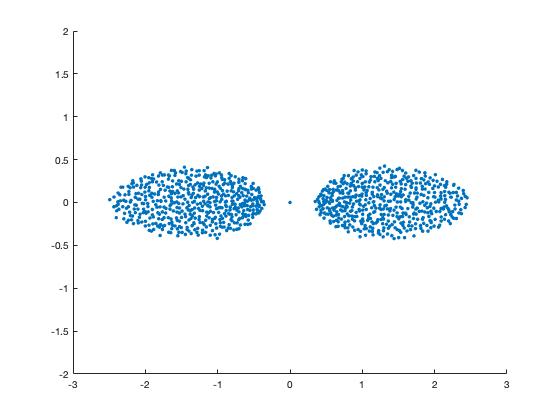}
    \caption{Eigenvalues of random matrix from Example 1}
    \label{fig:atom}
\end{figure}

\end{example}

\begin{example} \label{ex:2}
 \normalfont
Example 2 shows that the self-consistent pseudo-resolvent set of some random matrices may not be connected, i.e. 
apart from its infinite component, it  may also contain a separated island. 
In fact, Example 1 can be used to construct  such a  matrix. Consider the block-diagonal random matrix
\[  \diag(X_1, e^{-\ii \pi/k} X_2, \ldots, e^{-\ii (k-1)\pi/k} X_k) + \wt{X}  \]
where $X_1, \ldots, X_k$ are independent copies of  $X$  from~\eqref{Ematrix},
 $k$ is chosen sufficiently large that the self-consistent spectra of $X_1$ and $e^{-\ii \pi/k} X_2$  overlap.  Thus $k$ rotated copies of the spectrum of  $X$  from Fig.~\eqref{fig:atom}
 eventually enclose a compact region around the origin. Here 
 $\wt{X}$ is $kN \times kN$ random matrix with independent entries having a very small variance. 
  The matrix $\wt{X}$ is added just to make the matrix of variances  of the entire matrix primitive;
  when the variance of its entries is sufficiently small then its addition only causes the self-consistent pseudo-spectrum to change by a little. 

Using the values in the simulation in Example 1, it suffices to take $k=6$ and $\sigma_{\wt{X}}^2 = 0.0001$, as Figure \ref{fig:discon} shows.\\
\begin{figure}[ht]
    \centering
    \includegraphics[width=.5\textwidth]{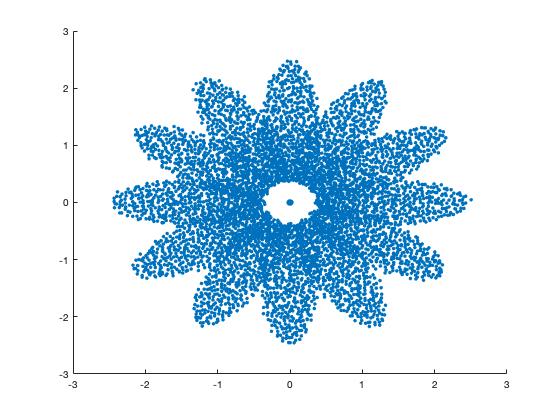}
    \caption{Eigenvalues of random matrix from Example 2}
    \label{fig:discon}
\end{figure}
\end{example}

\begin{example} \label{ex:3}
 \normalfont

Finally, we now consider an example where the location of the spectrum can be explicitly  computed. 
This model represents a system with 2 subnetworks, the connections between elements of different subnetworks, as well as connections between elements of the second subnetwork are all independent, but interconnections between elements of the first component are correlated. This model is of interest in neuroscience, where a network of neurons can be divided into an excitatory and inhibitory network, and the excitatory subnetwork is known to have an over-representation of bi-directional connections \cite{Markram}, \cite{correlationsMarti}.

Let $N$ be an even number, and $X$ be a block random matrix with the following form:

\begin{equation}\label{Ex3}
 X = \begin{pmatrix}  X^{11} & X^{12} \\ X^{21} & X^{22}  \end{pmatrix} ,
 \end{equation}
where $X^{11}$, $X^{12}$, $X^{21}$, and $X^{22}$ are each independent $N/2 \times N/2$ matrices. The matrix $X^{11}$ is an elliptic random matrix, whose entries have variance $\frac{2}{N}\sigma_{11}^2 = \frac{1}{N}$. The remaining matrices each have i.i.d. entries with variance $\sigma^2_{ij}$ in the $(i,j)$ block. 

In this case 
\[
S =  \frac{2}{N}\begin{pmatrix}  \frac{1}{2}J &  \sigma^2_{12} J \\  \sigma^2_{21} J& \sigma^2_{22} J\end{pmatrix},\qquad  T = \frac{1}{N} \begin{pmatrix} \rho J & 0 \\0 & 0 
\end{pmatrix},
\]
 where $J$ is the $N/2 \times N/2$ matrix of all ones. In the appendix, we show that the right edge  $\zeta^*$ 
  of the spectrum can be computed for this model.

In Figure \ref{fig:1}, the eigenvalues of a $2000 \times 2000$ random matrix with complex Gaussian entries with $\sigma^2_{12}=0.05$, $\sigma^2_{21}=0.5$, $\sigma^2_{22} =0.25$, and $\rho = 0.8$. In this case the right most edge point is $\zeta^* \approx 1.29$.

\begin{figure}[ht]
    \centering
    \includegraphics[width=.5\textwidth]{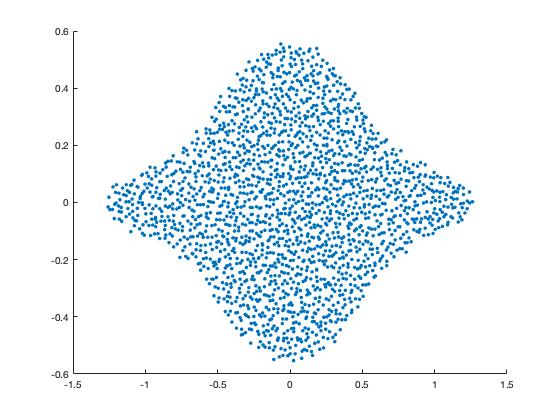}
    \caption{Eigenvalues of random matrix from Example 3}
    \label{fig:1}
\end{figure}
\end{example}

\subsection{Extended strategy of proof}

The proofs of our main results, Theorems~\ref{thm:heat} and \ref{thm:heatelliptictype}, are both based on reducing the understanding of the long time asymptotic of the differential equation \eqref{ode} to understanding the product of the resolvents associated to $X$ and $X^*$.   The solution to \eqref{ode} can be represented as $u_t= e^{t(gX-1)}u_0$ and the matrix exponential is expressed
via contour integration of the resolvent. Therefore 
 with initial value $u_0$, distributed uniformly on the $N$ dimensional unit sphere, 
 the squared norm of the solution, when averaged over the initial conditions, is given by
\begin{align*} \E_{u_0} \|u_t\|^2_2 = &\tr_N e^{t(gX^*-I)}e^{t(gX-I)} \\ =&  \left( \frac{1}{2\ii \pi} \right)^2 \oint_{\gamma} d \zeta_1  \oint_{\ol \gamma} d \ol{\zeta}_2 e^{t(g\zeta_1 + g\ol{\zeta_2} -2)  }  \tr_N (X-\zeta_1)^{-1}   (X^*-\ol{\zeta}_2)^{-1}  ,\end{align*}
where $\gamma$ is a  closed  curve that encloses the eigenvalues of $X$ traversed in the counterclockwise direction and $\ol{\gamma}$ is the same curve traversed in the clockwise direction.   This formula immediately follows 
from the residue theorem. 
 To represent the solution to \eqref{ode} we need to take the matrix exponential of $X$. More generally, we
may consider other functions $f(X)$ of $X$ as well. 
Note that, when $X$ is non-normal, due to the lack of a spectral theorem, 
  $f(X)$ can only be defined for analytic test functions $f$ via  contour integration of the resolvent,
\bels{f(X) in terms of resolvent}{f(X) =  - \frac{1}{2 \pi \ii} \oint_{\gamma} d \zeta f(\zeta) (X-\zeta)^{-1} ,  }
 where the curve $\gamma$ encircles the eigenvalues of $X$  counterclockwise.
Hence  knowledge about the asymptotic  location of the spectrum of $X$ is  necessary to define $ f(X)$. Along the proof of Theorem~\ref{prp:b as resolvent} we will see
 that in a certain sense $\spec(X)$ concentrates on the bounded set $ \cal{R}^c$.
 
The following theorem identifies the limit of $\tr_N f(X) $ and $\tr_N f(X)g(X^*)$ for analytic functions $f$ and $g$ in the general case, i.e. when the entries of $T$ are not required to be non-negative. We need to assume, however, that $f, g$ are analytic essentially on the entire unique 
unbounded connected component of $\cal{R}$, which exists because $\cal{R}^c$ is, by Proposition~\ref{prp:Solution of vector equation}, compact.

We set
\bels{def of Delta zeta}{
\Delta_\zeta := \min\{\frak{r}(\frak{D}_{\abs{ \frak{b}(\zeta)}^2} S)^{-1}-1,1\} ,
}
for $\zeta \in \cal{R}$ and note that $\Delta_\zeta>0$ is equivalent to~\eqref{sidecond}.

\begin{theorem}[Limits of analytic observables]  \label{thm:mainthmgen} 
Let $X$ be an elliptic-type matrix that satisfies Assumptions (A), (B), and \mbox{(2.C-F)}.
There exists a constant $c_*$, depending only on $\abs{\varrho}$ and $L$ in \eqref{eq:prim}, such that the following holds true. 
For any $\eps>0$
let $\cal{R}^\eps$ be the unique   unbounded connected component of $\{\zeta \in \cal{R}: \min\{\abs{\zeta},\Delta_\zeta\} \ge N^{-c_*\1\eps}\}$. Let $f$, $g$ be  analytic functions on $\cal{R}^\eps$ and $\gamma \in \cal{R}^\eps$ be a positively oriented, closed curve encircling all points in the complement of $\cal{R}$ exactly once.  Then we have
\begin{equation}\label{first} \P\left(  \left| \tr_N f(X) +  \frac{1}{2\ii \pi}  \oint_{\gamma} d \zeta f(\zeta) \avg{ \frak{b}(\zeta)}  \right|  \le \frac{  N^\epsilon }{N^{1/2}}  \right)  \ge 1 - \frac{C_{\epsilon ,\nu}}{N^\nu} \,,
 \end{equation}
as well as
\begin{align}\label{second} 
&\P\left(  \left| \tr_N f(X) g(X^*) - \left( \frac{1}{2\ii \pi} \right)^2 \oint_{\gamma} d \zeta_1  \oint_{\ol \gamma} d \ol{\zeta}_2 f(\zeta_1) g(\ol \zeta_2)
K(\zeta_1, \zeta_2)\right|
  \le \frac{ N^\epsilon }{N^{1/2}}  \right) % 
  \ge 1 - \frac{C_{\epsilon,\nu}}{N^\nu} 
\end{align}
for any $\eps\in (0,1)$,  $\nu \in \N$ and a positive constant $C_{\eps,\nu}$,  with 
\begin{equation}\label{def:Kell}
{K(\zeta_1, \zeta_2):=  \frac{1}{N}  \sum_{i,j=1}^N \left[ 
( \frak{D}_{ \frak{b}(\zeta_1)\overline{ \frak{b}(\zeta_2)}}^{-1}- S)^{-1} \right]_{ij}.}
\end{equation} 
 Here $\ol{\gamma}$ is the negative orientation of $\gamma$ and $\frak{b}$ is the
 unique solution to the EDE~\eqref{definition of b}--\eqref{sidecond}. 
 Recall $\avg{ \frak{b}(\zeta)} = \frac{1}{N} \sum_{i=1}^N  \frak{b}_i(\zeta)$. 
 In addition to $\eps,\nu$, the constant $C_{\epsilon,\nu}$ depends on $\| f|_\gamma\|_\infty+ \|g|_\gamma\|_\infty$, i.e. the largest absolute value of $f$ and $g$ on $\gamma$, as well as the  (implicit) constants in Assumptions (B) and (2.C-F). 
\end{theorem}
Implicit in the use of $K(\zeta_1, \zeta_2)$ is the statement that the matrix {$\frak{D}_{\mathfrak{b}(\zeta_1)\overline{\mathfrak{b}(\zeta_2)}}^{-1}- S$} is invertible for $\zeta_1, \zeta_2\in 
\cal{R}^\epsilon$. The proof of this theorem is given at the end of Section~\ref{sec:MDE}.

\begin{remark} \label{rem:simple}
If all of the entries of $X$ are  independent, 
 in particular $T =0$, then  $\cal{R} = \{ \zeta \; ; \; |\zeta|^2> \frak{r}(S) \}$, $\frak{b}(\zeta)=-1/\zeta$ and 
 Theorem~\ref{thm:mainthmgen} reduces to the main result  Theorem~2.3  of \cite{EKR} with 
\begin{equation}\label{eq:indker} K(\zeta_1, \zeta_2) =  \frac{1}{N} \sum_{i,j=1}^N \left[ ( \zeta_1 \overline{\zeta_2} - S)^{-1} \right]_{ij} . \end{equation}
On the other hand, if $X$ is instead an elliptic random matrix, then $\cal{R} = E_\rho^c$ 
and \eqref{def:Kell} simplifies to
\begin{equation}\label{eq:ellker}K(\zeta_1, \zeta_2) = \frac{ \frak{b}(\zeta_1) \ol{ \frak{b}(\zeta_2)}}{1 -  \frak{b}(\zeta_1) \ol{ \frak{b}(\zeta_2)}}, \end{equation}
 using that the constant vector is the eigenvector corresponding to the only non-zero eigenvalue, $1$, of $S$, 
where $ \frak{b}(\zeta) = \frac{1}{2\varrho}(-\zeta +\sqrt{\zeta^2-4 \varrho})$ solves~\eqref{definition of b}--\eqref{sidecond}. 
The square root is chosen with a branch cut along the segment $[-2\sqrt{\varrho},2\sqrt{\varrho}]	$ so that $\sqrt{\zeta^2-4 \varrho} - \zeta \to 0$ as $\zeta \to \infty$. The formula \eqref{eq:ellker} is used inside the proof of Theorem~\ref{thm:heat} to arrive at the explicit asymptotic expression for $\E_{u_0} \|u_t\|_2^2$ in \eqref{besseldecay}.
We remark that this formula can also be deduced from the non-rigorous calculations in \cite{MC}, even though it does not appear directly; see also \cite{ correlationsMarti}.
\end{remark}

The spectral analysis of a general non-Hermitian random
 matrix $X\in \C^{N\times N}$ starts with its Hermitization, i.e. defining
the $2N\times 2N$ Hermitian matrix
\bels{2x2 symmetrization}{
\bs{H}_{\zeta}=\mtwo{0 & X -\zeta}{X^* -\ol{\zeta}&0}\,,
}
and the corresponding resolvent $\bs{G}_\zeta(z) := (\bs{H}_{\zeta}-z)^{-1}$
with spectral parameter $z\in \C\setminus \spec{\bs{H}_{\zeta}}$.
Note that
\bels{resG}{
    (X-\zeta)^{-1} = \lim_{\eta\to 0+}\bs{G}_\zeta(i\eta)_{21},
}
where $\bs{G}_{21}$ indicates the $N\times N$ lower left block of $\bs{G}_\zeta$.
The main advantage of this Hermitization is that its resolvent is stable for genuinely complex
spectral parameters  $z\in \C\setminus \R$. Additionally, any point $E\in \R$ in its spectrum
can be approached by considering the resolvent at $z=E+\ii\eta$. 
Large $N$ limits of Hermitian random
matrices have therefore been extensively studied, in particular their resolvents become
approximately deterministic with their limit given by the solution $\bs{M}$ 
of a deterministic equation,
the Matrix Dyson equation (MDE), see \eqref{eq:genMDE} later.
The MDE for our special case $\bs{H}_{\zeta}$ will be given in~\eqref{2x2 MDE}
and its solution, $\bs{M}_{\zeta}(z )$ turns out to be a block diagonal matrix,
see~\eqref{form of 2x2 M} later.
  Since $X$ is of elliptic-type, the entries of $\bs{H}_{\zeta}$ have a specific correlation structure which is a  special case of
the general correlations extensively studied in~\cite{AjankiCorrelated, EKS}. In these papers
{\it local laws}, i.e. statements  crudely  of the form
\bels{locl}{  \bs{G}_\zeta(z) \approx \bs{M}_{\zeta}(z ),
}
were proven in the regime $\eta=\Im z\gg N^{-1}$.
While $\bs{H}_{\zeta}$ does not satisfy the basic {\it flatness} assumption   ubiquitous 
in~\cite{AjankiCorrelated, EKS} due to its large zero blocks, our analysis 
will be done in the regime where $\Re z=0$ is outside of the self-consistent
spectrum of $\bs{H}_{\zeta}$. Here flatness is not required for the local law~\eqref{locl}.
 Roughly speaking  this regime is
 equivalent
to $\zeta\in \cal{R}$; at least on the random matrix level we clearly  have 
\bels{specs}{
    \zeta\in \spec{X} \;\; \Longleftrightarrow\;\; 0\in \spec{\bs{H}_{\zeta}}.
}
The main part of our technical work is to make the deterministic counterparts
of the relations~\eqref{resG} and~\eqref{specs}
 rigorous, with effective controls. In particular, the EDE in~\eqref{definition of b}
used to define the set $\cal{R}$ turns out to be one of the components of the MDE
for $\bs{H}_{\zeta}$ in the limit $\eta=\Im z\to 0+$, and the solution $ \frak{b}(\zeta)$ 
of EDE is the boundary value of the diagonal part of the 
(2,1)-component of the  solution matrix $\bs{M}_{\zeta}$.
While intuitively  all these claims are natural, their rigorous proofs are delicate
since they involve interchanging the large $N$ and the $\eta\to 0+$ limits. 
The main reason why this is possible is that outside of the spectrum the corresponding MDE
is stable against small perturbations. Technically, a good lower bound on
 $\min\{\abs{\zeta}, \Delta_\zeta\}$ guarantees stability.

In addition to single resolvents $(X-\zeta)^{-1}$ needed for~\eqref{first}, the product of two resolvents of the form  $(X-\zeta_1)^{-1}(X^*-\ol{\zeta}_2)^{-1}$ at two different spectral parameters $\zeta_1, \zeta_2\in \cal{R}$ is needed to compute $f(X) g(X^*)$  via~\eqref{f(X) in terms of resolvent}. %we need to understand products of two resolvents of the form  $(X-\zeta_1)^{-1}(X^*-\ol{\zeta}_2)^{-1}$ at two different spectral parameters $\zeta_1, \zeta_2\in \cal{R}$.
This product may be viewed as a specific
rational function of $X$ hence it has its own hermitized linearization (see e.g. Appendix
of \cite{bimodal}). However, instead of using the general theory of linearizations, 
in this paper we follow a more direct route. Introducing an additional parameter $\alpha\in \R$
we define a $4N\times 4N$ matrix and its resolvent as
\[
\bs{H}^{(\zeta_1, \zeta_2)}_\alpha := 
\mfour
{0 & 0 & 0 & X^* - \ol{\zeta}_2 }
{ 0 & 0 & X -\zeta_1  & \alpha }
{0 & X^*- \ol{\zeta}_1  & 0 & 0}
{ X - {\zeta}_2& \alpha & 0 & 0}, \qquad  \bs{G}^{ (\zeta_1, \zeta_2)}_\alpha(z) := 
 (\bs{H}^{(\zeta_1, \zeta_2)}_\alpha  - z )^{-1}
 \]
 and use the algebraic relation
 \bels{alg}{
    \frac{d}{d\alpha}\Big|_{\alpha=0} [\bs{G}^{ (\zeta_1, \zeta_2)}_\alpha(0)]_{31}
    = (X-\zeta_1)^{-1}(X^*-\ol{\zeta}_2)^{-1}
 }
 whenever both sides exist,
 where the index indicates the corresponding (3,1)-block. We thus need to understand the 
 deterministic approximations of both sides of the key relation~\eqref{alg} as $N\to\infty$.
 Roughly speaking we need to differentiate, with respect to $\alpha$, the local 
 law for $\bs{G}^{ (\zeta_1, \zeta_2)}_\alpha(z)$ at $\eta=\Im z=0+$. 
 Since we work outside of the spectrum, the MDE for $\bs{H}^{(\zeta_1, \zeta_2)}_\alpha$
 is stable and thus interchanging the $\alpha$-derivative,  the large $N$ limit
 and the $\eta=\Im z\to 0+$  limit can be justified after a careful analysis
 that takes up most of the paper.
 
 We now start the actual proof, by first  giving a brief overview on the MDE theory
in Section~\ref{sec:Hermitian random matrices and the matrix Dyson equation}.

 \bigskip
 
 {\it Conventions.} 
The $N$-independent constants ($\varphi_p, \varrho,c_0,L, c_1,C_1 $) 
 in our assumptions  and the footnote associated to Assumption (2.F)  are called {\it model parameters}. 
In the rest of the paper, for two $N$-dependent positive quantities $a=a_N$ and $b=b_N$ we write $a\lesssim b$
  if there is an $N$-independent constant $C$ such that $a\le Cb$.  The constant $C$ may depend on the model parameters of the corresponding assumptions.
  We write $a\sim b$ if $a\lesssim b$ and $b\lesssim a$ both hold. We will also
  apply this convention entry-wise when $a,b$ are vectors and in the sense of quadratic forms when  $a, b$ are positive definite  Hermitian matrices.  Furthermore, we write $a \approx b$ for $a = b(1+o(1))$ in a locally specified limit.

\bigskip

\section{Hermitian random matrices and the matrix Dyson equation}
\label{sec:Hermitian random matrices and the matrix Dyson equation}

In this section we provide a brief overview of how resolvents $G(z):=(H-z)^{-1}$ of Hermitian random matrices $H=H^*$ are analyzed when their dimension tends to infinity.  Although  most of the discussion would also apply 
 to the much more general setup of decaying correlations from \cite{AjankiCorrelated,EKS}, for concreteness we will only consider $H \in \cal{A}:=\C^{K \times K} \otimes \C^{N \times N}$ of the form
\bels{form of H}{
H = h_0 \otimes I + q \otimes X + q^* \otimes X^*\,,
}
where $q,q^* \in \C^{K \times K}$, $h_0=h_0^* \in \C^{K \times K}$ is self-adjoint and $X$ is a random matrix that  belongs  to the elliptic-type ensemble. Note that the matrices $\bs{H}_{\zeta}$ and $\bs{H}^{(\zeta_1, \zeta_2)}_\alpha$
 from the previous section are of this form
with $K=2$ and $K=4$, respectively. In general, we consider $K$ fixed, then in 
 the limit as $N \to \infty$, the resolvent $G(z)$ is well approximated by a deterministic matrix $M(z) \in \cal{A}$ that satisfies the {\it matrix Dyson equation (MDE)}.  This equation is written in  the form 
\begin{equation} \label{eq:genMDE} M(z)=\Phi_z(M(z)) \,, \qquad \Phi_z(R):= (\E H-z I -  \cal{S}[R])^{-1}\,, \end{equation}  
where $z \in \C$ is the spectral parameter with positive imaginary part, $\im z>0$, and the linear {\it self-energy operator}
 $\cal{S}: \cal{A} \to \cal{A}$ is determined by the covariances among the entries of $H$ through
\bels{general definition of cal S}{
\cal{S}[R]:= \E \, (H-\E H) R(H-\E H)\,.
}
The self-energy operator is self-adjoint with respect to the  natural Hilbert-Schmidt  
scalar product on $\cal{A}$, i.e. $\tr_{KN} R_1^* \cal{S}[R_2] =\tr_{KN} \cal{S}[R_1]^*R_2$, and it is positivity preserving, i.e. it leaves the cone of positive semidefinite matrices invariant. 

Equation \eqref{eq:genMDE} has a unique solution with positive definite imaginary part, i.e. with $M \in \cal{A}_+:=\{R \in \cal{A}: (R-R^*)/2\ii >0\}$.
 Furthermore, \eqref{eq:genMDE} can be viewed as a
  fixed point equation for $M$, where the function $\Phi_z$  on the right hand side  is a contraction in the Carath\'eodory metric on $\cal{A}_+$. 
  Thus it can effectively  be solved by 
  iteration starting from any matrix with positive imaginary part. For details we refer to \cite{Helton01012007}. In particular, $M(z)$ lies inside any closed subset $\cal{B}\subset \cal{A}_+$ that is left invariant by $\Phi_z$.
  
Associated to the solution $M \in \cal{A}_+$ of \eqref{eq:genMDE} is the {\it self-consistent density of states} $\rho$, the probability measure on the real line whose Stieltjes transform is $\tr_{KN} M(z)$. Thus, $\rho$ is uniquely determined by the identity
\bels{scDOS}{
\tr_{KN} M(z) = \int_\R \frac{\rho(\dd \tau)}{\tau -z} \,, \qquad \Im z >0\,.
}
 The support of $\rho$ is called the {\it self-consistent spectrum} of $H$.
 For existence and uniqueness of this correspondence we refer the reader e.g. to \cite{AjankiCorrelated}, Proposition 2.1.

  The resolvent of the random matrix $H$ satisfies a perturbed version of \eqref{eq:genMDE}, namely 
 \[
 1+(z-\E H + \cal{S}[G(z)]) G(z)= D\,,
 \]
where $D=D(z) \in \cal{A}$ is a random error matrix. For $D=0$ we recover the MDE \eqref{eq:genMDE}. Matrices $H$ of the form \eqref{form of H} fall into the general class of Hermitian random matrices with decaying correlations among their entries considered in \cite{EKS}. In particular, $H$ satisfies \cite[Assumptions~(A), (B), (C) and (D)]{EKS} and thus \cite[Theorem~2.1]{EKS} is applicable. Meaning that for $z$ separated from $\rho$, we have that $G$ is well approximated by $M$ in following the sense: 

\begin{lemma}[Optimal local law away from the self-consistent spectrum] \label{lem:oploclaw} There exists a $c>0$ such that for any $z$ satisfying $N^{-c}~\le~\dist(z, \supp \rho) \le N^{100}$, we have that 
\[ \P \big[ \absb{\scalar{x}{(G(z)-M(z))y}}\le \norm{x}_2\norm{y}_2N^{-1/2+\eps}\big] \ge 1-C_{\eps,\nu}N^{-\nu}\,,  \]
for  any $\eps>0$, deterministic vectors $x,y$ and $\nu \in \N$.
\end{lemma}

  For $H$ of the form \eqref{form of H} and $X$ of elliptic-type, we have $\E H = h_0 \otimes I \in \cal{A}^d$, where $ \cal{A}^d\subset \cal{A}$ is the subalgebra  spanned by block diagonal matrices of the form $a\otimes \frak{D}_r$ with $a \in \C^{K \times K}$ and $r \in \C^N$. In particular, we can identify  $\cal{A}^d= \C^{K \times K} \otimes \C^N$, where on $\C^N$ the multiplication is entrywise. Moreover, $\cal{S}$ preserves  $\cal{A}^d$ and, thus, \eqref{eq:genMDE} can be interpreted as an equation on $\cal{A}^d$ instead of $\cal{A}$ and we have $M \in \cal{A}^d$ as well. 

In order to see that the solution $M$ to \eqref{eq:genMDE} will depend analytically on the data $\E H$ and $\cal{S}$ we take the  derivative  of the function  $J(R):= R-\Phi_z(R)$ for $R \in \cal{A}^d$
and find 

\[
\nabla J(R) = \cal{L}_{\Phi_z(R)}  \,, 
\]
 where $\cal{L}_R:\cal{A}^d \to \cal{A}^d$ is a linear map defined through
\bels{gen stability operator}{
\cal{L}_R[Z]:= Z-R(\cal{S}[Z])R\,.
}
We will refer to $\cal{L}=\cal{L}_M=\nabla J(M)$  as the {\it stability operator} associated to the MDE. Since $M$ is invertible by definition, its analytic dependence on the data is ensured by the implicit function theorem as long as $\cal{L}$ is invertible. Note that $\cal{L}$ is  restricted to $\cal{A}^d$ since $J$ preserves the space of block diagonal matrices.  
Hence the main technical information for analyzing the stability of the MDE is the invertibility of its stability operator. This is often a hard problem since  $\cal{L}$ depends on the non-explicit solution to the MDE.

\section{Self-consistent pseudospectrum and resolvent} 
\label{sec:Self-consistent pseudospectrum}

We continue considering elliptic-type matrices with Assumptions (A), (B), and \mbox{(2.C-F)}. 
In this section, we 
describe the asymptotic behavior of the resolvent $ (X-\zeta)^{-1}$ of $X$ in the limit $N \to \infty$ and identify the set of $\zeta$'s for which such description is possible
 with very high probability,  namely the self-consistent pseudo-resolvent set, $\cal{R}$. 
We have three definitions which we will prove are equivalent.  
The first one was given in Definition~\ref{def:Rb}; using that the EDE is the (2,1)-component of the MDE for 
the Hermitization $\bs{H}_{\zeta}$
defined in~\eqref{2x2 symmetrization}, this is the deterministic version of
the identity~\eqref{resG}. This definition is intuitive and easy to understand but in its original form is useless,
  since the EDE is hard to control. The second definition (Definition~\ref{def:self-consistent pseudospectrum} below)
is more complicated as it relies on certain bounds on the entire 
solution $\bs{M}_{\zeta}$ of the MDE for $\bs{H}_{\zeta}$  but will be heavily used in the proofs. Finally, 
we will also give a third equivalent definition in Theorem~\ref{thm:eq}
in terms of the support of the  self-consistent density of states $\rho_\zeta$ for $\bs{H}_{\zeta}$. 

From the third  definition we will see that the self-consistent pseudo-resolvent set provides an asymptotic description 
of the  complement of the $\eps$-pseudospectrum  of $X$ in the regime of consecutive limits $N \to \infty$  first, 
followed by  $\eps \to 0$. 
Here we recall that the  $\eps$-pseudospectrum  of $X$  is defined by the resolvent through
\[{\textstyle \spec_\eps}(X):=\{\zeta \in\C  : \norm{(X-\zeta)^{-1}} >\eps^{-1}\}\,.\]

In Section~\ref{Subsec:Self consistent pseudospectrum via MDE}, we begin by describing the MDE and its solution. We then use the MDE to prove Proposition~\ref{prp:Solution of vector equation}, which provides important properties of the solution to the EDE. Then in Section~\ref{sec:Equivalence of the definitions of R},
we give alternative, equivalent, characterizations of the self-consistent 
pseudo-resolvent set, as well as, certain quantitative version of these equivalences. The proofs of the equivalence of these characterizations are given in Section \ref{sec:Properties of R} after additional properties of the MDE are proven.

\subsection{Self-consistent pseudospectrum via MDE}
\label{Subsec:Self consistent pseudospectrum via MDE}

To study the inverse of $X-\zeta$, we recall  its Hermitization  $\bs{H}_{\zeta}$  
and the corresponding resolvent $\bs{G}_\zeta(z) := (\bs{H}_{\zeta}-z)^{-1}$ 
from~\eqref{2x2 symmetrization}. 
The covariances of the entries of $X$ are encoded in the following operators acting on 
matrices $R\in\C^{N\times N}$:  
\bels{def S and T operators}{ \scr{S}[R] := \E[X   R X^*], \qquad    \scr{T}[R] := \E[X   R X] .  }
Using the standard inner product, $\scalar{A}{B}:= \tr_N A^* B$,  
the adjoints of these operators are
\bels{def S* and T* operators}{  \scr{S^*}[R] = \E[X^*   R X]  , \qquad   \scr{T^*}[R] = \E[X^*   R X^*] .  }
These operators are the key input data for the MDE associated to $\bs{H}_\zeta$, see \eqref{eq:genMDE}, that determines the limiting behavior of the resolvent of $X$. In fact, with 
\bels{A and cal S in 2x2 case}{
\bs{A}_\zeta:=\mtwo{0 & \zeta}{\ol{\zeta} &0}\,, \qquad \cal{S}\mtwo{R_{11} & R_{12}}{R_{21} & R_{22}}:=
\mtwo{\scr{S}[R_{22}] & \scr{T}[R_{21}]}{\scr{T}^* [R_{12}] & \scr{S}^* [R_{11}]}\,,
}
for $R_{ij} \in \C^{N \times N}$, the MDE for $\bs{M}_\zeta = \bs{M}_\zeta(z)$ takes the form
\bels{2x2 MDE}{
\bs{M}_\zeta = \Phi^{2 \times 2}_{z,\zeta}(\bs{M}_\zeta):= -(z + \bs{A}_\zeta + \cal{S}[\bs{M}_\zeta])^{-1}\,,
}
where mostly  we will consider spectral parameters $z=\ii \eta$  with $\eta >0$ on the imaginary axis. The definition of $\cal{S}$ in \eqref{A and cal S in 2x2 case} is consistent with the general definition from \eqref{general definition of cal S}. Note that the operators $\scr{S}$ and $\scr{T}$ from \eqref{def S and T operators} as well as their adjoints from \eqref{def S* and T* operators} leave the space of diagonal matrices invariant, i.e. 
\[
 \scr{S}[\frak{D}_r]= \frak{D}_{Sr}\,,\qquad  \scr{S}^*[\frak{D}_r]= \frak{D}_{S^*r}\,, \qquad \scr{T}[\frak{D}_r]= \frak{D}_{Tr}\,,\qquad \scr{T}^*[\frak{D}_r]= \frak{D}_{T^*r}\,, 
 \]
 for any $r \in \C^N$, where $S$ and $T$ are from \eqref{definition of S and T}.

The subspace of $2 \times 2$-block diagonal matrices with purely imaginary blocks on the diagonal is invariant under the operation 
 $\Phi^{2 \times 2}_{z,\zeta}$ 
 for $z=\ii \eta$. Thus, the unique solution to the MDE with $\Im \bs{M}_\zeta>0$  takes the form
\bels{form of 2x2 M}{
\bs{M}_\zeta(\ii \eta)= \mtwo{\ii \frak{D}_{a_\zeta} & \frak{D}_{b_\zeta}^*}{\frak{D}_{b_\zeta} &\ii \frak{D}_{d_\zeta}} = \mtwo{\ii\1a_\zeta & \ol{b_\zeta}}{b_\zeta &\ii \1d_\zeta }\,,
}
for  some  complex vector  $b_\zeta= b_\zeta(\ii \eta)\in \C^N$ and some positive vectors $a_\zeta= a_\zeta(\ii \eta),d_\zeta= d_\zeta(\ii \eta) \in \R^N$. 
In \eqref{form of 2x2 M} we slightly abused the notation by identifying the diagonal matrix $\frak{D}_u$ with the vector $u \in \C^N$. The self-consistent density  of states (cf. \eqref{scDOS}) associated to $\bs{H}_\zeta$ is denoted by $\rho_\zeta$.

\begin{Proof}
[Proof of Proposition~\ref{prp:Solution of vector equation}]
 The proof of Property~\ref{prop:Uniqueness} is carried out last. 
To prove Properties~\ref{prop:Openness} and \ref{prop:Holomorphicity} we will show  below 
that for any $\zeta \in \cal{R}$  the defining equation  $J_\zeta(\frak{b})=0$ for $\frak{b}=\frak{b}(\zeta)$ with
\bels{Jzeta}{
J_\zeta(b):= 1+ (\zeta +Tb)b
}
is stable at $\frak{b}$ in the sense that its derivative $\nabla J_\zeta(\frak{b})$ is invertible and holomorphic in $\zeta$. By the implicit function theorem this implies that $J_{\zeta +\omega}(b)=0$ has a solution $b=b(\zeta+\omega)$ for any sufficiently small $\abs{\omega}$ and that $\omega \mapsto b(\zeta+\omega)$ is holomorphic. In particular, Property~\ref{prop:Holomorphicity} holds. Furthermore, $\cal{R}$ is open because $J_{\zeta+\omega}(b)$ has a solution and the side condition \eqref{sidecond} remains true for this solution $b$ and small enough $|\omega|$ because of the continuity of $\omega \mapsto b=b(\zeta+\omega)$.

Indeed, the derivative of $J_\zeta(b)$ with respect to $b$ is
\[
\nabla J_\zeta(b) = \frak{D}_{\zeta +Tb}+ \frak{D}_{b}T\,.
\]
Evaluated at the solution $\frak{b}=\frak{b}_\zeta$ we get 
\bels{nabla J}{
\nabla J_\zeta(\frak{b}) = -\frak{D}_{\frak{b}}^{-1}(1-\frak{D}_{\frak{b}}^2T)\,.
}
We show now that this matrix is invertible and, thus, that $J_\zeta(b)=0$ can be solved locally around $(\zeta, \frak{b}_\zeta)$. For the invertibility of $\nabla J_\zeta(\frak{b})$ let $|T|:=(\abs{t_{ij}})_{i,j}\in \R^{N \times N}$. Then using Assumption \eqref{assum:ellcorr} we have
\bels{bound T by S}{  \frak{D}_{|\mathfrak{b}|^2}|T| \preceq |\varrho| \frak{D}_{|\mathfrak{b}|^2} (S^{(1/2)} \odot (S^*)^{(1/2)} ) }
where $S^{(1/2)} $ is defined to be the matrix whose $(i,j)$ entry is $s^{1/2}_{ij}$, $\odot$ is the Hadamard product, and $\preceq$ is the entry-wise inequality on matrices, i.e. $A \preceq B$ if for all $i,j$, $ a_{ij}\le b_{ij}$.
Then using that for any matrices $A,B$ with non-negative entries, \begin{equation} \label{eq:hadspecrad} \frak{r}(A^{(1/2)}\odot B^{(1/2)}) \leq \frak{r}(A)^{1/2}\frak{r}( B)^{1/2} \end{equation} (see, for instance \cite{altenberg2013sharpened}, Theorem 13) we have 
\bels{bound SSstar by S}{ \frak{r}\Big( \frak{D}_{|\mathfrak{b}|^2} (S^{(1/2)} \odot (S^*)^{(1/2)} )  \Big)\leq  \frak{r}( \frak{D}_{|\mathfrak{b}|^2} S ),}
where we used that $\frak{D}_{|\mathfrak {b}|^2}S$ and $\frak{D}_{|\mathfrak {b}|^2}S^*$ have the same spectral radius.  Then using the side condition, \eqref{sidecond}, and that for any two non-negative matrices $A,B$ such that $A  \preceq B $, we have

\begin{equation} \label{eq:posradbound}
\mathfrak{r}(A) \leq  \mathfrak{r}(B),
\end{equation} we conclude
\begin{equation} \label{eq:gennonherm} \frak{r}(\frak{D}_{\mathfrak {b}^2} {{T}})\le  \frak{r}(\frak{D}_{|\mathfrak {b}|^2} {\abs{T}}) \le|\varrho| <1 \end{equation} for  $\zeta\in {\cal{R}}$. This proves invertibility of the derivative \eqref{nabla J}. 

To see that $0 \not \in \cal{R}$ first observe that $b=0$ never solves the EDE. 
If $b$ were a nonzero solution to $J_\zeta(b) =0$ at $\zeta =0$, then this equation would be  equivalent to $\frak{D}_{b^2}Tb=-b$, 
$b\ne0$, which contradicts \eqref{eq:gennonherm}. Thus, $J_0(b) =0$ does not have a solution $b$. 

To finish the proof of Property~\ref{prop:Openness} it remains to show $\{\zeta : \abs{\zeta}> C\} \subset \cal{R}$ for some $C>0$. For this purpose consider the equation
\[
0=J_{-\frac{1}{\omega}}(\omega\1 q)= 1-q+\omega^2 \1 qTq=:\wh{J}_\omega(q)
\]
for a complex number $\omega$ and $q \in \C^N$. Since $\nabla\wh{J}_0(q) =-1$ is invertible, the implicit function theorem
 implies that the equation $\wh{J}_\omega(q)=0$ 
 has a locally unique solution $q(\omega)$ in a neighborhood of $(\omega,q)=(0,1)$.  With the translation $\zeta = -\frac{1}{\omega}$ and $b(\zeta) = -\frac{1}{\zeta} \1q\pb{-\frac{1}{\zeta}}$ we have a solution $b=b(\zeta)$ to $J_\zeta (b)=0$ for large enough $\abs{\zeta}$. This solution clearly satisfies the side condition \eqref{sidecond} as $\abs{\zeta} \gg 1$. It also satisfies Property~\ref{prop:Asymptotics}. 

It remains to verify the uniqueness from Property~\ref{prop:Uniqueness}. 
For  $\zeta\in \cal{R}$ let  $\frak{b}=\frak{b}(\zeta)$ be a solution to 
 \eqref{definition of b} and \eqref{sidecond}.
We consider the MDE from \eqref{2x2 MDE}. As was explained in \eqref{form of 2x2 M}, the solution to this MDE is block diagonal, i.e. we can interpret $\bs{M}_\zeta(z) \in \C^{2 \times 2} \otimes \C^N \subset \C^{2N \times 2N}$. Thus, \eqref{2x2 MDE} is equivalent to $\cal{J}_{z,\zeta}(\bs{M})=0$ where $\cal{J}_{z,\zeta}: \cal{D} \to \C^{2 \times 2} \otimes \C^N$ 
is the restriction of $\bs{M} \mapsto \bs{M} - \Phi^{2 \times 2}_{z,\zeta}(\bs{M})$ to $\cal{D}=\cal{D}_{z,\zeta}$, the set of all  $\bs{M} \in \C^{2 \times 2} \otimes \C^N$ such that $\Phi^{2 \times 2}_{z,\zeta}(\bs{M})$ exists. In particular, 
 the directional derivative $\nabla_{\bs{R}}\cal{J}_{z,\zeta}(\bs{M})$ of $\cal{J}_{z,\zeta}$
  in the direction $\bs{R}\in  \C^{2 \times 2} \otimes \C^N$ is given by 
 \[
\nabla_{\bs{R}}\cal{J}_{z,\zeta}(\bs{M})= \bs{R} - \Phi^{2 \times 2}_{z,\zeta}(\bs{M})\cal{S}[\bs{R}]\Phi^{2 \times 2}_{z,\zeta}(\bs{M})\,,
\]
where 
\[
\cal{S}[\bs{R}]=\cal{S}\mtwo{r_{11} & r_{12}}{r_{21} & r_{22}} = \mtwo{Sr_{22} & Tr_{21}}{T^*r_{12} & S^*r_{11}}\,.
\]
Setting 
\bels{bs M zeta 0}{
\bs{M}_\zeta^{(0)}:= \mtwo{0 & \frak{b}^*}{\frak{b} &0}, \qquad  \frak{b}  = \frak{b}(\zeta),
}
 from the EDE~\eqref{definition of b} 
we see that $\cal{J}_{0,\zeta}(\bs{M}_\zeta^{(0)})=0$. Therefore the directional derivative of $\cal{J}_{0, \zeta}$ at
$\bs{M}_\zeta^{(0)}$ is given by
\bels{Def of L zeta}{
\cal{L}_\zeta\bs{R}:=\nabla_{\bs{R}}\cal{J}_{0, \zeta}(\bs{M}_\zeta^{(0)})=\bs{R}- \bs{M}_\zeta^{(0)}\cal{S}[\bs{R}]\bs{M}_\zeta^{(0)}
=\mtwo{L_{11}r_{11} & L_{12} r_{12}}{L_{21}r_{21} & L_{22} r_{22}}
\,
}
with the $N \times N$ matrices $L_{ij}$ defined as
 \bels{definition of lij matrices}{
 &L_{11} = 1 - \frak{D}_{|{\frak{b}}|^2}S^*\,,
\qquad 
L_{12} = 1 - \frak{D}_{{{\frak{b}}}^2}^*T^*\,,
\qquad
L_{21} = 1 - \frak{D}_{{{\frak{b}}}^2}T\,,
\qquad 
L_{22} = 1 - \frak{D}_{|{\frak{b}}|^2}S\,.
 }
The side condition \eqref{sidecond} and $\frak{r}( \frak{D}_{{{\frak{b}}}^2}T)\le \frak{r}(\frak{D}_{\abs{{\frak{b}}}^2} S)$
(from \eqref{bound T by S}--\eqref{eq:gennonherm})  together ensure the
  invertibility of all matrices $L_{ij}$, hence the invertibility of  $\cal{L}_\zeta$  by \eqref{Def of L zeta} because $\cal{L}_\zeta$ leaves each block invariant. Thus, $\cal{J}_{z,\zeta}(\bs{M})=0$ has a unique local solution $\bs{M}=\wt{\bs{M}}_\zeta(z)$ in a neighborhood of $(z,\bs{M})=(0,\bs{M}_\zeta^{(0)})$. 
   In particular, $\wt{\bs{M}}_\zeta(0) = \bs{M}_\zeta^{(0)}=\wt{\bs{M}}_\zeta(0)^*$ is self-adjoint.
  
   Taking the  inverse on both sides \eqref{2x2 MDE}  and then the derivative  
  with respect to $\eta = \Im z$ yields 
  \[
  \frac{1}{\bs{M}_\zeta}(\partial_{\eta} \bs{M}_\zeta )  \frac{1}{\bs{M}_\zeta^*} = \ii + \cal{S}[\partial_{\eta} \bs{M}_\zeta ]\,.
  \]
From this we conclude 
$\cal{L}_\zeta[\partial_\eta \wt{\bs{M}}_\zeta(0)]=  \ii\wt{\bs{M}}_\zeta(0)^2$ because $\wt{\bs{M}}_\zeta(0)$ is self-adjoint. 
 Note that each component of $\frak{b}$ is nonzero by the EDE, hence $\wt{\bs{M}}_{\zeta}(0)^2$ is strictly positive definite. 
 Moreover,  since $\cal{L}_{\zeta}^{-1}$ is positivity preserving we obtain that $\Im \wt{\bs{M}}_\zeta(\ii \eta)$ is 
strictly  positive definite for  all  sufficiently small $\eta>0$.  In particular, the local solution 
$\wt{\bs{M}}_\zeta(z)$ on the imaginary axis 
 coincides with the usual MDE solution discussed in this section, i.e. $\wt{\bs{M}}_\zeta(\ii \eta) = {\bs{M}}_\zeta(\ii \eta)$. Since ${\bs{M}}_\zeta(\ii \eta)$ is uniquely defined as the 
solution to~\eqref{2x2 MDE} with positive definite imaginary part and $\frak{b} =( \bs{M}_\zeta^{(0)})_{21}= ( \lim_{\eta \downarrow 0} {\bs{M}}_\zeta(\ii \eta))_{21}$ by construction, we conclude that $\frak{b}$ is unique.  
\end{Proof}

The following corollary collects some important insights from the proof of Proposition ~\ref{prp:Solution of vector equation}
that will be used later.

\begin{corollary}\label{crl:proof of prop 2.4}
For any $\zeta \in \cal{R}$ the Hermitian matrix $\bs{M}^{(0)}_\zeta \in \C^{2 \times 2} \otimes \C^N\subseteq \C^{2N \times 2N}$ from \eqref{bs M zeta 0} is a solution of the MDE \eqref{2x2 MDE} at $z=0$. It is obtained as a limit of solutions on the upper half plane with positive definite imaginary parts, i.e. $\bs{M}^{(0)}_\zeta = \lim_{\eta \downarrow 0}\bs{M}_\zeta(\ii \eta)$ where  $\bs{M}_\zeta(\ii \eta)$ solves \eqref{2x2 MDE} at $z=\ii \1\eta$ with  the  side condition that $\im \bs{M}_\zeta(\ii \eta)$ is positive definite. 
Furthermore, the associated stability operator $\cal{L}_\zeta $ defined  in \eqref{Def of L zeta} is invertible.
\end{corollary}

\subsection{Equivalent definitions of  $\cal{R}$}
\label{sec:Equivalence of the definitions of R}

We now introduce an equivalent definition of $\cal{R}$ which involves the entire solution to the MDE, \eqref{eq:MDE}.  

\begin{definition}[self-consistent pseudospectrum via MDE] \label{def:self-consistent pseudospectrum} Let $\bs{M}_\zeta=\bs{M}_\zeta(\ii \eta)$ be the solution to the MDE \eqref{2x2 MDE} of the form \eqref{form of 2x2 M} associated to the Hermitization $\bs{H}_\zeta$ of $X-\zeta$ from \eqref{2x2 symmetrization}. Then we define
\bels{hatRM}{
 \wh{\cal{R}} :&=\Big\{ \zeta \in \C  : \limsup_{\eta \2\downarrow \20}  \frac{1}{\eta} \|\Im 
  \bs{M}_\zeta(\ii \eta)\|
< \infty \Big\}
\\
&=\Big\{ \zeta \in \C  : \limsup_{\eta \2\downarrow \20}  
\frac{1}{\eta} \max\{ \|a_{\zeta}(\ii \eta)\|_\infty ,\|d_{\zeta}(\ii \eta)\|_\infty    \} < \infty \Big\} ,
 }
  where in the second equality we used that $\Im 
  \bs{M}_\zeta(\ii \eta)$ is diagonal. 
\end{definition}

The following theorem shows the equivalence of the two definitions of the self-consistent
pseudo-resolvent set and it also presents a third alternative definition:

\begin{theorem}[Equivalence of the definitions]\label{thm:eq}
We assume (A), (B), and \mbox{(2.C-F)}. Then the set $\cal{R}$ from Definition~\ref{def:Rb}
coincides with $\wh{\cal{R}}$ from Definition~\ref{def:self-consistent pseudospectrum}, i.e. 
$\cal{R}=\wh{\cal{R}}$.
Furthermore, for the self-consistent pseudo-resolvent we have
\bels{ABDetaat0} {  \frak{b}(\zeta)= \lim_{\eta \downarrow 0 } b_\zeta(\ii \eta), \qquad \mbox{for any $\zeta\in \cal{R}$,} }
where the left hand side is the unique solution to \eqref{definition of b} and \eqref{sidecond} and $b_\zeta$ on the left hand side is the off-diagonal contribution to the solution of MDE from \eqref{form of 2x2 M}. 
Finally, as a third alternative characterization of $\cal{R}$ we also have
\bels{third}{
\cal{R}= \{\zeta\in \C\; : \; 0\not\in \rm{supp} ( \rho_\zeta)\}.
}
\end{theorem}

All three characterizations of the self-consistent pseudo-resolvent set  are qualitative, they
 lack an effective control. To remedy this situation, 
 we define the following quantitative versions of~\eqref{hatRM}, 
 \eqref{sidecond} in Definition~\ref{def:Rb} and \eqref{third}:
 \bels{pseudorestau}{
\wh{\cal{R}}_\tau:=\Big\{ \zeta \in \C  : \limsup_{\eta \2\downarrow \20}  
\frac{1}{\eta} \max\{ \|a_{\zeta}(\ii \eta)\|_\infty ,\|d_{\zeta}(\ii \eta)\|_\infty    \} < \frac{1}{\tau} \Big\} 
}
 \bels{rtilde}{
 \wt{\cal{R}}_\delta:= \{ \zeta\in \cal{R} \; :\; \min\{\abs{\zeta}, \Delta_\zeta\} \ge\delta\},
 \qquad \cal{R}^\#_\omega:= \{ \zeta\in \cal{R} \; :\; \rho_\zeta([-\omega, \omega])=0\}
 }
 for any positive control parameters  $\tau, \delta,\omega$  (recall that $\Delta_\zeta$ was defined in~\eqref{def of Delta zeta}).
 Hence  Theorem~\ref{thm:eq} asserts that
 \[
    \cal{R}=  \bigcup_{\delta>0} \wt{\cal{R}}_\delta = \bigcup_{\omega>0}\cal{R}^\#_\omega
    =\bigcup_{\tau>0} \wh{\cal{R}}_\tau,
 \]
 where we also used that $0\not\in\cal{R}$.
 The following proposition, to be proven in Section~\ref{sec:Properties of R},
  is a more precise and effective version of Theorem~\ref{thm:eq}
 that compares the sets $\wt{\cal{R}}_\delta, \wh{\cal{R}}_\tau$ and $\cal{R}^\#_\omega$.

 \begin{proposition}\label{prp:eff} Under the conditions (A), (B), and \mbox{(2.C-F)}
 there exists a  (small) constant $c>0$ depending only on the model parameters
 and a  positive integer power $p\in \N$  depending only on $|\rho|$ and $L$ in~\eqref{eq:prim} 
 such that the following
 containments hold:
 \bels{RR}{
  \wh{\cal{R}}_\tau \subset  \wt{\cal{R}}_{c\tau} \cup\{|\zeta|\le c\tau\}\qquad 
     \wt{\cal{R}}_{\delta} \subset {\cal{R}}^\#_{c\delta^p},
         \qquad \cal{R}^\#_{\omega} \subset 
     \wh{\cal{R}}_{c\omega^p}, 
}
for any small $\tau, \delta, \omega>0$.
 \end{proposition}

Proposition~\ref{prp:eff} shows that  lower bounds
 on  $\Delta_\zeta$, on $\mbox{dist} (\mbox{supp}\rho_\zeta, 0)$
 and  on  $\eta/\|\Im \bs{M}_\zeta(\ii\eta)\|$
are polynomially comparable with each other. Moreover, all 
these three quantities depend smoothly on $\zeta\in \cal{R}$
but we cannot  effectively control their derivatives near the boundary of $\cal{R}$.
Therefore,  while the positivity of 
any (hence all) of these quantities characterize the set $\cal{R}$,  unfortunately
we do not have an effective bound that would relate
the size of these comparable quantities with the distance of $\zeta$ 
to the boundary of $\cal{R}$. We will be able to do this
near the rightmost point of $\cal{R}^c$
in the
case when $t_{ij}\ge 0$ (Section~\ref{sec:ellcorr}).

\subsection{Properties of $\bf{M}$ and $\mathfrak{b}$}
\label{sec:Properties of R}

In this section we demonstrate the usefulness of the alternative quantitative definition of the self-consistent pseudo-resolvent set from \eqref{pseudorestau}. We use it to prove several properties of $\bs{M}_\zeta$ and its off-diagonal entry $b_\zeta$ from \eqref{form of 2x2 M}.
Finally we prove Proposition~\ref{prp:eff} and thus Theorem~\ref{thm:eq}, translating these properties of $b_\zeta$ to the solution $\frak{b}$ of EDE via \eqref{ABDetaat0}.

\begin{lemma} \label{lemma:boundedM}
Fix  $\tau \in (0,1]$.  For any $\zeta\in \wh{\cal{R}}_\tau$  the limit $\bs{M}_\zeta(0): = \lim_{\eta\downarrow 0} \bs{M}_\zeta(\ii\eta)$ exists and satisfies the upper bound
$\norm{\bs{M}_\zeta(0)} \lesssim \frac{1}{\tau}$. The diagonal elements $a_\zeta$ and $d_\zeta$ from \eqref{form of 2x2 M}  vanish in this limit, i.e. $a_\zeta(0)=d_\zeta(0)=0$ and the off-diagonal element $b_\zeta(0)=\lim_{\eta \to 0+} {b}_{\zeta}(\ii \eta )$ is a solution to the 
extraspectral Dyson equation \eqref{definition of b}  and \eqref{sidecond}.  
In particular, \eqref{ABDetaat0} holds for $\zeta\in \wh{\cal{R}}_\tau$, hence $\wh{\cal{R}}=\cup_{\tau>0} \wh{\cal{R}}_\tau
\subset \cal{R}$. 
 Moreover,  $\wh{\cal{R}}_{\tau=1}$ contains a neighborhood of infinity, i.e. 
\bels{Eps}{
\{\zeta: \abs{\zeta} > C\} \subset \wh{\cal{R}}_1
}
for some large constant $C>0$ and on this neighborhood $\bs{M}_\zeta$ satisfies
\bels{largezeta}{
\norm{\bs{M}_\zeta(\ii \eta )} \lesssim \frac{1}{\eta +\abs{\zeta}}\,, \qquad \abs{\zeta}>C\,.
}
\end{lemma}
Before proving Lemma~\ref{lemma:boundedM} we record equations for the functions $a_\zeta=a_\zeta(\ii\eta),b_\zeta=b_\zeta(\ii\eta)$ and $d_\zeta=d_\zeta(\ii\eta)$ from the representation \eqref{form of 2x2 M} of $\bf{M}_\zeta$ that correspond to each of the blocks in \eqref{2x2 MDE}.

After multiplying \eqref{2x2 MDE} by the inverse of the right hand side and  rearranging the $(2,2)$ entry we find (dropping the subscript $\zeta$ and argument $\ii\eta$ for brevity)
\begin{equation}
\label{Bequation} 1+(Tb + \zeta)b =     a(S^*d+ \eta ) .
\end{equation}
Applying the Schur complement formula to the $(1,1)$ and $(2,1)$ entries of \eqref{2x2 MDE} and taking the inverse in both cases shows
\begin{equation} \label{MDE11} \frac{1}{a} =   S^*d+\eta  + \frac{\abs{Tb+ \zeta}^2}{S a+\eta}  \end{equation}
and
 \begin{equation} \label{MDE12} b = - \frac{a(T^*\ol{b} + \ol{\zeta})}{S a+\eta}. \end{equation}
Multiplying \eqref{MDE11}  by $a^2$  and then substituting \eqref{MDE12} and its adjoint leads to
\begin{equation} \label{Aeq} a =   ( S a+ \eta  ) \abs{b}^2 +  ( S^*d+ \eta)a^2 . \end{equation}
In a similar fashion we  also get
\begin{equation} 
 \label{Deq} d =    ( S^*d+ \eta  ) \abs{b}^2 +  ( Sa+ \eta)d^2 . \end{equation}

Note that the equations \eqref{Aeq}, \eqref{Deq} involve $b$ in an essential way, unlike  in \cite{EKR}, where $T=0$.
In fact, the case $T=0$, \eqref{Bequation} reduces to  
\[ b = - \frac{a  \ol{\zeta} }{S a+\eta} . \]
After substituting this relationship into \eqref{Aeq} one has
\[ a =   \frac{ S a+ \eta }{|\zeta|^2 +  ( S a+ \eta  ) ( S^*d+ \eta)} \le\frac{ S a+ \eta }{|\zeta|^2}, \]
and a similar relation for $d$. Then in \cite{EKR} the limiting behavior of $a$, as $\eta \to 0$, was deduced from the spectral radius of $S/|\zeta|^2$.
In the $T\ne 0$ case, the solution involves all three variables $a, d,$ and  $b$. The equation for $b$ is particularly critical when  $T$ has some negative entries. In fact, in order to solve \eqref{Bequation}
with the usual MDE analysis, one needs to make the technical assumption that $T$ is entry-wise non-negative, $t_{ij}\ge 0$.

\begin{proof}[Proof of Lemma~\ref{lemma:boundedM}]  
We will rely on basic properties of the solution to the MDE from~\cite{AjankiCorrelated, AEKband} whose locations
in these  papers we will cite
precisely. 
 Most importantly, as a main technical  input for the current and the subsequent
proof we use Lemma~D.1 from \cite{AEKband}. This lemma identifies the behavior of the solution $\bs{M}_\zeta= \bs{M}_\zeta(z)$
to the MDE \eqref{2x2 MDE} with the usual side condition $\Im \bs{M}_\zeta>0$ near the  imaginary axis
when the real part of the spectral parameter $z$ (in our applications  $\Re  z=0$) 
is away from the support of the self-consistent density of states.
Roughly  speaking this lemma states that $\bs{M}_\zeta(i\eta)$, $\eta>0$, extends continuously and in an explicitly controlled way
to $\eta=0$ when $0$ is away from the support of the self-consistent density of states.
Note that the  MDE in \cite{AEKband} is formulated in a very  general von Neumann algebraic setup, in our
application we work with the algebra of $(2N)\times (2N)$ matrices. In particular, $m$ in  \cite{AEKband} corresponds
to $\bs{M}_\zeta$ in this paper with some fixed $\zeta$.

Fix $\zeta\in \wh{\cal{R}}_\tau$. 
According to Proposition 2.1 in \cite{AjankiCorrelated}  there exists a compactly supported measure, $\bs{V}_\zeta$, on $\R$ taking values in the set of positive semidefinite $2N \times 2N$ matrices such that 
\begin{equation}\label{St}
 \bs{M}_\zeta(\ii \eta) = \int_{\R} \frac{\bs{V}_\zeta(d t)}{t - \ii \eta}. 
 \end{equation}
The self-consistent density of states, see~\eqref{scDOS}, is given by $\rho_\zeta(dE):=\frac{1}{\pi}\tr_{2N} V_\zeta(dE)$
 and its analytic extension to the upper half plane is $\rho_\zeta(z) := \frac{1}{\pi}\tr_{2N} \im \bs{M}_\zeta(z)$. 
 
 Next, we prove the bound   $\norm{\bs{M}_\zeta(0)}\lesssim \frac{1}{\tau} $. From \cite[Proposition 2.1]{AEKband}, we have that there exists $C\sim 1$ such that the $\supp \bs{V}_\zeta \subset \{\pm \abs{\zeta}\}+[-C,C]$. Then for $|\zeta| > 2C+1$, we apply the trivial bound $\norm{\bs{M}_\zeta(0)} \leq \dist(0, \supp \bs{V}_\zeta ) < 1 \lesssim
 \frac{1}{\tau}$. We will now assume $|\zeta| < 2C+1$.
From the representation~\eqref{St}, 
we bound the norm of $\bs{M}_\zeta(\ii \eta)$ by considering $ \langle x, \bs{M}_\zeta(\ii \eta)  y \rangle $
 for any vectors $x, y\in \C^{2N}$, which we estimate via the Schwarz inequality as 
 \begin{align*} 
  \Big|\Big\langle x, \int_{\R} \frac{\bs{V}_\zeta(d t)}{t - \ii \eta}   y \Big\rangle\Big|  \leq \frac{1}{2}  \int_{\R}  \frac{ \langle x, \bs{V}_\zeta(d t)    x \rangle +   \langle y, \bs{V}_\zeta(d t)    y \rangle     }{|t - \ii \eta|}   \,.
 \end{align*}
Combining this with the boundedness of $ \supp \bs{V}_\zeta$ and the assumed bound  $|\zeta|\le 2C+1$ yields
\begin{align*} \|\bs{M}_\zeta(\ii \eta) \| & 
\leq \max_{x: \|x\|_2 = 1} \int_{\R}  \frac{ \langle x, \bs{V}_\zeta(d t)    x \rangle }{| t - \ii \eta|}    
  \lesssim  \max_{x: \|x\|_2 = 1} \int_{\R}  \frac{ \langle x, \bs{V}_\zeta(d t)    x \rangle }{| t - \ii \eta|^2}    
  =\max_{x: \|x\|_2 = 1} \frac{1}{\eta}   |\langle x, \Im\bs{M}_\zeta(\ii \eta) x \rangle| \,
\end{align*} 
for $\eta \in (0,1)$.
Taking $\limsup_{\eta\downarrow 0}$ on both sides and using the assumption $\zeta \in \wh{\cal{R}}_\tau$,  along with the representation of $\bs{M}_\zeta$ in \eqref{form of 2x2 M}, yields
\[ \limsup_{\eta\downarrow 0} \|\bs{M}_\zeta(\ii \eta) \| \lesssim \frac{1}{\tau}. \]
The existence of the limit of $\bs{M}_\zeta(\ii\eta)$ as $\eta\downarrow0$ follows from the implication (i)$\Rightarrow$(iii) in Lemma~D.1 of \cite{AEKband}, 
as (i) is guaranteed by  the definition of $\wh{\cal{R}}_\tau$ from \eqref{pseudorestau}.   By definition of $\wh{\cal{R}}_\tau$ in \eqref{pseudorestau} we have $a_\zeta(\ii \eta)\to 0$ and $d_\zeta(\ii \eta) \to 0$ in the  limit $\eta \downarrow 0$ and therefore \eqref{Bequation} implies that ${b}_\zeta(0)$ satisfies \eqref{definition of b}. 
At the end of the proof of Proposition~\ref{prp:Solution of vector equation} in Section~\ref{Subsec:Self consistent pseudospectrum via MDE} we even showed that $\frak{b}(\zeta)$ is the $(2,1)$-component of $\lim_{\eta \downarrow 0} {\bs{M}}_\zeta(\ii \eta)$.

Since $s_{ij}+ \abs{t_{ij}} \lesssim \frac{1}{N}$ by Assumption (B) and \eqref{definition of S and T} we have the bound $\norm{\cal{S}[\bs{M}]} \lesssim \norm{\bs{M}}$.
 The estimate \eqref{largezeta} follows  from  this  and  writing \eqref{2x2 MDE} in the form 
\[
\bs{M}=-(\ii \eta+ \bs{A})^{-1}(\bs{I}+(\cal{S}[\bs{M}])(\ii \eta + \bs{A})^{-1})^{-1}\,, \quad
\mbox{with}  \quad \normb{(\ii \eta  + \bs{A})^{-1}}\lesssim (\eta +\abs{\zeta})^{-1},
\]
for $\bs{M}=\bs{M}_\zeta( \ii \eta)$, $\bf{A} = \bf{A}_\zeta$ and any sufficiently  large $\abs{\zeta}$. % 
The inclusion \eqref{Eps} is a consequence of \eqref{Aeq}, \eqref{Deq} and 
\eqref{largezeta}.  Indeed, using the large $\zeta$  bound from \eqref{largezeta} on  $\bs{M}=\bs{M}_\zeta(\ii \eta)$ in the two equations for $a$ and $d$  we have  $\norm{a}_\infty +\norm{d}_\infty \lesssim \abs{\zeta}^{-2}(\eta +\norm{a}_\infty +\norm{d}_\infty)$ for $\abs{\zeta} > C$. Thus, $\norm{a}_\infty +\norm{d}_\infty \lesssim \eta$ and \eqref{Eps} follows from 
the definition of $\wh{\cal{R}}_\tau$   in \eqref{pseudorestau}.  
\end{proof}

At this stage we are ready  to show the qualitative equivalence  statement, Theorem~\ref{thm:eq}. 

\begin{proof}[Proof of Theorem~\ref{thm:eq}]
The theorem is a consequence of Corollary~\ref{crl:proof of prop 2.4} and  Lemma~D.1 of \cite{AEKband}. Indeed, for $\zeta \in \cal{R}$ the condition in (iii) of Lemma~D.1 in \cite{AEKband} is satisfied by Corollary~\ref{crl:proof of prop 2.4} and, thus, also  (i) of the same lemma, i.e. $\zeta \in \wh{\cal{R}}$. The opposite inclusion $\wh{\cal{R}} \subset \cal{R}$ and the identity \eqref{ABDetaat0} were shown in Lemma~\ref{lemma:boundedM}. Thus, $\wh{\cal{R}} = \cal{R}$. 
The identity \eqref{ABDetaat0} was shown at the end of the proof of Proposition~\ref{prp:Solution of vector equation}.
Finally, \eqref{third} holds by the equivalence (iii)$\Leftrightarrow$(v) Lemma~D.1 from \cite{AEKband}. 
\end{proof}

Now we prepare the proof of the quantitative result, Proposition~\ref{prp:eff}.   Although
we already  know that $\wh{\cal{R}} = \cal{R}$, we will keep on writing $\wh{\cal{R}}$ when the context
requires the definition~\eqref{hatRM} rather than Definition~\ref{def:Rb}.  
The following lemma gives an effective version of the side condition, \eqref{sidecond}, on $\wh{\cal{R}}_{\tau}$.

\begin{lemma} \label{lem:specrad}
Fix $\tau>0$ and let $\zeta \in \wh{\cal{R}}_{C \tau}$, where $C\ge 4\max_i \sum_{j}s_{ij}$ is a positive constant and 
let  $\frak{b}({\zeta})$ be the solution to  \eqref{definition of b}  and \eqref{sidecond}. 
Then
\begin{equation} \label{eq:specradest}\frak{r} \big(\frak{D}_{\abs{\frak{b}(\zeta)}^2}S \big)\leq (1+\tau)^{-1}.\end{equation}
 \end{lemma}

\begin{proof}
We prove the statement with $b=b_{\zeta}(\ii \eta)$ instead of its limit $ \mathfrak {b}(\zeta)  $ (cf. \eqref{ABDetaat0}) with $\eta>0$ and then take  $\eta \to 0$. For brevity, we use $a= a_\zeta(\ii\eta)$, etc. 
Dividing \eqref{Aeq} by $\eta$ and using $d > 0$, (which follows from $\Im \bs{M} $ being positive definite) we have 
\begin{align*}  a/\eta &>  ( Sa/\eta + 1 ) \abs{b}^2
\end{align*} 
entry-wise, which we rearrange to
\begin{equation} \label{eq:specineq}  
a/\eta >  (1+\tau) \abs{b}^2 Sa/\eta    + 
 (  1 -  \tau S a/\eta  )\abs{b}^2 .  \end{equation}

Since $\zeta \in \wh{\cal{R}}_{C \tau}$, for sufficiently small $\eta$ we have 
\[  1 - \tau S a/\eta  \geq 1 - 2C^{-1} S 1\ge \frac{1}{2}, \]
i.e. the second term on the right side of \eqref{eq:specineq} is positive and $a/\eta > (1+\tau)\frak{D}_{\abs{b}^2}Sa/\eta$. Then by taking the inner product with the left Perron-Frobenius eigenvector of $\frak{D}_{\abs{b}^2}S$; as in, for instance \cite[Theorem 1.6]{book-seneta}, we conclude that 
$\frak{r}(\frak{D}_{\abs{b}^2}S) < (1+\tau)^{-1}$. Thus in the limit $\eta \downarrow 0$, \eqref{eq:specradest} also holds.

\end{proof}

\begin{lemma} 
\label{lmm:ell2 bound on b}  For $\zeta \in \wh{\cal{R}}  $ away 
 from the origin $\frak{b}(\zeta)$ is uniformly bounded in $\ell^2$. More precisely, 
 $\norm{\frak{b}(\zeta)}_2 \le \frac{1+\abs{\varrho}}{\abs{\zeta}}$ holds for all $\zeta \in \wh{\cal{R}}$
\end{lemma}
\begin{proof}
%Now we show that
We now prove the following bound on the operator norm of the $\frak{D}_{{\frak{b}}}{T}\frak{D}_{{\frak{b}}}$
\bels{bound on op norm of bTb}{
\norm{\frak{D}_{{\frak{b}}}{T}\frak{D}_{{\frak{b}}}}\le \abs{\varrho}\,,
}
for $\frak{b} =\frak{b}(\zeta)$ and  any $\zeta \in \wh{\cal{R}}$. Indeed, we have the chain of inequalities
\[ \| \mathfrak{D}_\mathfrak{b} T \mathfrak{D}_\mathfrak{b} \| \leq \| \mathfrak{D}_{|\mathfrak{b}|} |T| \mathfrak{D}_{|\mathfrak{b}|}\| 
\le \abs{\varrho}\| \mathfrak{D}_{|\mathfrak{b}|} (S^{(1/2)} \odot (S^*)^{(1/2)} )  \mathfrak{D}_{|\mathfrak{b}|}\| \le\abs{\varrho} \frak{r}(  \mathfrak{D}_{|\mathfrak{b}|^2} S), \]
where we used \eqref{bound T by S} and \eqref{bound SSstar by S}, we emphasize that in each inequality we are comparing the operator norm of the corresponding operator.
 The claim \eqref{bound on op norm of bTb} then follows  from Lemma~\ref{lem:specrad}, recalling that $\wh{\cal{R}} = \bigcup_{\tau>0}\wh{\cal{R}}_\tau$ (cf. Lemma~\ref{lemma:boundedM}). 

To bound $\|\mathfrak{b}(\zeta)\|_2$, we  rearrange \eqref{definition of b} to 
\[ -\mathfrak{b}= \frac{1}{\zeta}(1+  \mathfrak{D}_\mathfrak{b}T\mathfrak{b}). \]
Then we take the $\norm{\2\cdot\2}_2$-norm of the vectors on both sides and use \eqref{bound on op norm of bTb}. 
\end{proof}

The following lemma provides entry-wise upper and lower bounds on $\frak{b}(\zeta)$ when $\zeta$ is bounded away from $0$. 
 The bounds deteriorate as $\zeta$ approaches zero,  i.e. if $\zeta \in \D_\xi:= \{ z: |z|< \xi\}$ for small $\xi$. 
 Since these bounds play a crucial role in the upcoming analysis we introduce a variant of our comparison 
 relations $a \lesssim b$ and $a\sim b$ that track  the dependence of the implicit constant on $\xi$ when 
 $\zeta \in \C\setminus\D_\xi =\{\zeta: \abs{\zeta} \ge \xi\}$, namely we write $a \lesssim^\xi b$ for two quantities $a$ and $b$ that depend on $\zeta \in \D_\xi^c$ whenever $a \lesssim \xi^{-C_*} b$ holds for some positive  constant $C_*$, depending only on $\abs{\varrho}$ and $L$ in \eqref{eq:prim}. In particular, the following lemma shows that $\abs{\frak{b}(\zeta)}\sim^\xi 1/(1+\abs{\zeta})$ for $\zeta \in \wh{\cal{R}}\setminus \D_\xi$.

\begin{lemma}\label{lmm:bbound}
For $\zeta\in \wh{\cal{R}}$ the entries of $\mathfrak {b}(\zeta)$ and their derivatives satisfy 
\begin{equation}\label{bbound}
  \frac{ \abs{\zeta}}{1+\abs{\zeta}^2}\lesssim \abs{\mathfrak{b}(\zeta) }\lesssim \frac{1}{\abs{\zeta}}\,, \qquad \abs{\partial_\zeta \frak{b}(\zeta)} \lesssim \frac{1}{\abs{\zeta}^2}\,.
\end{equation} 
Furthermore, for any $\zeta \in \wh{\cal{R}}\cup \partial \wh{\cal{R}}$ with $\abs{\zeta}\ge \xi$  the function $\zeta \mapsto \frak{b}(\zeta)$ admits a holomorphic extension to a  $\delta$-neighborhood  $\D_\delta(\zeta)$ of $\zeta$ in $\C$ whose size  $\delta$  depends only on model parameters and on $\xi$, i.e. $\delta \gtrsim^\xi 1$. 
 The bounds \eqref{bbound} remain valid for this extension. 
\end{lemma}
\begin{proof}
We will show the upper bound on $\frak{b}(\zeta)$ first. The lower bound then follows from  
$1/\frak{b}(\zeta)= - (\zeta + T \frak{b}(\zeta))$ since $\frak{b}(\zeta)$ solves  
\eqref{definition of b}.  The behavior at large $\zeta$ is clear from  \eqref{definition of b}
and $| \frak{b}(\zeta)|\lesssim 1/|\zeta|$ following from \eqref{largezeta}
in  Lemma~\ref{lemma:boundedM}. Thus, it suffices to show $\abs{\mathfrak{b}(\zeta)} \lesssim \frac{1}{\abs{\zeta}}$ on the bounded set $\abs{\zeta}\le C$.  To do this, 
using the regularity of $T$ (assumption (2.F)) we extend the $\ell^2$ bound
 obtained in Lemma~\ref{lmm:ell2 bound on b}  to a uniform bound on the entries of $\mathfrak{b}(\zeta)$ exactly as in \cite{AjankiQVE}, Section 6.1. One simply follows the proof of  Proposition 6.6
in  \cite{AjankiQVE} line by line and sees $T$ satisfies all of the necessary properties.  
 This proves the first formula in \eqref{bbound}. 

To prove the bound~\eqref{bbound} on $\abs{\partial_\zeta \frak{b}(\zeta)}$,  
we  divide \eqref{definition of b} by $\frak{b}(\zeta)$ and  differentiate   with respect to $\zeta$ to  find
\bels{partial b}{
\partial_\zeta \frak{b} = \frak{D}_{\frak{b}}(1-\frak{D}_{\frak{b}}T\frak{D}_{\frak{b}})^{-1}{\frak{b}}
}
with $\frak{b}=\frak{b}(\zeta)$. From the upper bound on $\frak{b}$ and \eqref{bound on op norm of bTb}
 with $|\varrho|<1$  we conclude the bound on the derivative.

The invertibility of $1-\frak{D}_{\frak{b}}T\frak{D}_{\frak{b}}$ in \eqref{partial b}  also  % 
shows the existence of a holomorphic extension as claimed in Lemma~\ref{lmm:bbound},  
by the implicit function theorem. Indeed, we only need to check the stability of the defining  equation $J_\zeta(b)=0$ with $J_\zeta$ as in \eqref{Jzeta},
 which is equivalent to \eqref{definition of b}. By \eqref{nabla J} we have 
\[
\nabla J_\zeta(\frak{b})=-(1-\frak{D}_{\frak{b}}T\frak{D}_{\frak{b}})\frak{D}_{\frak{b}}^{-1},
\] 
which is invertible, proving the stability. This completes the proof of \eqref{bbound}. 
\end{proof}

Next we show that the MDE \eqref{2x2 MDE} is stable on the self-consistent pseudo-resolvent set, i.e. for $\zeta \in  \cal{R}=\wh{\cal{R}}$ (cf. Theorem~\ref{thm:eq}), and that this stability can be quantified in terms of $\Delta_{\zeta}$ and the distance of $\zeta$ to the origin.

\begin{lemma}\label{lmm:Stability}Let $\zeta\in {\cal{R}}$ with $\min\{\abs{\zeta},\Delta_{\zeta}\}\ge \xi$ for some $\xi \in (0,1)$ and  $\cal{L}_\zeta $ be the stability operator, acting on block diagonal matrices,  as defined  in \eqref{Def of L zeta}. Then  we
have 
\bels{Lbound}{
\norm{\cal{L}_{\zeta}^{-1}}_{sp}\lesssim^\xi 1.
}
 Here, $\norm{\cal{A}}_{sp}$ denotes the operator norm of $\cal{A}: \C^{2 \times 2} \otimes \C^N \to \C^{2 \times 2} \otimes \C^N$ induced by the Hilbert-Schmidt norm   $\| \bs{A} \|_{HS}:= ( \tr_{2N} \bs{A}\bs{A }^*)^{1/2}$ on block diagonal matrices $\bs{A} \in \C^{2 \times 2} \otimes \C^N $.
\end{lemma}

\begin{proof}
The action of $\cal{L}_\zeta$ on block diagonal matrices was described in \eqref{definition of lij matrices}. Each $L_{ij}$ from \eqref{definition of lij matrices} is of the form $L_{ij} =1-\frak{D}_{b_lb_r}Z$ with  some choice
$Z \in \{{T},{T}^*,{S},{S}^*\}$ and $b_l ,b_r \in \{\frak{b},\ol{\frak{b}}\}$ depending on $(i,j)$.    For any matrix $R \in \C^{N \times N}$ with $\frak{r}(\abs{R})<1$ by expanding the Neumann series we have the bound $\norm{(1-R)^{-1}} \le \norm{(1-\abs{R})^{-1}}$, where $\abs{R} = (\abs{r_{ij}})_{i,j}$ is the entry-wise absolute value of $R$.  Therefore, when $Z \in \{T,T^*\}$ we have, for any vector $u$, the bound 
\bels{estimate for Lij}{
\norm{{L}_{ij}^{-1} u}_{2} \le \norm{(1-\frak{D}_{\abs{b_lb_r}}\wt{S})^{-1}}\norm{u}_{2}\,,
}
 using that $|t_{ij}|\le \wt{s}_{ij}$ with $\wt{s}_{ij} := \sqrt{s_{ij}s_{ji}}$ the entries of $\wt{S}$. In this case we estimate \begin{align}\label{estimate for b2 tilde S}
\norm{(1-\frak{D}_{\abs{b_lb_r}}\wt{S})^{-1}} &\lesssim^\xi\normb{\pb{1-\frak{D}_{\abs{b_lb_r}}^{1/2}\wt{S}\frak{D}_{\abs{b_lb_r}}^{1/2}}^{-1}} \\
&= \frac{1}{1-\frak{r}(\frak{D}_{\abs{b_lb_r}}\wt{S})} \le \frac{1}{1-(\frak{r}(\frak{D}_{\abs{b_l}^2}S)\frak{r}(\frak{D}_{\abs{b_r}^2}S))^{1/2}}\,, \notag
\end{align}
where we have multiplied $(1-\frak{D}_{\abs{b_lb_r}}\wt{S})^{-1}$ in the first expression by $\frak{D}_{\abs{b_lb_r}}^{-1/2} \frak{D}_{\abs{b_lb_r}}^{1/2}$ on both sides and then used $\abs{\frak{b}({\zeta})} \sim^\xi 1$ from Lemma~\ref{lmm:bbound}. The last inequality uses \eqref{eq:hadspecrad}. For $Z \in \{S,S^*\}$ we get
\[
\norm{{L}_{ij}^{-1} u }_{2} \le \norm{(1-\frak{D}_{\abs{\frak{b}}^2}Z)^{-1}}\norm{u}_{2}\,.
\] 
To estimate this further we write
\[
\norm{(1-\frak{D}_{\abs{\frak{b}}^2}Z)^{-1}} = \frak{r}\norm{\frak{r}^{-1}-R}\,,
\]
where
 $\frak{r}=\frak{r}(\frak{D}_{\abs{\frak{b}}^2}Z)$ and $R=\frak{r}(\frak{D}_{\abs{\frak{b}}^2}Z)^{-1}\frak{D}_{\abs{\frak{b}}^2}Z$. Then we apply the following lemma, whose proof is postponed until after the proof of Lemma~\ref{lmm:Stability}. 
 Condition \eqref{quanprim} below,
  is satisfied   for some $\eps \sim^\xi 1$
   by Assumption \eqref{eq:prim} (with the same $L$ power) and by \eqref{bbound}.
\begin{lemma} \label{lemm:quanprim}
Let $R \in \R^{N \times N}$  have non-negative entries, be normalized through $\frak{r}(R)=1$ and  satisfy 
\begin{equation}\label{quanprim}
\frac{\eps}{N} \le  (R^L)_{ij} \le \frac{1}{\eps N}\,, \qquad ((R^*R)^L)_{ij} \ge \frac{\eps}{N}\,,
\end{equation}
for some $\eps  \in (0,1)$. 
Let $v_l$ and $v_r$ denote its left and right Perron-Frobenius eigenvectors corresponding to the isolated non-degenerate eigenvalue $\frak{r}(R)=1$, respectively.  Then 
\begin{align}\label{eq:R resolvent bound}
&\eps\avg{v_l}\le v_l \le \eps^{-1}\avg{v_l}\,, \quad\eps\avg{v_r}\le v_r \le \eps^{-1}\avg{v_r}\,, \\
&  \norm{(R-z)^{-1}}\le \frac{1}{\eps^2}\pbb{\frac{1}{\abs{1-z}}+ \frac{1}{\abs{z}-1+(2L+2)^{-1}\eps^7}}\,,\notag
\end{align}
hold for any $z \in \C$ with $\abs{z} >1-(2L+2)^{-1}\eps^7$.  In particular, $\spec(R) \subset \{1\} \cup  \D_{1-(2L+2)^{-1}\eps^7}$ . 
\end{lemma}
 We conclude that $\norm{\frak{r}^{-1}-R} \lesssim^\xi  (\frak{r}^{-1}-1)^{-1}$. Thus, the bound on $\norm{{L}_{ij}^{-1}u}_{2}$, for all possible choices of $Z$, $b_l$ and $b_r$ and $u$, implies 
 \begin{equation}\label{estimate for Lzeta}
 \norm{\cal{L}_{\zeta}^{-1}}_{sp}\lesssim^\xi \Delta_\zeta^{-1}\le \xi^{-1}\lesssim^\xi 1.
 \end{equation}
 This completes the proof of Lemma~\ref{lmm:Stability}.
\end{proof}

\begin{proof}[Proof of Lemma~\ref{lemm:quanprim}] Since left and right eigenvectors of $R$ and $R^L$ coincide and $R^L$ has strictly positive entries, the eigenvectors $v_l$ and $v_r$ with $Rv_r=v_r$ and $R^*v_l =v_l$, respectively, are unique (up to scaling) and the eigenvalue $\frak{r}(R)=1$ is isolated due to the Perron-Frobenius theorem.  The upper and lower bounds on the eigenvectors in \eqref{eq:R resolvent bound} are an immediate consequence of the upper and lower bounds on $R^L$ in \eqref{quanprim}. The proof of the bound on the resolvent of $R$ in \eqref{quanprim}  follows exactly the proof of Lemma~A.1 from \cite{EKR} by simply tracking the dependence on $\eps$ explicitly.
The bounds on the location of the spectrum then follow from the boundedness of the resolvent.
 \end{proof}
 
 Now we quantify for  $\zeta \in \cal{R}$ the gap in the support of the self-consistent density of states $\rho_\zeta$ associated with the MDE \eqref{2x2 MDE}. In other words, we prove the second inclusion in \eqref{RR}. 
 \begin{lemma} \label{lmm:Gap in rho zeta} For $\zeta\in \cal{R}$ with $\min\{\abs{\zeta},\Delta_{\zeta}\}\ge \xi$ we have  $\dist(\supp \rho_{\zeta},0)  \gtrsim^\xi1$. 
 \end{lemma}

 \begin{proof}
As in the proof of Proposition~\ref{prp:Solution of vector equation} we construct a local solution $z \mapsto \bs{M}_\zeta(z)$ of the MDE \eqref{2x2 MDE} in a neighborhood of $z=0$ in $\C$. However, this time we need an effective control on the size of the neighborhood and, thus, require the bound $\norm{\cal{L}_{\zeta}^{-1}}_{sp}\lesssim^\xi 1$ from Lemma~\ref{lmm:Stability}.

By the implicit  function theorem applied on the space of block diagonal matrices with the norm $N^{-1/2}\norm{\1\cdot\1}_{HS}$, the equation $J_{z,\zeta}(\bs{M})=0$ with $J_{z,\zeta}(\bs{M}):= \bs{M}-\Phi^{2 \times 2}_{z,\zeta}(\bs{M})$ (cf. \eqref{2x2 MDE})  has a unique local solution $z \mapsto \bs{M}_\zeta(z)$ for $z\in \D_\delta(0)$ with some $\delta>0$. In fact, here  $\delta \gtrsim^\xi1$ because of the quantitative bound $\norm{\cal{L}_{\zeta}^{-1}}_{sp}\lesssim^\xi 1$. Since $J_{z,\zeta}$ maps self-adjoint matrices to  self-adjoint matrices for real parameters $z$, we conclude that $\bs{M}_\zeta(z)=\bs{M}_\zeta(z)^*$ if $z \in \R$. Furthermore, as we already saw in the proof of Proposition~\ref{prp:Solution of vector equation}, the relation $\cal{L}_{\zeta}\big[\partial_\eta \bs{M}_{\zeta}(\ii \eta)|_{\eta=0}\big]=  \ii\bs{M}_{\zeta}(0)^2$ implies that $\im \bs{M}_\zeta(\ii \eta)$ is positive definite for any sufficiently small $\eta >0$. In particular, the local solution coincides with the standard MDE solution for $z$ in the complex upper half plane. Since $\rho_\zeta(z) = \frac{1}{\pi}\tr_{2N} \im \bs{M}_\zeta(z)=0$ 
 for  $z \in \D_\delta(0)\cap \R$, we conclude $\dist(\supp \rho_{\zeta},0) \ge \delta \gtrsim^\xi1$. 
 \end{proof}

 We have  collected the ingredients to show Proposition~\ref{prp:eff}, the quantitative version of  Theorem~\ref{thm:eq}.
 
 \begin{proof}[Proof of Proposition~\ref{prp:eff}] 
The first relation in~\eqref{RR}  follows from Lemma~\ref{lem:specrad} and the second 
relation follows from
Lemma~\ref{lmm:Gap in rho zeta}.    Finally,
the third relation comes 
 from the effective version\footnote{The constants in the different 
 equivalent statements (i)--(vi) of  Lemma~D.1 of \cite{AEKband}
 were stated to depend on each other effectively; in fact this 
 dependence is   polynomial  following directly from
 that proof.}
  polynomial dependence of the constants 
 in the 
 implication (v)$\Rightarrow$(i) in Lemma~D.1 of \cite{AEKband}.
\end{proof}

 Now we have the necessary tools to prove  Theorem~\ref{prp:b as resolvent}. We will apply  Corollary~2.3 from \cite{EKS} for random matrices with correlated entries to the matrix $\bs{H}_\zeta$. This corollary asserts that $\bs{H}_\zeta$ does not have eigenvalues away from the support of its associated self-consistent density of states. 
We note that the matrix $\bs{H}_\zeta$ does not satisfy assumption (CD) of \cite{EKS}. 
The condition (CD) was designed to describe
ensembles $H$, where only those matrix elements are strongly correlated that 
are close to each other within the matrix $H$. The matrix $\bs{H}_\zeta$ generated from an elliptic-type ensemble $X$ has strongly correlated
matrix elements $x_{ij}$ and $x_{ji}$ that are positioned far from each other.
However,  by Example 2.11 of \cite{EKS}  on block matrices,  $\bs{H}_\zeta$ satisfies the  more general assumptions (C) and (D), under which Corollary 2.3 of \cite{EKS}  still holds.

\begin{proof}[Proof of Theorem~\ref{prp:b as resolvent}]
We begin by proving $(i)$.   
Let $\zeta\in \cal{R}$ with $\min\{\abs{\zeta},\Delta_{\zeta}\}\ge \xi$ for $\xi = N^{-\gamma}$ for some sufficiently small $\gamma>0$ to be chosen later.
In particular, here $\xi$ 
is $N$-dependent. 
 In Lemma~\ref{lmm:Gap in rho zeta} we have already seen that $\dist(\supp \rho_{\zeta},0)  \gtrsim^\xi1$, i.e. that zero lies outside the asymptotic spectrum of $\bs{H}_\zeta$.  Now we use this information  to apply Corollary~2.3 from \cite{EKS} to see that zero is not an eigenvalue of $\bs{H}_\zeta$ with very high probability. From the definition of $\bs{H}_\zeta$ in \eqref{2x2 symmetrization} this is equivalent to $\zeta$ not being an eigenvalue of $X$.

Indeed, by Corollary~2.3 from \cite{EKS}  we conclude that 
\bels{Gap in spectrum of Hzeta}{
\P \pb{\spec (\bs{H}_\zeta )\cap [-N^{-C \gamma},N^{-C \gamma}]\ne \emptyset}\le N^{-\nu}
}
for all $\nu \in \N$ and sufficiently large $N$, where $C=C_*>0$ is a  constant, depending only on $\abs{\varrho}$ and $L$ in \eqref{eq:prim}. By a standard stochastic continuity argument, i.e. by choosing a fine grid of points $\zeta \in {\cal{R}}$ with $\min\{\abs{\zeta},\Delta_\zeta\} \ge N^{-\gamma}$ for $\gamma=c_*$ a constant depending only on $\abs{\varrho}$ and $L$ in \eqref{eq:prim}, taking a union bound of the event in \eqref{Gap in spectrum of Hzeta} over this grid and then using the  Lipschitz-continuity of the spectrum of $\bs{H}_\zeta$ in $\zeta$ with $N$-independent Lipschitz constant,  we infer 
\[
\P \pb{0 \not \in \spec (\bs{H}_\zeta) \text{ for all } \zeta \in {\cal{R}} \text{ with } \min\{\abs{\zeta},\Delta_\zeta\} \ge N^{-c_*} }\ge 1- N^{-\nu}
\]
 for all $\nu \in \N$ and sufficiently large $N \in \N$. This implies the statement of $(i)$ in Theorem~\ref{prp:b as resolvent}
  since $0 \in \spec (\bs{H}_\zeta)$ is equivalent to $\zeta\in \spec (X)$.
  
Now we show $(ii)$ of Theorem~\ref{prp:b as resolvent}.
 Using that $\abs{\zeta}\ge N^{-c_*\epsilon}$ and $\Delta_\zeta\ge  N^{-c_*\epsilon}$  from the
 conditions of part (ii), similarly  to as in the proof of $(i)$ we have 
 $\dist(0, \supp \rho_\zeta) \gtrsim N^{-C_* \1\eps}$ 
 for some constant $C_*>0$, depending on $\abs{\varrho}$ and $L$ in \eqref{eq:prim}. Then we apply Theorem~2.1 from \cite{EKS} to the matrix $\bs{H}_\zeta$ using that for its resolvent at the origin $(\bs{H}_\zeta^{-1})_{21}=(X-\zeta)^{-1}$ and that $(\bs{M}_\zeta(0))_{21} =\frak{D}_{\frak{b}(\zeta)}$ by \eqref{form of 2x2 M} and \eqref{ABDetaat0}.  
\end{proof}

\section{Hermitization for resolvent products}\label{sec:herm}

In the previous section we hermitized the resolvent $(X-\zeta)^{-1}$ of $X$ via \eqref{2x2 symmetrization}
 and studied its deterministic limit via the MDE \eqref{2x2 MDE} for $2N\times 2N$ matrices, 
in fact studying $2\times 2$ block diagonal matrices was sufficient. From this analysis we proved
that the spectrum of $X$ concentrates close to a deterministic set, the self-consistent pseudospectrum.
Moreover, the resolvent $(X-\zeta)^{-1}$ is sufficient to compute $\tr_N f(X)$ for analytic functions $f$.
For our basic quantity $\tr_N f(X)g(X^*)$, however, we need to understand the product of two resolvents
$(X-\zeta_1)^{-1}$ and $(X^*-\ol{\zeta}_2)^{-1}$ with different spectral parameters, which requires a bigger
hermitization.

In this section 
 we will linearize and hermitize the product $(X-\zeta_1)^{-1}(X^*-\ol{\zeta}_2)^{-1}$. % 
For this purpose  we introduce $4 \times 4$-block matrices consisting of $N \times N$ blocks. To distinguish the matrices of various sizes, we now introduce some notation  that will be valid in Sections~\ref{sec:herm}--\ref{sec:MDE}.
 In what follows bold capital letters (for example $\bs{R}$)  are used to denote $4N \times 4N$ matrices and bold calligraphic letters (for example $\bs{\cal{S}}$) to denote linear operators on $4N \times 4N$ matrices. Operators on $N \times N$ matrices are denoted by script letters (for example $\scr{S}$).
Given a $4 N \times 4 N$ block matrix $\bs{R}$ we define $R_{ij}$, $1 \leq i,j \leq 4$, to be its $(i,j)$ block. Given an $N\times N$ matrix $R$, the matrix $\bs{E}_{ij}(R)$, $1 \leq i,j \leq 4$, is a $4 N \times 4 N$ block matrix with $(i,j)$ block equal to $R$ and the remaining blocks equal to zero. We use the shorthands $\bs{E}_{ij}(r):= \bs{E}_{ij}(\frak{D}_r)$ for $r \in \C^{N}$, as well as $\bs{E}_{ij}:= \bs{E}_{ij}(I)$. 
The norm, $\|\cdot \|$, when applied to matrices will denote the usual operator norm induced by the Euclidean vector norm. When $\|\cdot \|$ is applied to operators acting on matrices, it denotes the 
operator norm induced by the matrix norm $\| \cdot\|$. 

\subsection{Structure of MDE}
Given analytic functions $f$ and $g$, we will show  that to compute $f(X)g(X^*)$, it suffices to consider the following three parameter family of matrices $\bs{H}^{\mathfrak{Z}}_\alpha$ 
with some small $\alpha\in \R$ and $\mathfrak{Z}:=(\zeta_1,\zeta_2) \in \C^2$, as well as their resolvents $\bs{G}^{\mathfrak{Z}}_\alpha(z)$ at a spectral parameter $z$ in the upper half plane, $z \in \C^+$:

\bels{def of H alpha}{\bs{H}^{\mathfrak{Z}}_\alpha := 
\mfour
{0 & 0 & 0 & X^* - \ol{\zeta}_2 }
{ 0 & 0 & X -\zeta_1  & \alpha }
{0 & X^*- \ol{\zeta}_1  & 0 & 0}
{ X - {\zeta}_2& \alpha & 0 & 0}, \qquad  \bs{G}^{\mathfrak{Z}}_\alpha(z) :=  (\bs{H}^{\mathfrak{Z}}_\alpha  - z )^{-1}.}
The reason for constructing $\bs{H}^{\mathfrak{Z}}_\alpha$ in this way is that $\partial_\alpha(\bs{G}^{\mathfrak{Z}}_\alpha(0))_{31}|_{\alpha=0} = (X-\zeta_1)^{-1}(X^*-\ol{\zeta}_2)^{-1}$ whenever both sides exist. Therefore the asymptotic analysis of the resolvent $\bs{G}^{\mathfrak{Z}}_\alpha(0))_{31}$ of the hermitian matrix $\bs{H}^{\mathfrak{Z}}_\alpha$ provides information about the resolvent product of interest. Furthermore, for $\alpha=0$, the matrix $\bs{H}^{\mathfrak{Z}}_\alpha$ decouples into the direct sum of two linearizations of the form \eqref{2x2 symmetrization} at $\zeta = \zeta_1$ and $\zeta = \zeta_2$, respectively.

The MDE \eqref{eq:genMDE} corresponding to this $4N \times 4N$ matrix takes the form
\begin{equation} \label{eq:MDE} -\bs{M}^{\mathfrak{Z}}_\alpha(z)^{-1}  = z \bs{I} + \bs{A}^{\mathfrak{Z}} + \bs{\alpha} + \bs{\cal{S}}[\bs{M}^{\mathfrak{Z}}_\alpha(z)] ,
\end{equation}
where $\bs{\cal{S}}: \C^{4N \times 4N} \to \C^{4N \times 4N}$ is given by
\[
\bs{\cal{S}}[\bs{R}]
\,=\,
\mfour
{ \scr{S}^*[R_{44}] &\scr{T}^*[R_{43}] & \scr{S}^*[R_{42}] &  \scr{T}^*[R_{41}]}
{\scr{T}[R_{34}] & \scr{S}[R_{33}] & \scr{T}[R_{32}]& \scr{S}[R_{31}]}
{\scr{S}^*[R_{24}]&\scr{T}^*[R_{23}]& \scr{S}^*[R_{22} ] &\scr{T}^*[R_{21}] }
{\scr{T}[R_{14}]  &\scr{S} [R_{13}]  & \scr{T}[R_{12}]  &\scr{S}[ R_{11}]}\,,
\]
\[ \bs{A}^{\mathfrak{Z}} = \mfour
{0 & 0 & 0 &  \ol{\zeta}_2 }
{ 0 & 0 & {\zeta_1} &0}
{0 & \ol{\zeta_1} & 0 & 0}
{ \zeta_2 & 0& 0 & 0},\quad \bs{\alpha} = \alpha( \bs{E}_{24} +\bs{E}_{42}).\]
 The operators $\scr{S},\scr{S}^*, \scr{T}$ and $\scr{T}^*$ were introduced in \eqref{def S and T operators} and \eqref{def S* and T* operators}.

We let  $\bs{M}^{\mathfrak{Z}}_\alpha(z)$ denote the unique solution to \eqref{eq:MDE} with positive imaginary part. To this solution we associate the self-consistent density of states as in \eqref{scDOS}, i.e.  
a probability measure $\rho^{\mathfrak{Z}}_\alpha$ on $\R$ such that 
\begin{equation} \label{def of DOS}  \tr_{4N}  {\bs{M}}^{\mathfrak{Z}}_\alpha(z) = \int_\R \frac{\rho^{\mathfrak{Z}}_\alpha(dx) }{x-z}  . \end{equation}

We will consider $\bs{M}^{\mathfrak{Z}}_\alpha(z)$ at $z=\ii \eta$,  $\eta>0$,  and its derivative with respect to $\alpha$, for $\alpha$ in a small neighborhood of $0$.  As we already mentioned,
when $\alpha=0$, equation \eqref{eq:MDE} decouples into two sets of equations, one depending on $\zeta_1$ and one on $\zeta_2$. 
The restriction of \eqref{eq:MDE} to either of these equations carries the information for a single resolvent $(X -  \zeta)^{-1}$, in particular it allows us to determine the location of the pseudospectrum of $X$ in the $N \to \infty$ limit as demonstrated in Theorem~\ref{prp:b as resolvent}-$(i)$.
%Lemma~\ref{lmm:concentration of spectrum}.
 When considering this restriction to the first and fourth blocks and to the second and third blocks, we have the following identification 
\begin{equation} \label{def:ABD} 
\bs{M}_{\zeta_1}(\ii \eta) = \mtwo{\left({\bs{M}}^{\mathfrak{Z}}_0(\ii \eta)\right)_{22} & \left({\bs{M}}^{\mathfrak{Z}}_0(\ii \eta)\right)_{23}  }{\left({\bs{M}}^{\mathfrak{Z}}_0(\ii \eta)\right)_{32}  & \left({\bs{M}}^{\mathfrak{Z}}_0(\ii \eta)\right)_{33}  }\,, \qquad
\bs{M}_{\zeta_2}(\ii \eta) = \mtwo{\left({\bs{M}}^{\mathfrak{Z}}_0(\ii \eta)\right)_{11} & \left({\bs{M}}^{\mathfrak{Z}}_0(\ii \eta)\right)_{41}  }{\left({\bs{M}}^{\mathfrak{Z}}_0(\ii \eta)\right)_{14}  & \left({\bs{M}}^{\mathfrak{Z}}_0(\ii \eta)\right)_{44}  }
,
 \end{equation}
 where $\bs{M}_\zeta(\ii \eta)$ is the solution of the $2 \times 2$-MDE \eqref{2x2 MDE} 
  with block-diagonal structure given in   \eqref{form of 2x2 M}. All other entries of ${\bs{M}}^{\mathfrak{Z}}_0(\ii \eta)$ vanish identically.

\subsection{Stability of MDE at zero}\label{sec:MDEat0}
In this section we consider the MDE \eqref{eq:MDE} for some $\mathfrak{Z}=(\zeta_1, \zeta_2)
\in {\cal{R}}\times {\cal{R}}$ when $\alpha = \eta = 0$,
 and bound the stability operator at this point (cf. \eqref{gen stability operator}). 
 Recall  the key information on the vector $\mathfrak{b}(\zeta) $ summarized in 
 Proposition~\ref{prp:Solution of vector equation}. 
  In what follows, we use the short hand $\mathfrak {b}_i =\mathfrak {b}(\zeta_i) $, $i=1,2$. 
  With the extension of the solution to the  $2\times 2$ block MDE \eqref{2x2 MDE} 
  to $\eta =0$ from Lemma~\ref{lemma:boundedM}, we have that 
\begin{equation}
 \label{eq:MDEat0}
\bs{M}_0^{\mathfrak{Z}}(0)\,:=\,\bs{E}_{14}\p{ \mathfrak {b}_2 }+\,\bs{E}_{23}\p{\ol{\mathfrak {b}}_1 }+\,\bs{E}_{32}\p{ \mathfrak {b}_1}+\,\bs{E}_{41}\p{ \ol{\mathfrak {b}}_2}
\end{equation}
solves  \eqref{eq:MDE}, with $\eta =  \alpha = 0$. Recall  that for any vector $r\in \C^N$ 
 the matrix $\bs{E}_{ij}(r)$ is the $4 N \times 4 N$ block matrix with $(i,j)$ block equal to the diagonal matrix  $\frak{D}_r$ and the remaining blocks equal to zero. 
Using $\bs{M}=\bs{M}_0^{\mathfrak{Z}}(0)$ we consider the stability operator of the MDE \eqref{eq:MDE}
 $\bs{\cal{L}}:=\bs{1}-\bs{\cal{C}}_{\bs{M}}\bs{\cal{S}}$. Here we used the notation
$\bs{\cal{C}}_{\bs{M}}[\bs{R}]:= \bs{M} \bs{R}\bs{M}$ for the sandwiching operator acting on any  matrix $\bs{R}$.
Bounding the inverse of the stability operator at $\alpha=\eta=0$ will allow us to deduce properties of the solution to the MDE in a neighborhood of $\alpha=\eta=0$.

We will work on the space of $4N\times 4N$-matrices equipped with the norm $\| \cdot \|$, induced by the Euclidean norm,  as well as with the
 Hilbert-Schmidt norm, $\| \bs{R} \|_{HS}:= ( \tr_{4N} \bs{R}\bs{R }^*)^{1/2}$.  
 These norms induce the  operator norm $\|\cdot \|$ and  the spectral norm,
$\|\cdot \|_{sp}$, respectively, on operators acting on such matrices.

\begin{lemma}\label{lem:CMnorm}
Assume that $\mathfrak{Z} =(\zeta_1, \zeta_2) $ with $\zeta_1,\zeta_2 \in {\cal{R}}\setminus \D_\xi$ for some $\xi \in[ N^{-c},1]$ with a sufficiently small universal constant $c>0$.
 Then the following bound holds:
\begin{equation} \label{eq:stabopbound} \| (\bs{1}-\bs{\cal{C}}_{\bs{M}_0^{\mathfrak{Z}}(0)}\bs{\cal{S}})^{-1}\| \lesssim^\xi \Delta_\mathfrak{Z}^{-1}, \end{equation}
  where the implicit constant in \eqref{eq:stabopbound} depends on  the model parameters and we defined
  \bels{definition of Delta Zeta}{
  \Delta_{\mathfrak{Z}}:= \min\{\Delta_{\zeta_1}, \Delta_{\zeta_2}\}\,,
  }
  with $\Delta_\zeta$ as in \eqref{def of Delta zeta}.
\end{lemma}

\begin{proof} Set $ \bs{\cal{L}}:=\bs{1}-\bs{\cal{C}}_{\bs{M}_0^{\mathfrak{Z}}(0)}\bs{\cal{S}}$. 
First note that  it suffices to bound  $\| \bs{\cal{L}}^{-1}\|_{sp}$ %
since it directly implies a comparable bound for  $\| \bs{\cal{L}}^{-1}\|$ exactly as in  \cite{EKR}, Lemma~3.1.
Similarly to the $2\times 2$-setting from the proof of Proposition~\ref{prp:Solution of vector equation}, the operator 
 $\bs{\cal{L}}$ leaves the blocks in the $4 \times 4$-block structure on $\C^{4N \times 4N}$ invariant, i.e. there are operators $\cal{L}_{ij}: \C^{N \times N} \to \C^{N \times N}$ such that $\bs{\cal{L}}[\bs{E}_{ij}(R)] = \bs{E}_{ij} (\cal{L}_{ij}[R])$ for $i,j =1, \dots ,4$. 
  Each operator $\cal{L}_{ij}$ is of the form $\cal{L}_{ij} [R]=R-\frak{D}_{b_l}(\cal{Z}[R])\frak{D}_{b_r} $ 
   for some choice  $\cal{Z}% 
  \in \{ \scr{S},\scr{S}^*, \scr{T}$ , $\scr{T}^*\}$
and $b_l ,b_r\in \{\frak{b}_1,\ol{\frak{b}}_1,\frak{b}_2,\ol{\frak{b}}_2\}$
 depending on $(i,j)$. Recall that $\abs{\frak{b}_i} \sim^\xi 1$ from Lemma~\ref{lmm:bbound}. We already encountered this situation in the proof of Theorem~\ref{prp:b as resolvent}-$(i)$ %Lemma~\ref{lmm:concentration of spectrum} 
 with the only difference being that now we have two spectral parameters $\zeta_1$ and $\zeta_2$ and thus two vector valued functions   $\frak{b}_1$ and $\frak{b}_2$. However, this does not effect the proof of the upper bound 
 on $\cal{L}_{ij}^{-1}$  from \eqref{estimate for Lzeta} when restricted to block diagonal matrices, except that the expression on the right hand side is now bounded by  $\Delta_\mathfrak{Z}^{-1}$ instead of $\Delta_\zeta^{-1}$.  We recall that both diagonal and off-diagonal matrices are left invariant by $\cal{L}_{ij}$. Hence we have now proved invertibility on block diagonal matrices 
  
  It remains to show the bound on $\cal{L}_{ij}^{-1}$ when restricted to off-diagonal matrices, i.e. matrices $R \in \C^{N \times N}$ with 
  $r_{kk}=0$. For this purpose we use the easily  checkable bound $\norm{\cal{Z}[R]}_{HS}\lesssim \frac{1}{N} \norm{R}_{HS}$ for off-diagonal $R$. 
   The upper bound on $\frak{b}$ in \eqref{bbound}  together with 
     $|\zeta_i|\ge N^{-c}$  for some small $c>0$ 
     guarantees that $\| \frak{D}_{b_l}(\cal{Z}[R])\frak{D}_{b_r}\|_{HS} \le N^{-1/2} \| R\|_{HS}$,
     i.e.  $\cal{L}_{ij}$ has bounded inverse on off-diagonal matrices.  This establishes  \eqref{eq:stabopbound}. 
\end{proof}

\subsection{Expansion of the MDE near zero} \label{sec:nerozero}
We now extend the bound on the solution to the MDE and on the stability operator   from the special $\alpha=\eta=0$ case discussed in 
Section~\ref{sec:MDEat0} to an entire neighborhood  
\bels{def of Upsilon}{
\Upsilon_{\mathfrak{Z}} := \{ (\alpha,z)\in \R\times \C  :  |\alpha|, |z| <  \kappa \Delta_\mathfrak{Z}^2/4   \},
}
where $\Delta_{\mathfrak{Z}}$ was defined in \eqref{definition of Delta Zeta}. The value of $\kappa\gtrsim^\xi 1$ is chosen sufficiently small  in the proof of Lemma~\ref{lem:nearzero} below and $\xi$ is the lower bound on the distance of $\zeta_1,\zeta_2$ to zero.
We recall the notation  $\D_\xi:= \{ z: |z|< \xi\}\subset\C$.
\begin{lemma} \label{lem:nearzero} 
The solution $\bs{M}^{\mathfrak{Z}}_{\alpha}(z)$ to the MDE \eqref{eq:MDE}, with $\alpha \in \R$ and $\Im z > 0$ 
has a unique smooth extension to all $(\mathfrak{Z},\alpha,z)$, where  $\mathfrak{Z}=(\zeta_1,\zeta_2)$ with  $ \zeta_1,\zeta_2 \in{\cal{R}}\setminus\D_\xi$ and  $(\alpha,z)\in \Upsilon_\mathfrak{Z}$, provided $\kappa$ is chosen sufficiently small,  depending on  the model parameters and $\xi$. Here,  $\xi \in[ N^{-c},1]$ with some universal constant $c>0$ and $\kappa$ depends on $\xi$ at most polynomially, i.e. $\kappa\gtrsim^\xi 1$. This extension is analytic in the variables $(\re \zeta_1, \im \zeta_1,\re \zeta_2, \im \zeta_2,\alpha,z)$. 
Moreover, for  $(\alpha,z)\in \Upsilon_\mathfrak{Z}$, the following hold: 
\begin{enumerate}
\item[(1)]  $\bs{M}^{\mathfrak{Z}}_{\alpha}(z)$ satisfies the bound 
\begin{equation} \label{eq:contofM} \| \bs{M}^{\mathfrak{Z}}_0(0) - {\bs{M}}^{\mathfrak{Z}'}_\alpha(z)  \| \lesssim^\xi (|\alpha| + |z|+|\mathfrak{Z} - \mathfrak{Z}'|) \Delta_{\mathfrak{Z}}^{-1}   , \end{equation}
for any $\mathfrak{Z}' \in {\cal{R}}\times {\cal{R}}$ such that $|\mathfrak{Z} - \mathfrak{Z}'| \leq  \kappa \Delta_{\mathfrak{Z}}^{2}.$ 
\item[(2)]  the inverse of the stability operator satisfies the bound 
\begin{equation} \label{eq:stabop}  \|   ( \bs{1} -\bs{\cal{C}}_{\bs{ M}^{\mathfrak{Z}}_\alpha(z)}\bs{\cal{S}})^{-1} \| \lesssim^\xi \Delta_{\mathfrak{Z}}^{-1}. \end{equation}
\item[(3)]  when $\alpha$ and $z$ are real, the solution $ {\bs{M}}^{\mathfrak{Z}}_\alpha(z)$ is self-adjoint.
\end{enumerate}
 The implicit constants in these statements depend 
 on the model parameters.  
\end{lemma}
\begin{proof} The proof follows by an application of the implicit function theorem exactly as in Subsection~3.2 of \cite{EKR} using the new definition of $\Delta_\mathfrak{Z}$ from \eqref{definition of Delta Zeta} and the stability bound at $(\alpha, z)=(0,0)$ in \eqref{eq:stabopbound}. 
 The parameter $\mathfrak{Z}$ here plays the same role as the variable denoted by $\zeta$ in \cite{EKR}.   
In \cite{EKR} we explicitly have that $\|\bs{M}_0^\mathfrak{Z}(0) \|$ is bounded by $1$, this was used after (3.8).  In the present case we have the bound $\|\bs{M}_0^\mathfrak{Z}(0) \| \sim \max\{\norm{\frak{b}_1}_{\infty},\norm{\frak{b}_2}_{\infty}\}\lesssim^\xi 1$ from Lemma~\ref{lmm:bbound}; this modification only causes values of the  constants to change.
\end{proof}

The bound on the stability operator in \eqref{eq:stabop} (cf. its general definition in \eqref{gen stability operator}) implies stability of the MDE locally around any point $(\alpha,z)\in \Upsilon_\mathfrak{Z}$ as explained in Section~\ref{sec:Hermitian random matrices and the matrix Dyson equation}.

Bounding the support of the deterministic self-consistent density  of states $\rho_\alpha^\mathfrak{Z}$ from \eqref{def of DOS} away from zero is a key step in the proof of Theorem \ref{thm:mainthmgen} because it allows to apply the local law from \cite{EKS} in the regime away from the asymptotic spectrum. The following proposition is a consequence of Lemma~\ref{lem:nearzero} and  provides such a bound.

\begin{proposition} \label{prop:gapineig}  Let $\mathfrak{Z}=(\zeta_1,\zeta_2)$ with  $ \zeta_1,\zeta_2 \in{\cal{R}}\setminus\D_\xi$ and $\alpha$ such that  $(\alpha,0)\in \Upsilon_\mathfrak{Z}$. Here,  $\xi \in[ N^{-c},1]$ with some universal constant $c>0$. Then
\[ \dist(\supp  \rho_\alpha^\mathfrak{Z} ,0) \geq \kappa \Delta_\mathfrak{Z}^2,\]
where $\kappa\gtrsim^\xi 1$ stems from the definition of $\Upsilon_\mathfrak{Z}$ in \eqref{def of Upsilon}.
\end{proposition}
\begin{proof}
The proposition is proven just as Corollary 3.4 in \cite{EKR}. 
\end{proof}

\section{Asymptotics of resolvent products}\label{sec:MDE}
In this section we  state and prove the main technical theorem, Theorem \ref{thm:resolvents} below and afterwards use it in the proof of Theorem~\ref{thm:mainthmgen}. 
 The outline of its proof  is similar to that of Theorem 2.9 in \cite{EKR}, but the presence of correlations in the elliptic-type ensemble introduces new challenges.
We now briefly recall the ideas in the proof of Theorem 2.9 in \cite{EKR}. The first step is to introduce a Hermitian block matrix whose blocks are linear in $X$ and $X^*$ such that one of the 
blocks of its resolvent is $\alpha(X-\zeta_1)^{-1}(X^*-\ol{\zeta_2})^{-1}$. This is accomplished by $\bs{H}^{\mathfrak{Z}}_\alpha$, defined in \eqref{def of H alpha} and the $(3,1)$ 
block of its resolvent, $\bs{G}^{\mathfrak{Z}}_\alpha(0) $. 
Then we consider \eqref{eq:MDE}, the MDE whose solution approximates the resolvent, $\bs{G}^{\mathfrak{Z}}_\alpha(z) $. The difference between this solution and the resolvent is bounded by an optimal local law. Finally, we show $\bs{M}^{\mathfrak{Z}}_\alpha(z)$ is analytic in a neighborhood of $0$ and compute $\left(\d_{\alpha}\bs{M}^{\mathfrak{Z}}_\alpha(0)|_{\alpha=0}\right)_{31}$, an approximation of $(X-\zeta_1)^{-1}(X^*-\ol{\zeta_2})^{-1}$.

\begin{theorem} \label{thm:resolvents} Let $X$ satisfy Assumptions (A), (B) and (2.C-F).
There exists a (small) universal constant $c>0$  such that  
\begin{align*}  \P\Big(\sup_{\zeta_1,\zeta_2}  \Big| \tr_N \big[ (X-\zeta_1)^{-1}(X^*-\ol\zeta_2)^{-1}\big] -K(\zeta_1,\zeta_2)\Big|  \ge \frac{  N^{\epsilon}}{N^{1/2} } \Big) 
 \leq \frac{C_{\epsilon,\nu}}{N^\nu}
\end{align*}
hold for all $\eps>0$, $\nu \in \N$  
and  some constant  $C_{\epsilon,\nu}$ that may also depend on the model parameters in the Assumptions (B), and (2.C-F). Here,
the supremum is taken over all $\zeta_1,\zeta_2 \in {\cal{R}}\setminus \D_{N^{-c\1\eps}}$ with $\Delta_{(\zeta_1,\zeta_2)} \ge N^{-c\1\eps}$ and 
 the kernel $K(\zeta_1, \zeta_2)$ is from \eqref{def:Kell}.
\end{theorem}

The main novelties in this theorem, compared to Theorem 2.9 of \cite{EKR}, are twofold.
First, on the MDE level, both the set $ {\cal{R}}$ and the vector $\mathfrak {b}$  appearing in $K(\zeta_1, \zeta_2)$
 are not explicit. 
In \cite{EKR} the  self-consistent spectrum $ {\cal{R}}^c$ is simply the unit disk and  $\mathfrak {b}(\zeta_i)$   equals the constant vector with entries $ - \zeta_i^{-1}$. In the present case, both these quantities must be defined implicitly, introducing new difficulties. 
Second, on the random matrix level, in order to compare the resolvent of our elliptic-type random matrices with the solution to the MDE, we
need to  use the optimal local law from \cite{EKS}.

 The main inputs we use from  \cite{EKS} are Theorem~2.1 and Corollary~2.3. As explained before the proof of Theorem~\ref{prp:b as resolvent} %Lemma~\ref{lmm:concentration of spectrum} 
 the matrix $\bs{H}_\zeta$  does not satisfy assumption (CD) in~\cite{EKS}, and neither does  $\bs{H}^{\mathfrak{Z}}_\alpha$. However, the  more general assumptions (C) and (D)  hold for $\bs{H}^{\mathfrak{Z}}_\alpha$ and therefore Theorem 2.1 and Corollary 2.3 of \cite{EKS}  are still applicable.
 
 Theorem~2.1 from \cite{EKS} proves that there exists a universal constant $c_0$ such that for all $\eps>0$, sufficiently small, if \begin{align}\label{eq:dist}\dist(z,\supp(\rho^{\mathfrak{Z}}_\alpha))>N^{-c_0\epsilon},\end{align} then a local law (stated precisely in Proposition~\ref{prp:locallaw} below) holds at $z$ with a precision $N^{-1+\eps}$. 

Now we prove our main technical result that relies on two additional results, Proposition~\ref{prp:locallaw} and Lemma~\ref{lmm:unifformnoeigs}.  Both are stated and proved directly after the proof of Theorem \ref{thm:resolvents} by using Theorem~2.1 and Corollary~2.3 from \cite{EKS}, respectively. 

\begin{proof}[Proof of Theorem \ref{thm:resolvents}] 
For any  $\mathfrak{Z}\in {\cal{R}}^2$ and sufficiently small $\alpha$,  we have 
\begin{align}\label{threeline}
\big| \tr_N (X-\zeta_1)^{-1}&(X^*-\ol\zeta_2)^{-1} -K(\zeta_1,\zeta_2)\big|\\ \label{eq:firstest}% 
\leq &\big| \tr_N  (X-\zeta_1)^{-1}(X^*-\ol\zeta_2)^{-1} - \alpha^{-1} \tr_N ( \bs{G}^{\mathfrak{Z}}_\alpha(0) )_{31} \big| \\ \label{eq:secondest}% 
 &+\big| \alpha^{-1} \tr_N ( \bs{G}^{\mathfrak{Z}}_\alpha(0)  )_{31}  -\alpha^{-1} \tr_N ( \bs{M}^{\mathfrak{Z}}_\alpha(0) )_{31}  \big| \\  \label{eq:thirdest}% 
 &+\big| \alpha^{-1} \tr_N ( \bs{M}^{\mathfrak{Z}}_\alpha(0)  )_{31}  - K(\zeta_1,\zeta_2)\big|,
\end{align}
where  recall that  $(\bs{R})_{ij}=R_{ij}$ is the  $(i,j)$-th block of the $4N\times 4N$-block matrix $\bs{R}$. 

To estimate \eqref{eq:firstest}, we use the proof of Lemma~2.10 of \cite{EKR} to obtain the bound
\begin{equation} \label{eq:firstineq} \left|  \tr_N   (X-\zeta_1)^{-1}(X^*-\ol\zeta_2)^{-1} 
 - \frac{1}{\alpha} \tr_N (\bs{G}^{\mathfrak{Z}}_\alpha(0))_{31} \right| \Psi_{\mathfrak{Z}} \\ \lesssim \alpha^2 N^{\eps}
 \end{equation}
uniformly in $\mathfrak{Z}$, where $\Psi_{\mathfrak{Z}} := \bbm{1}(\spec (|\bs{H}^{\mathfrak{Z}}_0|) \subset [N^{-\eps/4}/2,\infty))$. 
Note that the indicator function $\Psi_{\mathfrak{Z}}$ in \eqref{eq:firstineq} can be replaced by 
$\Psi_{\mathfrak{Z},\alpha} := \bbm{1}(\spec (|\bs{H}^{\mathfrak{Z}}_\alpha| )\subset [N^{-\eps/4},\infty))$ since $\Psi_{\mathfrak{Z},\alpha}=1$ implies $\Psi_{\mathfrak{Z}}=1$ when $\abs{\alpha} \le N^{-C\1\eps}$. 
 Later in Lemma~\ref{lmm:unifformnoeigs} we will show   that inserting the characteristic function $\Psi_{ \mathfrak{Z},\alpha}$ 
 in \eqref{eq:firstineq} is affordable with the desired probability.

Next, \eqref{eq:secondest} is bounded using Proposition~\ref{prp:locallaw}   stated and proven later.  
 In particular, we get 
\begin{equation} \label{locallawoutline} \Big| \alpha^{-1}\tr_N (\bs{G}^{\mathfrak{Z}}_\alpha(0))_{31} -  \alpha^{-1} \tr_N( \bs{M}^{\mathfrak{Z}}_{\alpha}(0)  )_{31}    \Big|  \lesssim  |\alpha|^{-1} N^{-1+\epsilon}  \end{equation} 
with probability $1- O(N^{-\nu})$ for any $\nu \in \N$. 

Finally, \eqref{eq:thirdest} is estimated using Lemma~\ref{lem:nearzero} and the computation of  $\d_{\alpha}\bs{M}^{\mathfrak{Z}}_\alpha(0)|_{\alpha=0}$. The result, whose proof is given separately below, 
is
\begin{equation} \label{eq:alphaexpansion} \Big| \alpha^{-1} \tr_N  ( \bs{M}_{\alpha}^{\mathfrak{Z}}(0)  )_{31}  - K(\zeta_1,\zeta_2)  \Big|  \lesssim^{\delta}   \frac{|\alpha|}{\Delta_{\mathfrak{Z}}^3} , \end{equation} 
where $K(\zeta_1,\zeta_2)$ is the kernel from \eqref{eq:indker} and $\delta =N^{-c\1\eps}$ is the lower bound on the absolute values of $\zeta_1,\zeta_2$. 
Collecting the estimates \eqref{eq:firstineq} - \eqref{eq:alphaexpansion} and choosing $\alpha= N^{-1/2}$ gives the bound  of order $N^{-1/2+\epsilon}$  for \eqref{threeline}. 
\end{proof}

\begin{proof}[Proof of \eqref{eq:alphaexpansion}] 
From Lemma~\ref{lem:nearzero}, we have $\bs{M}_{\alpha}^{\mathfrak{Z}}$ is analytic in $\alpha$. In fact, we find
\begin{equation} 
\label{eq:derivMDE}
\left \| \partial_\alpha^k \bs{M}_{\alpha}^{\mathfrak{Z}}(z) \right\|  \lesssim^\delta \Delta_{\mathfrak{Z}}^{-2k+1}   ,\qquad k=1,2\,,
\end{equation}
where  $\zeta_1,\zeta_2 \in {\cal{R}}\setminus\D_\delta$. To see \eqref{eq:derivMDE} for $k=1$ we differentiate  \eqref{eq:MDE} with respect to   $\alpha$, solve the resulting equation for $\partial_\alpha\bs{M}_{\alpha}^{\mathfrak{Z}} $ and use \eqref{eq:stabopbound} to invert the stability operator. For $k=2$ we proceed by differentiating \eqref{eq:MDE} twice with respect to   $\alpha$, solving for $\partial_\alpha^2\bs{M}_{\alpha}^{\mathfrak{Z}} $ and again applying \eqref{eq:stabopbound} as well as \eqref{eq:derivMDE} for $k=1$. 
 Combining \eqref{eq:derivMDE} with the fact that $\left(\bs{M}_{0}^{\mathfrak{Z}}\right)_{31}=0$ (cf. \eqref{def:ABD}) gives 
\begin{equation}  \Big| \alpha^{-1} \tr_N  ( \bs{M}_{\alpha}^{\mathfrak{Z}}(0) )_{31}  -  \tr_N  \left( \d_{\alpha}|_{\alpha=0}  \bs{M}_{\alpha}^{\mathfrak{Z}} \right)_{31}  \Big|  \lesssim^\delta   \frac{|\alpha|}{\Delta_{\mathfrak Z}^3} . \end{equation} 
Following the computation in Section 3.3 from \cite{EKR},  we have  
\[ \d_{\alpha}|_{\alpha=0}  \bs{M}_{\alpha}^{\mathfrak{Z}} = (\bs{1}-\bs{\cal{C}}_{\bs{M}_0^{\mathfrak{Z}}(0)}\bs{\cal{S}})^{-1}[\bs{M}_0^{\mathfrak{Z}} (\bs E_{24} + \bs E_{42} )  \bs{M}_0^{\mathfrak{Z}} ] , \]
 recalling that $\bs{\cal{C}}_{\bs{M}}\bs{R}:= \bs{M} \bs{R}\bs{M}$ is the sandwiching operator and
$\bs E_{ij}$ is a $4N\times 4N$ block matrix with the $N\times N$ identity in the $(i,j)$ block
and otherwise zero. 
Then using that only the $(2,4)$ block is mapped into the $(3,1)$ block by $ (\bs{1}-\bs{\cal{C}}_{\bs{M}_0^{\mathfrak{Z}}(0)}\bs{\cal{S}})^{-1}$ gives
\begin{equation} \label{eq:31entry}  
\left( \d_{\alpha}|_{\alpha=0}  \bs{M}_{\alpha}^{\mathfrak{Z}} \right)_{31}  
=\frak{D}_r\,, \qquad \mbox{with}\quad  r:= (1-\frak{D}_{\frak{b}_1\ol{\frak{b}}_2}S)^{-1}\frak{b}_1\ol{\frak{b}}_2
=(\frak{D}_{\frak{b}_1\ol{\frak{b}}_2}^{-1}-S)^{-1}1
, \end{equation}
where we have used \eqref{eq:MDEat0} and  $\frak{b}_j = \frak{b}(\zeta_j)$. 

Note that  $1-\frak{D}_{\frak{b}_1\ol{\frak{b}}_2}S$ is invertible, as it is the action on the diagonal of the $(3,1)$ block of the operator $\bs{1}-\bs{\cal{C}}_{\bs{M}_0^{\mathfrak{Z}}(0)}\bs{\cal{S}}$, which was bounded in Lemma~\ref{lem:CMnorm}, and thus so is $\frak{D}_{\frak{b}_1\ol{\frak{b}}_2}^{-1}-S$.
Putting these together yields~\eqref{eq:alphaexpansion}.
\end{proof}

 Finally, we complete the proof of the two remaining technical results that were used to establish Theorem~\ref{thm:resolvents}.
First we  show  that  the resolvent, $  \bs{G}^{\mathfrak{Z}}_\alpha(0) $, is well approximated by $\bs{M}_{\alpha}^\mathfrak{Z}(0)  $, the solution to \eqref{eq:MDE}.   Then we prove 
that a gap in the self-consistent density of states, $\rho^{\mathfrak{Z}}_\alpha$, near $0$ that we established in Proposition \ref{prop:gapineig} also implies a gap in the spectrum of $\bs{H}^{\mathfrak{Z}}_\alpha$. 

\begin{proposition}\label{prp:locallaw}
There exist (small and large) constants $c_*>0$ and $C_*>0$, depending only on $\abs{\varrho}$ and $L$ in \eqref{eq:prim}, such that for any sufficiently small $\epsilon > 0$ and 
$\mathfrak{Z}=(\zeta_1, \zeta_2)$ with  $ \zeta_1,\zeta_2 \in{\cal{R}}\setminus\D_{N^{-c_* \1\eps}}$, $\Delta_{\mathfrak{Z}} \ge N^{-c_* \1\eps}$, as well as  $\alpha,z \in \C$ with $\abs{\alpha}+\abs{z}\le N^{-C_*\1\eps} $ the following high probability estimate is satisfied 
\[\P\left(  \left|  \tr_N ( \bs{M}^{\mathfrak{Z}}_\alpha(z))_{ij}  -  \tr_N  ( \bs{G}^{\mathfrak{Z}}_\alpha(z) )_{ij}    \right|  \le \frac{N^{\epsilon} }{N }  \right) \ge 1- \frac{C_{\epsilon,\nu}}{N^\nu} \]
for any $i,j=1, \dots,4$, $\nu \in \N$ and some constant $C_{\epsilon,\nu}>0$.
\end{proposition} 
\begin{proof}
Proposition~\ref{prop:gapineig} ensures a lower bound on the gap in the self-consistent spectrum around zero of the form $\dist(\supp  \rho_\alpha^\mathfrak{Z} ,z) \geq N^{-C_0 \eps}$ with some universal constant $C_0>0$. The lemma now follows directly from Theorem~2.1 in \cite{EKS}
 with a special choice of deterministic matrix $B=\bs{E}_{ji}$ in Eq. (4b) of  \cite{EKS}. 
\end{proof}

\begin{lemma} \label{lmm:unifformnoeigs} There exist (small and large)  constants $c_*>0$ and $C_*>0$, depending only on $\abs{\varrho}$ and $L$ in \eqref{eq:prim}, such that for any sufficiently small $\eps>0$  the high probability bound
\begin{align}
\label{HZalpha away from zero}
&\P\Big( \spec( |\bs{H}^{\mathfrak{Z}}_\alpha| )\subset [N^{-\eps},\infty)  :  \text{ for all } \mathfrak{Z} \in ({\cal{R}}\setminus \D_{N^{-c_*\1\eps}})^2 \text{ with } \Delta_{\mathfrak{Z}} \ge N^{-c_*\1\eps}, \abs{\alpha} \le N^{-C_*\1\eps} \Big)\\
&> 1 - \frac{C_{\epsilon,\nu}}{N^\nu}\notag
\end{align}
holds
for all $\nu \in \N$ and some $C_{\epsilon,\nu}$.
\end{lemma}
\begin{proof}
For a sufficiently small choice of $\eps$, Proposition~\ref{prop:gapineig} implies for all pairs $(\mathfrak{Z}, \alpha)$ appearing in  \eqref{HZalpha away from zero} that $\dist(\supp  \rho_\alpha^\mathfrak{Z} ,0) \geq  3N^{-\eps}$, i.e. the self-consistent spectrum is bounded away from zero. We apply Corollary 2.3 from \cite{EKS} to infer
\[
\P\Big( \spec( |\bs{H}^{\mathfrak{Z}}_\alpha| )\subset [2N^{-\eps},\infty)  \Big)> 
1 - \frac{C_{\epsilon,\nu}}{N^\nu}\,.
\]
Now we perform a stochastic continuity argument in order to bring the union over all $(\mathfrak{Z}, \alpha)$ inside the probability at the price of losing the factor $2$ from the interval $[2N^{-\eps},\infty)$.  Here we use the Lipschitz continuity of the spectrum of $\bs{H}^{\mathfrak{Z}}_\alpha$ in $(\mathfrak{Z}, \alpha)$. 
\end{proof}

\begin{proof}[Proof of   Theorem~\ref{thm:mainthmgen}]
From Theorem~\ref{prp:b as resolvent}-$(i)$ %Lemma~\ref{lmm:concentration of spectrum}  \com [Undefined tag] 
we know that $\cal{R}^\eps$ does not contain any eigenvalues of $X$ with very high probability. Thus, the path $\gamma$,
 given in  Theorem~\ref{thm:mainthmgen},  encircles all eigenvalues of $X$ exactly once. By Cauchy's theorem  we get
\[\tr_N f(X) g(X^*) = \left( \frac{1}{2\ii \pi} \right)^2 \oint_{\gamma} d \zeta_1  \oint_{\ol \gamma} d \ol{\zeta}_2 f(\zeta_1) g(\ol \zeta_2)  \tr_N (X-\zeta_1)^{-1}   (X^*-\zeta_2)^{-1} .\]
We  apply Theorem~\ref{thm:resolvents} 
 to obtain
\[\tr_N f(X) g(X^*) = \left( \frac{1}{2\ii \pi} \right)^2 \oint_{\gamma} d \zeta_1  \oint_{\ol \gamma} d \ol{\zeta}_2 f(\zeta_1) g(\ol \zeta_2)  K(\zeta_1,\zeta_2)  + \epsilon_N \]
where 
\begin{equation} \label{eq:epserror} |\epsilon_N|  \lesssim   \| f |_{\gamma}  \|_{\infty}  \| g |_{\gamma}  \|_{\infty} \frac{N^\epsilon }{N^{1/2} } \end{equation}  with probability at least $ 1 - N^{-\nu}  $ for any $\nu \in \N$  and $N$ sufficiently large. This concludes the proof of \eqref{second}. 

Finally, \eqref{first}  follows from \eqref{second}
 by setting $g(\ol{\zeta}_2) =1$ and computing the residue at $\ol \zeta_2 = \infty$. Here we used the asymptotics $\zeta_2\mathfrak {b}(\zeta_2)   \to  -1$ as
$\zeta_2\to \infty$ from Property~4 of Proposition~\ref{prp:Solution of vector equation}.
We remark that the relationship \eqref{first} can alternatively be deduced directly from the MDE by proving an analogous theorem to Theorem \ref{thm:resolvents} but with just 
one resolvent. Such a theorem follows by a similar argument and requires only using the $2 \times 2$-block MDE \eqref{2x2 MDE}.
\end{proof}

\section{Long time asymptotics}\label{sec:long}

In this section we consider the system of ODEs
\begin{equation} \label{eq:ODEx}
\partial_t u_t \,=\,  -u_t + g X u_t
\end{equation}
with initial value $u_0$ distributed uniformly on the $N$ dimensional unit sphere and coefficient $ g>0 $ chosen less than the real part of the rightmost point in the self-consistent pseudospectrum of $X$.
 The squared norm of the solution, when averaged over the initial conditions is given by
\begin{align*} \E_{u_0} \|u_t\|^2_2 =& \tr_N e^{t(gX^*-I)}e^{t(gX-I)} \\
=&   \left( \frac{1}{2\ii \pi} \right)^2 \oint_{\gamma} d \zeta_1  \oint_{\ol \gamma} d \ol{\zeta}_2 e^{t(g\zeta_1 + g\ol{\zeta_2} -2)  }  \tr_N (X-\zeta_1)^{-1}   (X^*- \ol{\zeta}_2)^{-1}  ,\end{align*}
where $\gamma$ is a curve that encloses the eigenvalues of $X$ traversed in the counter-clockwise direction and $\ol{\gamma}$ is the same curve traversed in the clockwise direction. We are interested in the large $N$ and long time regime with $t \leq N^{c}$ for some small $c >0$.

We first consider the elliptic ensemble, where the computations are done explicitly, then we consider the elliptic-type ensemble and show the large $t$ asymptotics are universal.

\subsection{Elliptic Ensemble: proof of Theorem~\ref{thm:heat}} \label{sec:longellipse}

We now consider the elliptic ensemble, satisfying Assumptions (1.C-D), and prove Theorem~\ref{thm:heat}. Recall that the correlation between the $(i,j)$ and $(j,i)$ matrix entries is given by $\varrho = |\varrho|e^{\ii \theta} $.  In this case the operators $\scr{S}$ and $\scr{T}$ from \eqref{def S and T operators} act on diagonal matrices as
\[
\scr{S}[\frak{D}_r] = \avg{r}\,, \qquad \scr{T}[\frak{D}_r ]= \varrho\2 \avg{r}\,, \qquad r \in \C^N\,, \qquad \avg{r} =\frac{1}{N}\sum_{i=1}^Nr_i.
\] 
This implies that $a_\zeta, d_\zeta$ and $b_\zeta$ from \eqref{form of 2x2 M} are all constant vectors and the MDE \eqref{2x2 MDE} reduces to a $2 \times 2$-matrix equation. In particular, $\frak{b}(\zeta)$ from Proposition~\ref{prp:Solution of vector equation} is a constant vector and can thus be interpreted as a function $\frak{b}:\cal{R} \to \C$ that satisfies the simple quadratic equation 
\begin{equation} \label{beqell} -\frak{b}(\zeta)^{-1} = \varrho\2 \frak{b}(\zeta) + \zeta\,.  \end{equation}
From \eqref{beqell} we read off the level sets of the absolute value of $\frak{b}$ as
\[
\{\zeta: \abs{\frak{b}(\zeta)} = \tau\} = \ee^{\ii \theta/2} \2\partial E(\tau^{-1}+\tau\abs{\varrho},\tau^{-1}-\tau\abs{\varrho})\,,
\]
for any $\tau < 1$, 
where $E(a,b)\subset \C $ denotes the closed domain enclosed by the ellipse $\partial E(a,b)$ with center at zero, semi-major axis $a$ along the real line and semi-minor axis $b$ along the imaginary axis. Thus, $\abs{\frak{b}(\zeta)} \to 1$ for  $\zeta \to \partial E_\varrho$, 
where $E_\varrho$ was defined in \eqref{def:ellipse}.  
Since 
$\Delta_\zeta = \min\{\abs{\frak{b}(\zeta)}^{-2}-1,1\}$ in this setting we conclude $\cal{R}=E_\varrho^c$. 
In what follows we will furthermore use that $\Re\zeta \leq \sqrt{1 + |\varrho|^2 + 2 \Re\varrho}$, for all $\zeta \in E_{\varrho}$.

\begin{proof}[Proof of Theorem ~\ref{thm:heat}] Fix $\epsilon \in (0,1/2)$ and $\varrho$ such that $0 < |\varrho| < 1$. From Theorem~\ref{thm:mainthmgen} and Remark~\ref{rem:simple} we have
\begin{equation} \label{eq:integralut} \E_{u_0} \|u_t\|_2^2 =   \left( \frac{1}{2\ii \pi} \right)^2 \oint_{\gamma} d \zeta_1  \oint_{\ol \gamma} d \ol{\zeta}_2 e^{t(g\zeta_1 + g\ol{\zeta_2} -2)  }
\frac{\frak{b}(\zeta_1) \ol{\frak{b}(\zeta_2)}}{1 - \frak{b}(\zeta_1) \ol{\frak{b}(\zeta_2)}} +\epsilon_N ,
\end{equation}
for any closed path $\gamma$ that encircles $E_\varrho$ exactly once and lies inside  the 
set $\cal{R}^\eps$ defined in Theorem~\ref{thm:mainthmgen}. It is easy to see that in the elliptic case we have
explicitly 
\[
\cal{R}^\eps =  \ee^{\ii \theta/2} \2 E(\tau_\eps^{-1}+\tau_\eps\abs{\varrho},\tau_\eps^{-1}-\tau_\eps\abs{\varrho})^c\,,\qquad 
\tau_\eps:=\frac{1}{\sqrt{1+N^{-c\1\eps}}}\,.
\]
From \eqref{eq:epserror}, in the proof of Theorem \ref{thm:mainthmgen}, we have the bound on the error term, \begin{equation} \label{epps}
| \epsilon_N| \lesssim N^{-1/2+\epsilon} \sup_{\zeta_1 \in \gamma, \zeta_2 \in \ol \gamma} |e^{t(g\zeta_1 + g\ol{\zeta_2} -2) }|\leq N^{-1/2+\epsilon} \sup_{\zeta \in \gamma} |e^{2t(g\Re\zeta -1) }|,
\end{equation}
with overwhelming probability.

We choose the  contour of integration as a dilation of the boundary of $E_\varrho$, namely
\begin{equation}\label{eq:gamma}
\gamma:=\cB{ \zeta =  \pB{1+ \frac{2}{1-\abs{\varrho}} N^{-c \epsilon}} e^{\ii \theta/2} \left( |\varrho| e^{\ii \varphi}    + e^{-\ii \varphi}  \right)| \varphi \in [0,2 \pi] }\,.
\end{equation}
An elementary calculation shows that $\gamma \in \cal{R}^\eps$ and that 
\[ \max_{\zeta \in \gamma} \Re \zeta = \left( 1 + |\varrho|^2 + 2 \Re\varrho \right)^{1/2} \pB{1+ \frac{2}{1-\abs{\varrho}} N^{-c \epsilon}}. \]
From the restriction $g \leq \left( 1 + |\varrho|^2 + 2 \Re\varrho \right)^{-1/2} $ and \eqref{epps} it follows that there exists a constant, $c_\varrho$, such that, for all $t \leq N^{c_\varrho \epsilon}$, we get the estimate $| \epsilon_N| \lesssim N^{-1/2+\epsilon}$.

We now turn to the integral  in \eqref{eq:integralut}. After making the change of variables $\zeta = w + \varrho  w^{-1}$ with $|w|>1$, and noting that from \eqref{beqell} that $w=-1/\frak{b}(\zeta) $, we have
\begin{align} \label{integralrep}&  \left(\frac{1}{2\pi\ii}  \right)^2 \oint_{\gamma'} d w_1(1- \varrho w_1^{-2})  \oint_{\ol{\gamma}'}  d \ol{w}_2(1- \ol{\varrho}\ol{w}_2^{-2})
 e^{t(g ( w_1+  \varrho /w_1) +g( \ol{w_2}+ \ol{\varrho} /\ol{w_2})-2)} (w_1 \ol{w_2}-1)^{-1}, \end{align}%  \\
where $\gamma' $ is the image of $\gamma$ under the change the variables, traversed clockwise, and $\ol{\gamma}'$ is the same curve, traversed counterclockwise. It is easy to see $\gamma'$ is close to the unit circle.

To compute the integral in $w_2$ we expand the exponential as the generating function for the modified Bessel functions, valid for all $u \not = 0$, \cite[Eq. 9.6.33]{Abramowitz},
\begin{equation} \label{eq:Besselgenfun} e^{\frac{x}{2}\left(u + \frac{1}{u}\right)} = \sum_{k=-\infty}^{\infty} I_k(x) u^{k} \end{equation}
recall that  $I_k$ is the $k$-th  modified Bessel function of  the first kind. Using that $|w_1|>1$ and $|w_2|>1$, we get
\begin{align*} 
&  \left(\frac{1}{2\pi\ii}  \right) \oint_{\ol{\gamma}'}  d \ol{w}_2(1- \ol{\varrho}\ol{w}_2^{-2}) e^{tg( \ol{w_2}+ \ol{\varrho} /\ol{w_2}) } (w_1 \ol{w_2}-1)^{-1}  \\
&=  \sum_{j=1}^{\infty} \sum_{k = -\infty}^{ \infty }   \left(\frac{1}{2\pi\ii}  \right) \oint_{\ol{\gamma}'}  d \ol{w}_2(1- \ol{\varrho}\ol{w}_2^{-2}) \left( \frac{\ol{w}_2 }{ \sqrt{\ol{\varrho}} }\right)^k   I_k(2 \sqrt{\ol{\varrho} } t g  )   \left(\frac{1}{w_1 \ol{w}_2}\right)^{j}.
\end{align*}% 
The integral is zero unless $k = j-1$ or $k = j+1$ leading to
\begin{align*} 
& =  \sum_{j=1}^{\infty} \left(  (\sqrt{\ol{\varrho}})^{1-j} I_{j-1}(2\sqrt{\ol{\varrho}} tg)   -   (\sqrt{\ol{\varrho}})^{1-j} I_{j+1}(2\sqrt{\ol{\varrho}} tg)\right) w_1^{-j}.
\end{align*}% 
Substituting this we continue from \eqref{integralrep} as   
\begin{align*} %
 \eqref{integralrep} = \left( \frac{1}{2\ii \pi} \right) & e^{-2t} \oint_{\gamma'} d w_1(1-\varrho w_1^{-2})   e^{gt ( w_1+  \varrho /w_1)}\\
&  \times \sum_{j=1}^{\infty}   \left(  (\sqrt{\ol{\varrho}})^{1-j} I_{j-1}(2\sqrt{\ol{\varrho}} tg)   -   (\sqrt{\ol{\varrho}})^{1-j} I_{j+1}(2\sqrt{\ol{\varrho}} tg)\right) w_1^{-j} . %
\end{align*}

Then computing the integral over $w_1$  (by essentially repeating the above computation), using the relationship $I_j(\ol{z}) =\ol{I_j(z)}$, and the three-term relationship for $I_j$, \eqref{integralrep} simplifies to
\begin{equation} %
 e^{-2t}  \sum_{j=1}^{\infty}  \left|  (\sqrt{\varrho})^{1-j} I_{j-1}(2\sqrt{\varrho} tg)   -   (\sqrt{\varrho})^{1-j} I_{j+1}(2\sqrt{\varrho} tg)\right|^2 
 = e^{-2t}  \sum_{j=1}^{\infty}  |\varrho|^{-j} \left| \frac{j}{ tg} I_{j}(2\sqrt{\varrho} tg)  \right|^2,  
\label{eq:soltodiffeq}
\end{equation}
as desired,  proving \eqref{besseldecay}.

Now we turn to the asymptotic bounds \eqref{asymp1} and \eqref{2side1}. 
 In order to  extract the leading order behavior of  \eqref{eq:soltodiffeq} for $t$ large,
 we add the negative index terms to the infinite series, which allows us to apply well known identities for the Bessel function. We then show these additional terms are much smaller for large $t$. Thus we compute  
\begin{align} \label{eq:simpbesssol}
 \sum_{j=-\infty}^{\infty}  |\varrho|^{-j} &  \Bigg| \frac{j}{ tg} I_{j}(2\sqrt{\varrho} tg)  \Bigg| ^2
  =\sum_{j=-\infty}^{\infty}   |\varrho|^{-j}   \left|  \sqrt{\varrho} I_{j-1}(2\sqrt{\varrho} tg)   -   \sqrt{\varrho} I_{j+1}(2\sqrt{\varrho} tg)\right| ^2 \nonumber \\
=& \sum_{j=-\infty}^{\infty}  |\varrho|^{1-j}  I_{j-1}(2\sqrt{\varrho} tg) I_{j-1}(2\sqrt{\ol{\varrho}} tg) + |\varrho|^{1-j}  I_{j+1}(2\sqrt{\varrho} tg)I_{j+1}(2\sqrt{\ol{\varrho}} tg)  \nonumber  \\
&   -2 |\varrho|^{1-j}  \left(  I_{j-1}(2\sqrt{\ol{\varrho}} tg)I_{j+1}(2\sqrt{\varrho} tg)  +   I_{j-1}(2\sqrt{\varrho} tg)I_{j+1}(2\sqrt{\ol{\varrho}} tg)  \right). 
\end{align}
To simplify this expression we apply Graf's addition theorem; see, for instance \cite[Eq. 9.1.79]{Abramowitz}, at a complex angle. For the reader's convenience we record the identity in the following lemma.
\begin{lemma} Let $c \in \C$ be a non-zero constant, $x,y \in \C$ and $\nu$ and $n$ be integers, then   
\[  \sum_{n=- \infty}^\infty c^n   I_{n+\nu}(x) I_{n}(y) = \left(  \frac{ x   + y c^{-1} }{ x   +  y c } \right)^{\nu/2}  I_{\nu}\Big( \sqrt{x^2 + y^2 + xy(c +c^{-1}) } \Big) . \]
\end{lemma}

\begin{proof}
We will prove the following identity, the lemma follows by equating the coefficients in $t$
\[  \sum_{n=- \infty}^\infty  \sum_{m=- \infty}^\infty  I_{m}(x) t^m  I_{n}(y) \left(  \frac{c}{t} \right)^n = \sum_{\nu=-\infty}^{\infty}I_{\nu}( \sqrt{x^2 + y^2 + xy(c +c^{-1}) } ) \left(  \frac{ x   + y c^{-1} }{ x   +  y c } \right)^{\nu/2}  t^{\nu}. \]
From \eqref{eq:Besselgenfun} the left side equals
\begin{align*} 
&\exp\left(\frac{x}{2}\left(t + \frac{1}{t} \right) \right) \exp\left(\frac{y}{2} \left(\frac{c}{t} + \frac{t}{c}\right)  \right)\\
&= \exp \left(\frac{1}{2} \left( (   x +  y c^{-1}       )( x +  y c        )   \right)^{1/2}  \left(t \left(  \frac{  x +  y c^{-1}       }{   x +  y c       }    \right)^{1/2}  + \frac{1}{t} \left( \frac{   x +  y c^{-1}       }{  x +  y c         }    \right)^{-1/2}   \right)    \right)\\
&=    \sum_{\nu=-\infty}^{\infty}I_{\nu}\left( \sqrt{x^2 + y^2 + xy(c +c^{-1}) } \right) \left(  \frac{ x   + y  c^{-1} }{ x   +  y c } \right)^{\nu/2}  t^{\nu}.
\end{align*}

\end{proof}
After applying this identity, \eqref{eq:simpbesssol} simplifies to 
\[  (1+ |\varrho|^2) I_0(2t  g \sqrt{2\Re{\varrho} +|\varrho|^2 + 1 }  )  - 2\Re\left( \frac{\varrho + |\varrho|^2 }{ \varrho + 1 } \right) I_2(2tg \sqrt{2\Re{\varrho} +|\varrho|^2 + 1 }  )  ,  \]
then using the well known asymptotics $I_m(x) \sim  (2\pi x)^{-1/2} e^x$ as $x\to \infty$ (for $m=0,2$), we have 
\begin{equation}\label{eq:Iasy}  \sum_{j=-\infty}^{\infty}  |\varrho|^{-j} \left| \frac{j}{ tg} I_{j}(2\sqrt{\varrho} tg)  \right|^2\sim 
\frac{e^{2t  g \sqrt{2\Re{\varrho} +|\varrho|^2 + 1 } }}{{\sqrt{2\pi}\sqrt{ 2t  g \sqrt{2\Re{\varrho} +|\varrho|^2 + 1 }  }}}\left((1+|\varrho|^{2})  - 2\Re\left( \frac{\varrho + |\varrho|^2 }{ \varrho + 1 } \right) \right)  .
 \end{equation}
for large $t$. 
We now show that the large $t$ behavior of \eqref{eq:soltodiffeq} and \eqref{eq:Iasy} are the same by bounding the negative index terms. We begin with an elementary inequality that follows from the following integral representation, valid for any integer $k$, \cite[Eq. 9.6.19]{Abramowitz},
\[ I_k(z) = \frac{1}{\pi} \int_0^\pi e^{z\cos{\theta}} \cos(k\theta) ~d \theta. \]
Taking the absolute value we have
\[ | I_k(z) | 
\leq  \frac{1}{\pi}  \int_0^\pi e^{|\Re z| \cos{\theta}}~d \theta  = I_0(|\Re z|) .  \]  

Returning to the negatively indexed terms in \eqref{eq:Iasy} we apply the above inequality and the trivial inequality $ I_0(x) <  e^{x} $ for $x >0$ to get:
\begin{align*}
\sum_{j=-\infty}^{-1}  |\varrho|^{-j}  \left| \frac{j}{ tg} I_{j}(2\sqrt{\varrho} tg)  \right| ^2 
& \leq I_{0}(|\Re(2\sqrt{\varrho} tg)|)^2 \frac{1}{ t^2 g^2}   \sum_{j=1}^{\infty}  |\varrho|^{j}  j^2  \leq  e^{4 |\Re(\sqrt{\varrho} tg)|} \frac{1}{ t^2 g^2}  \frac{|\varrho|(1+|\varrho|)}{(1-|\varrho|)^3}.
\end{align*}
For large values of $t$ this expression is much smaller than the right side of \eqref{eq:Iasy}, as \\ $4 |\Re\sqrt{\varrho} | <2 \sqrt{2\Re{\varrho} +|\varrho|^2 + 1 }$ when $|\varrho|< 1$. 
\end{proof}
\subsection{Elliptic-type matrices: Proof of Theorem \ref{thm:heatelliptictype}} \label{sec:ellcorr}

In this section we consider \eqref{eq:ODEx}, for $X$ of elliptic-type, satisfying the assumptions of Theorem \ref{thm:heatelliptictype}. In particular we will always assume $t_{ij} \geq 0$  
in the following. We  begin by establishing geometric bounds on the set  $\cal{R}^c$, which will allow us to choose a good contour of integration, $\gamma$. Then we turn to computing the leading order of the integral, thus, proving Theorem \ref{thm:heatelliptictype}.  

Recall that $\zeta^*=\max_{\zeta\in \cal{R}^c}\Re \zeta$. The following  proposition shows  $\cal{R}^c$ is located in the disk of radius $\zeta^*$, centered at the origin.
\begin{proposition}\label{prp:rightmost} For elliptic-type ensembles satisfying  
Assumptions (A), (B), (2.C-F) with $t_{ij}\ge 0$ for all $i, j$,  the
self-consistent pseudospectrum $\cal{R}^c$ is symmetric about the real axis and
is contained in the ball of radius $\zeta^*$:
\bels{contain}{
\cal{R}^c\subset \{\zeta\in \C\; : \; |\zeta|\le \zeta^*\},% \quad \mbox{with}\quad
%\zeta^*:=\max_{\zeta\in \cal{R}^c}\Re \zeta.
}
In particular, the rightmost point of $\cal{R}^c$ is unique and lies on the real axis.
\end{proposition}

In the proof of  this proposition, we will make use rightmost point of $\cal{R}^c$ along the real axis:
\bels{new definition of zeta*}{
\wt \zeta^*:= \max_{\zeta \in \cal{R}^c\cap\R}\zeta.
}  
This maximum here (and also in the definition of $\zeta^*$)
exists because $\cal{R}$ is open by Proposition~\ref{prp:Solution of vector equation}. 
Note that part of the goal of  Proposition~\ref{prp:rightmost}  is to show that  $\wt  \zeta^* = \zeta^* $. In \cite{EKR} the $T=0$ case is considered, and it is shown that the self-consistent pseudospectrum is in fact equal to the disk with radius $\zeta^*$. 

\begin{proof}[Proof of Proposition~\ref{prp:rightmost}]

The non-negativity of $T$ along with \eqref{definition of b} implies that  $\mathfrak{b}(\ol{\zeta}) =\ol{\mathfrak{b}(\zeta)}$, so $\mathfrak{b(\zeta)}$ is real for $\zeta\in \cal{R}\cap \R$, 
 and therefore the set $\cal{R}$ is symmetric across the real axis. 

The remainder of the proposition follows from the following lemma on the magnitude of $\mathfrak{b}(\zeta)$ for $\zeta$ with magnitude greater than $\wt \zeta^*$, whose proof we postpone until the conclusion of the current proof.

\begin{lemma} \label{lem:bboundcirc}
 Let  $\widehat \zeta \in \R$ such that $\wt \zeta^* <\widehat \zeta $.  Then $\wh{\zeta} \in \cal{R}$ and for all $\zeta$ such that $|\zeta| > \widehat \zeta$ we have $\zeta \in \cal{R}$ and  $|\mathfrak{b}(\widehat \zeta)| > |\mathfrak{b}( \zeta)| $. 
\end{lemma}

Indeed, this lemma implies that for any $\zeta \in \C$ with $\abs{\zeta} > \zeta^* $, $\mathfrak{r}(\mathfrak{D}_{|\mathfrak{b}( \zeta)|^2} S) < \mathfrak{r}(\mathfrak{D}_{\mathfrak{b}( \wt \zeta^*)^2} S) =1$ and thus $\zeta$ lies inside $\cal{R}$. In particular, we see that $\wt \zeta^* = \zeta^*$, completing the proof of Proposition~\ref{prp:rightmost}. 
\end{proof}

\begin{proof}[Proof of Lemma~\ref{lem:bboundcirc}] First we have $\widehat \zeta \in \cal{R}$ by definition of $\wt \zeta^*$ and that $\wt \zeta^* <\widehat \zeta $. 
 To control the absolute value of $\mathfrak{b}(\zeta)$, we take the absolute value of \eqref{definition of b} and consider the resulting equation. Since $T$ is a non-negative matrix, we have $\mathfrak{b}(\widehat \zeta)<0$ (cf. \eqref{bneg} below) and therefore
\begin{equation} \label{eq:absbeq}
 |\mathfrak{b}(\widehat \zeta)|^{-1} =  -T|\mathfrak{b}(\widehat \zeta)| + \widehat \zeta . \end{equation}
Now let $\zeta \in \C$ with $\abs{\zeta}>\wh{\zeta}$. We show that $\zeta \not \in \partial \cal{R}$, i.e. the boundary of $\cal{R}$ lies inside $\{\zeta : \abs{\zeta} \le \wh{\zeta}\}$. For this purpose let  $\zeta  \in \partial \cal{R}\cup\cal{R}$. By Lemma~\ref{lmm:bbound} we can extend $\frak{b}$ holomorphically to a neighborhood of $\zeta$. This extension still satisfies \eqref{definition of b} and $\mathfrak{r}(\mathfrak{D}_{\mathfrak{b}( \zeta)^2} S) \le1$ by continuity of $\Delta_{\zeta}$ from \eqref{def of Delta zeta}. Again taking absolute value in \eqref{definition of b} we get 
\[ |\mathfrak{b}(\zeta)|^{-1} =  |T\mathfrak{b}(\zeta) + \zeta| \geq  -T|\mathfrak{b}( \zeta)| + | \zeta| .\] 
Letting $\frak{z} :=  T|\mathfrak{b}(\zeta)| + |\mathfrak{b}( \zeta)|^{-1} \geq |\zeta| > \widehat \zeta$
we also have
\begin{equation} \label{abs b equation} |\mathfrak{b}( \zeta)|^{-1} = -T|\mathfrak{b}(\zeta)| + \frak{z} . \end{equation}
So that $|\mathfrak{b}(\zeta)|$ satisfies \eqref{eq:absbeq} when $\widehat \zeta$ is replaced by the vector $\frak{z} \in \R_+^N$.

To verify the bound $|\mathfrak{b}(\widehat \zeta)| > |\mathfrak{b}(\zeta)|$ we will use the following technical lemma, which we prove after the conclusion of the current proof, to show that the solution to \eqref{abs b equation} is decreasing in $\frak{z}$. 

\begin{lemma}\label{lem:decreasing}
Let $\frak{z}_1,\frak{z}_2 \in \R_+^{N}$ be such that $\wt \zeta^* < \frak{z}_1 < \frak{z}_2$ entry-wise. Assume
that there exist vectors $\beta_i\in \R_+^N$, $i=1,2$, such that $ \beta_i^{-1} =  -T \beta_i + \frak{z}_i $ and $ \mathfrak{r}(\mathfrak{D}_{\beta_i^2} S) \le 1$, then $\beta_2 < \beta_1$ entry-wise. 
\end{lemma}
Choosing $\frak{z}_1 = \widehat \zeta, \frak{z}_2 = \zeta, \beta_1 = |\mathfrak{b}(\widehat \zeta)|, \beta_2 = |\mathfrak{b}(\zeta)|$, we conclude $|\mathfrak{b}(\widehat \zeta)| > |\mathfrak{b}(\zeta)|$. Since $\Delta_{\wh{\zeta}}>0$ by  \eqref{def of Delta zeta} and $\wh{\zeta} \in \cal{R}$ we see that $1>\mathfrak{r}(\mathfrak{D}_{\mathfrak{b}( \wh{\zeta})^2} S)>\mathfrak{r}(\mathfrak{D}_{\mathfrak{b}( \zeta)^2} S) $ and thus $\zeta \in \cal{R}$ because $\mathfrak{b}( \zeta)$ solves \eqref{definition of b} and \eqref{sidecond} also holds, completing the proof of Lemma~\ref{lem:bboundcirc}. 
\end{proof}

We now prove the technical lemma.
\begin{proof}[Proof of Lemma~\ref{lem:decreasing}]
Taking the difference of $  \beta_i^{-1} =  -T \beta_i + \frak{z}_i  $ at the two points yields
\[ \beta_2^{-1} - \beta_1^{-1} =  T(\beta_1-\beta_2) + \frak{z}_2 - \frak{z}_1 .   \]
Rearranging and solving for $\beta_1-\beta_2$ leads to
\[ \beta_1-\beta_2 = ( 1  - \mathfrak{D}_{\beta_1\beta_2}T)^{-1} (\beta_1\beta_2 ( \frak{z}_2 - \frak{z}_1  )  )\,,    \] 
which we will see is entry-wise non-negative once we verify that $( 1  - \mathfrak{D}_{\beta_1\beta_2}T)^{-1} $ is a matrix with  non-negative entries. To show this it suffices to bound the spectral radius of the non-negative matrix $\mathfrak{D}_{\beta_1\beta_2}T$ by 1.
Using the spectral radius inequality \eqref{eq:hadspecrad}, we have
\[ \frak{r}(\mathfrak{D}_{\beta_1\beta_2}T )^2 \leq \frak{r}(\mathfrak{D}_{\beta_1^2} T ) \frak{r}(\mathfrak{D}_{\beta_2^2} T) .\]
We see that each of the terms on the right side are smaller than $1$ exactly as in the proof of \eqref{eq:gennonherm}, using as an input the assumption $ \mathfrak{r}(\mathfrak{D}_{\beta_i^2} S) \le 1$. 
\end{proof}

We now establish lower bounds of $\zeta^*$ by relating $\mathfrak{b}(\zeta)$ to the solution to a quadratic  vector equation. Since $t_{ij}\ge 0$ the defining equation \eqref{definition of b} becomes a quadratic  vector equation  of the type studied in \cite{AjankiQVE} at spectral parameter $\zeta$. 
  As a result (cf. Theorem~2.1 in \cite{AjankiQVE}), the solution $\mathfrak{b}(\zeta)$ has a Stieltjes transform representation, similar to \eqref{St}, i.e.
 \begin{equation}\label{St1}
    \mathfrak{b}_i(\zeta) = \int_\R \frac{ \mu_i( {\rm d} x)}{x-\zeta}, \qquad \im \zeta>0,
 \end{equation}
 for some probability measures $\{ \mu_i \}_{i=1}^N$ on the real line whose supports lie in  
 an $N$-independent compact set. 
  The measures $\mu_i$ are symmetric around the origin because $\mathfrak{b}(-\zeta)= - \mathfrak{b}(\zeta)$ by the symmetry of the equation \eqref{definition of b} and the uniqueness of its solution. 
 The relation \eqref{St1} also implies that $\frak{b}$ can be analytically extended to $\C\setminus \supp \mu$, where $\supp \mu:=\cup_i  \supp \mu_i$.

The  following lemma shows that  $\zeta^*$ lies to the right of $\supp \mu$, with an effective lower bound on the distance.
 \begin{lemma} \label{lmm:zeta* away from zero} $ \dist(\zeta^*, \supp \mu)= \zeta^*-\max \supp \mu \gtrsim 1$
 \end{lemma}

  \begin{proof}[Proof of Lemma \ref{lmm:zeta* away from zero}]
 We first show a non-effective version of the lemma, namely that $\zeta^*$ lies to the right of $\supp \mu$, i.e.  $\zeta^* \ge \max \supp \mu$.  For this we view 
the extraspectral Dyson equation  \eqref{definition of b} 
 as a special case of the general Dyson equations studied in \cite{AEKband} and will apply Lemma~D.1 from \cite{AEKband} with the self-energy operator $r \mapsto T\2r$ and $a=0$. In the notation from \cite{AEKband} the Dyson equation is formulated on the commutative von Neumann algebra $\cal{A}=\C^N$ with entry-wise multiplication. 

 Let $\zeta \ge  \zeta^*$. Since $\cal{R}$ is open and $\frak{b}$ is analytic by definition 
 in Proposition~\ref{prp:Solution of vector equation}, 
we can apply the implication (iv) $\Rightarrow$ (v)   
of Lemma~D.1 from \cite{AEKband} 
and see that $\zeta \in \C \setminus \supp \mu$, and thus $\wt  \zeta^* \ge \max \supp \mu$. 
 Furthermore, $\mathfrak{b}_i(\zeta) \not = 0$ from \eqref{definition of b} and $\frak{b}(\zeta)<0$ for sufficiently large $\zeta>0$ by Property~4 in Proposition~\ref{prp:Solution of vector equation}.
Therefore,
   \begin{equation}\label{bneg}\mathfrak{b}(\zeta) < 0\,, \qquad \zeta > \zeta^*\,. \end{equation}

We will now show that $\zeta^* \gtrsim 1$. 
Since $\mu_i$ is a symmetric probability measure, we have $\supp \mu \subset [- \zeta^*, \zeta^*]$. By the representation \eqref{St1} this implies 
 \[
\abs{ \frak{b}_i(\zeta) + \zeta^{-1}}\le \int_\R \mu_i( {\rm d} x)\frac{ \abs{x}}{\zeta\abs{x-\zeta}} \le \frac{2\2 \zeta^*}{\zeta^2}\,, \qquad \zeta \ge 2 \2 \wt \zeta^*\,,
 \]
 and, thus, the lower bound $\abs{\frak{b}(\zeta)}\ge \frac{1}{2\1\zeta}$ for all $\zeta \ge 4 \2 \wt  \zeta^*$. Since $\zeta \in \cal{R}$, we have that $\frak{r}(\frak{D}_{4\abs{\zeta}^2}^{-1}S)\le\frak{r}(\frak{D}_{\abs{\frak{b}(\zeta)}^2}S)<1$. 
  Since $\frak{r}(\frak{D}_{4\abs{\zeta}^2}^{-1}S) = \frak{r}(S)/(4|\zeta|^2)$,  
 this implies the lower bound $\zeta \gtrsim 1$, and hence $ \wt  \zeta^* \gtrsim 1$. 
 
Finally, we prove the effective bound  on $\dist(\zeta^*, \supp \mu)$. First we note that $\abs{\frak{b}(\zeta_\delta)}\lesssim 1$ with $\zeta_\delta:=\zeta^*+\delta$ for any $\delta>0$ due to Lemma~\ref{lmm:bbound} and the lower bound on $\zeta^*$. Since $\zeta_\delta > \max \supp \mu$ the $\R^N$-valued function $\frak{b}$ is analytic around this point. Furthermore, $\zeta_\delta \in \cal{R}$ implies $\frak{r}(\frak{D}_{\abs{\frak{b}(\zeta_\delta)}^2}S)<1$ and, thus, $\frak{r}(\frak{D}_{\abs{\frak{b}(\zeta_\delta)}^2}T)<\abs{\varrho}$. Therefore, $\norm{(1-\frak{D}_{\abs{\frak{b}(\zeta_\delta)}^2}T)^{-1}} \le \frac{1}{1-\abs{\varrho}}$. From the implication (iii) $\Rightarrow$ (v) in  Lemma~D.1 of \cite{AEKband} we now conclude that $\dist(\zeta_\delta, \supp \mu)\gtrsim 1$. Since $\delta>0$ 
 was arbitrary, the statement of the lemma follows. 
 \end{proof}

Since Theorem \ref{thm:mainthmgen} is stated for $\zeta \in \cal{R}^{\eps}$  we will need to control the distance between these sets and $\cal{R}^c$. The following proposition will give an effective lower bound on this distance.

\begin{proposition}\label{prp:Delta ok on path}
 \[ \min_{\zeta: \abs{\zeta} \ge \zeta^*+\eps}\Delta_\zeta\ge C \eps\] for all $\eps \in (0,1/C)$, where $C>0$ is a constant, depending on 
the model parameters.  
\end{proposition}
\begin{proof}
Let  $\zeta \in \R \cap \cal{R}$ such that $\zeta > \zeta^*$.  Since $\mathfrak{b}(\zeta)= \ol{\mathfrak{b}(\zeta)}$, the matrix $\mathfrak{D}_{\mathfrak{b}(\zeta)^2} S$ has positive entries and, thus, its spectral radius coincides with its Perron-Frobenius eigenvalue.  
We denote by prime $'$ the derivative $\frac{d}{d\zeta}$
 and use first order perturbation theory to get  
\[ \mathfrak{r}(\mathfrak{D}_{\mathfrak{b}(\zeta)^2} S)' = \langle v_l , 2 \mathfrak{D}_{\mathfrak{b}(\zeta) \mathfrak{b}(\zeta)'} S   v_r \rangle , \] 
where $v_l,v_r$ and the left and right Perron-Frobenius eigenvectors of $\mathfrak{D}_{\mathfrak{b}(\zeta)^2} S$, respectively, normalized so that $\langle v_l,v_r \rangle = 1$. 
The bound for large values of $\abs{\zeta}$ is clear due to the behavior of $\frak{b}(\zeta)$ as $\zeta \to \infty$  in Proposition~\ref{prp:Solution of vector equation}. Thus we consider $\abs{\zeta}\lesssim 1$. By Lemma~\ref{lmm:zeta* away from zero} we also have $\zeta>\zeta^* \gtrsim 1$. 
From \eqref{bbound}  we have $-\mathfrak{b}(\zeta)=\abs{\mathfrak{b}(\zeta)}\sim 1$. 
Thus, the uniform primitivity assumption \eqref{eq:prim} also holds for $\mathfrak{D}_{\mathfrak{b}(\zeta)^2} S$. By standard Perron-Frobenius theory we therefore conclude $v_l,v_r \sim 1$.   Differentiating \eqref{definition of b} gives ${\mathfrak{b}(\zeta)'=(\mathfrak{D}_{\mathfrak{b}(\zeta)}^{-2}-T)^{-1}1}$. We note the inverse is well defined by \eqref{eq:gennonherm}.  Expanding the inverse yields 
\[\mathfrak{b}(\zeta)' = \sum_{k=0}^\infty (\mathfrak{D}_{\mathfrak{b}(\zeta)}^2 T)^k \mathfrak{b}(\zeta)^2 >  \mathfrak{b}(\zeta)^2 \sim 1 . \]
where we have used that $t_{ij}\geq 0$. Combining all these estimates, we see  $-\mathfrak{r}(\mathfrak{D}_{\mathfrak{b}(\zeta)^2} S)'  \gtrsim 1 $. 
The conclusion of the proof follows by noting that $\Delta_\zeta\ge \Delta_{\abs{\zeta}}$ from Lemma \ref{lem:bboundcirc} and that $\partial_\zeta\Delta_\zeta \gtrsim 1$ at $\zeta >\zeta^*$.

\end{proof}

%Additionally, from the previous lemmas, we also get the following effective bound on $\Delta_\zeta$:
%
%\begin{proposition}\label{prp:Delta ok on path}
% \[ \min_{\zeta: \abs{\zeta} \ge \zeta^*+\eps}\Delta_\zeta\ge C \eps\] for all $\eps \in (0,1/C)$, where $C>0$ is a constant, depending on 
%the model parameters.  
%\end{proposition}
%\begin{proof}
% \end{proof}

With the bounds on the location of $\cal{R}$ at hand, we now turn to the proof of Theorem \ref{thm:heatelliptictype}.

\begin{proof}[Proof of Theorem \ref{thm:heatelliptictype}]
As in Section~\ref{sec:longellipse} we have the squared norm of the solution to \eqref{eq:ODEx}, when averaged over the initial conditions  given by
\begin{align*} \E_{u_0} \|u_t\|_2^2 &= \tr_N e^{t(gX^*-I)}e^{t(gX-I)} \\
&=   \left( \frac{1}{2\ii \pi} \right)^2 \oint_{\gamma} d \zeta_1  \oint_{\ol \gamma} d \ol{\zeta}_2 e^{t(g\zeta_1 + g\ol{\zeta_2} -2)  }  \tr_N (X-\zeta_1)^{-1}   (X^*-\zeta_2)^{-1}  .\end{align*} 
We take $\gamma$ to be the curve $\zeta= \zeta^* (1+CN^{-c \eps})\ee^{\ii \varphi}$, $\varphi \in [0,2\pi]$ and $\ol{\gamma}$ is the same curve traversed in the clockwise direction. By Proposition~\ref{prp:Delta ok on path}  we have the lower bound $\Delta_{\zeta} \ge N^{-c\eps}$   for any $\zeta\in\gamma$  if $C>0$ is chosen large enough.

Applying Theorem~\ref{thm:mainthmgen} we have
\[ \E_{u_0} \|u_t\|_2^2 =  \left( \frac{1}{2\ii \pi} \right)^2 \oint_{\gamma} d \zeta_1  \oint_{\ol \gamma} d \ol{\zeta}_2 e^{t(g\zeta_1 + g\ol{\zeta_2} -2)  }  K(\zeta_1,\zeta_2) + \epsilon_N  \]
where $\epsilon_N \leq N^{-1/2+\epsilon}$ for any $\epsilon>0$ and $t \leq N^{c\epsilon}$ with high probability as in \eqref{epps}.
Computation of this integral relies on understanding the singularity of  $K(\zeta_1,\zeta_2)$ for $\zeta_1,\zeta_2$ near $\zeta^*$, which  is determined by the behavior of  the isolated eigenvalue of  $L(\zeta_1,\ol{\zeta}_2) := \frak{D}_{\frak{b}_1\ol{\frak{b}}_2}^{-1}-S $ near $0$, where $\frak{b}_i=\frak{b}(\zeta_i)$. 

In this proof we holomorphically extend  $\frak{b}(\zeta)$ to a neighborhood of $\zeta^*$. This is possible due to Lemma~\ref{lmm:zeta* away from zero} and the Stieltjes transform representation \eqref{St1}. We still denote the extension by $\frak{b}(\zeta)$. Since $\zeta^*$ is the unique point in $\cal{R}^c$ with maximal real part (cf. Proposition~\ref{prp:rightmost}) 
 we also have $\zeta^* \in \partial \cal{R}$. By  Proposition~\ref{prp:Solution of vector equation} we conclude that $\Delta_{\zeta^*}=0$. Otherwise the holomorphic extension of $\frak{b}$ around $\zeta^*$ would ensure solvability of \eqref{definition of b} with side condition \eqref{sidecond}  in a neighborhood of $\zeta^*$ and, thus, $\zeta^*$ would be in the interior of $\cal{R}$. Therefore, zero is an eigenvalue of $L(\zeta^*,\zeta^*) = \frak{D}_{\abs{\frak{b}(\zeta^*)}^2}^{-1}(1-\frak{D}_{\abs{\frak{b}(\zeta^*)}^2}S)$. 

Furthermore, due to Lemma~\ref{lemm:quanprim} applied to $R=\frak{D}_{\abs{\frak{b}(\zeta^*)}^2}S$ the eigenvalue at zero is isolated by a spectral gap  
 from the rest of the spectrum of $L(\zeta^*,\zeta^*)$.  In particular, there exists a contour separating the   eigenvalue at $0$
   from the rest of the spectrum, with the resolvent \eqref{eq:R resolvent bound}  effectively controlled along it. Using analytic perturbation theory we conclude that   $L(\zeta_1,\ol{\zeta}_2)$ still has an isolated eigenvalue $\lambda=\lambda(\zeta_1,\ol{\zeta}_2)$ that is closest to zero for $\zeta_1$ and $\zeta_2$ sufficiently close to $\zeta^*$ and satisfies 
\[
\lambda(\zeta^*,\zeta^*)=0\,.
\]
The eigenvalue  $\lambda$ and all derived quantities in the following depend analytically on $\zeta_1$ and $\ol{\zeta}_2$.    Let  $P(\zeta_1, \ol{\zeta}_2):= \scalar{v_l}{\1\cdot\1} v_r$ the corresponding rank $1$ projection, where $v_r=v_r(\zeta_1, \ol{\zeta}_2)$ and $v_l=v_l(\zeta_1, \ol{\zeta}_2)$ are the right and left eigenvectors of $L(\zeta_1, \ol{\zeta}_2)$ with respect to eigenvalue $\lambda$, respectively, normalized so that $\langle v_l,v_r \rangle =1$. We also set $Q(\zeta_1, \ol{\zeta}_2):=1-P(\zeta_1, \ol{\zeta}_2)$. For ease of readability we often drop the dependence on the argument from $v_l$, $v_r$, $P$, and $Q$. In \eqref{eq:Defelltypecoeff}, \eqref{eq:elltypecoeff} and in Lemma~\ref{lemma:derivativescorr}, below, these quantities are evaluated at $\zeta^*$. To study the singularity we consider the equation $\lambda(\zeta_1,\ol{z}_2(\zeta_1))=0$ with an analytic function  $\ol{z}_2$  around its solution $\lambda(\zeta^*,\zeta^*)=0$.  Analyticity of $\ol{z}_2$ is a consequence of the analyticity of $\lambda$ in both its arguments. 
In what follows we use $\partial_1= \partial_{\zeta_1}$ and  $\ol{\partial}_2= \partial_{\ol{\zeta_2}}$. 
Before proceeding we  define the coefficient  $A(S,T)$ that appears in \eqref{besseldecaygenell}. Let 
\begin{equation} \label{eq:Defelltypecoeff}   A(S,T)  :=   \frac{\avg{v_l}\avg{v_r} }{\ol{\partial}_{2}\lambda\scalar{v_l}{v_r}\sqrt{ \partial_1^2\ol{z}_2}} \,,   \end{equation} 
where the right side is evaluated at $\zeta_1=\ol{\zeta_2} = \zeta^*$. Here
 $v_l=v_l(\zeta^*,\zeta^*)$, $v_r=v_r(\zeta^*,\zeta^*)$ are the unique (up to scaling) left and right eigenvectors with
 positive entries of $L = \frak{D}_{\abs{\frak{b}}^2}^{-1}-{S}=\frak{D}_{\frak{b}}^{-2}-{S}$ corresponding to its zero eigenvalue, where $\frak{b} = \frak{b}(\zeta^*)$ is real. 
 In the following lemma we provide an alternative formula for \eqref{eq:Defelltypecoeff} that also shows it is positive.

\begin{lemma}  \label{lemma:derivativescorr} 
The coefficient $A(S,T)$ from \eqref{eq:Defelltypecoeff} satisfies
\begin{equation} \label{eq:elltypecoeff}   A(S,T)   =    \frac{\avg{v_l}\avg{v_r} }{ \scalar{v_l}{v_r} \sqrt{  2 \langle v_l v_r/\abs{\frak{b}}^2 , ( {1}+F )({1}-F)^{-1}  x^2   \rangle   \langle v_l v_r x /\abs{\frak{b}}^2\rangle }} ,    \end{equation} 
where the right side is again evaluated at $\zeta_1=\ol{\zeta_2} = \zeta^*$. Here we used the notation
\bels{def of F with T}{
F:=\frak{D}_{\abs{\frak{b}}}T\frak{D}_{\abs{\frak{b}}}\,, \qquad x:=(1-F)^{-1}\abs{\frak{b}}\,. 
}
 Furthermore, at  $\zeta_1=\ol{\zeta_2} = \zeta^*$ we have
\[ \ol{\d}_{2} \lambda  > 0, \qquad  \d_1 \ol{z}_2 = - 1, \qquad  \d_1^2 \ol{z}_2 > 0, \]
 as well as $\abs{\ol{\d}_{2} \lambda} \sim 1 $ and $\abs{\d_1^2 \ol{z}_2} \sim 1$. In particular, $A(S,T) \sim 1$. 
\end{lemma}

We will prove this lemma after the conclusion of the proof of Theorem \ref{thm:heatelliptictype}, which we now return to.
Let $\Gamma_{\eps,\delta}=\{ \tau_\eps\2\ee^{\ii \varphi}: \varphi \in [-\delta,\delta]\}$ with $\tau_\eps:=\zeta^* (1+C\1N^{-c \epsilon})$. Then since the operator $L(\zeta_1,\ol{\zeta}_2)$ is invertible for  $|\zeta_1|,|\zeta_2| > \zeta^*$ with the norm of its  inverse bounded by $\Delta_{(\zeta_1, \zeta_2)}^{-1}$ and by the lower bound on $\Delta_{(\zeta_1, \zeta_2)}$ from Proposition~\ref{prp:Delta ok on path} we have  
\bes{
 \oint_\gamma \frac{\dd \zeta_1}{2\pi \ii} \oint_{\ol{\gamma}} \frac{\dd \ol{\zeta}_2}{2\pi \ii}\ee^{t(g\zeta_1 +g \ol{\zeta}_2-2)}K(\zeta_1,\zeta_2)
=&
\int_{\Gamma_{\eps, \delta_1}} \msp{-8}\frac{\dd \zeta_1}{2\pi \ii} \int_{\ol{\Gamma}_{\eps,\delta_2}} \msp{-8}\frac{\dd \ol{\zeta}_2}{2\pi \ii}\ee^{t(g\zeta_1 +g \ol{\zeta}_2-2)}K(\zeta_1,\zeta_2)\\
&+\ord(N^{c \1\eps}\ee^{ t(g\tau_\eps(\cos\delta_1+\cos\delta_2)-2)}),
}
as long as $\delta_1,\delta_2 \sim 1 $. In particular, the error  is exponentially small in $t$.

We now break the kernel $K(\zeta_1,\zeta_2)$ into two parts, corresponding to the spectral projection $P$ associated to the eigenvalue $\lambda(\zeta_1,\overline{\zeta}_2)$ and the complement projection $Q=1-P$, yielding
\[ K(\zeta_1,\zeta_2)= \frac{\avg{P(\zeta_1,\ol{\zeta}_2)1}}{\lambda(\zeta_1,\overline{\zeta}_2)} + \scalar{1}{QL(\zeta_1,\ol{\zeta}_2)^{-1}Q1} \,.   \]
Here  ${L}^{-1}{Q}$ is uniformly bounded on an $\epsilon' \sim 1$ neighborhood of $\zeta^*$ and can  be analytically extended in $\zeta_1,\ol{\zeta}_2$ to $\D_{\eps'}(\zeta^*) = \{\zeta : \abs{\zeta-\zeta^*}<\eps'\}$  because of Assumption~(2.E). Thus the integral of the second term is $O(e^{-2t(1-g \zeta^* (1-\epsilon')})$.										

We now consider the leading term from the projection onto $\lambda(\zeta_1,\overline{\zeta}_2)$, namely
\[
\int_{\Gamma_{\eps, \delta_1}} \frac{\dd \zeta_1}{2\pi \ii} \int_{\ol{\Gamma}_{\eps, \delta_2}} \frac{\dd \ol{\zeta}_2}{2\pi \ii}\ee^{t(g\zeta_1 +g \ol{\zeta}_2-2)}\ \frac{\avg{P(\zeta_1,\ol{\zeta}_2)1}}{\lambda(\zeta_1,\overline{\zeta}_2)}\,.
\]

\begin{figure}[ht] 
\centering
  \includegraphics[width=2in]{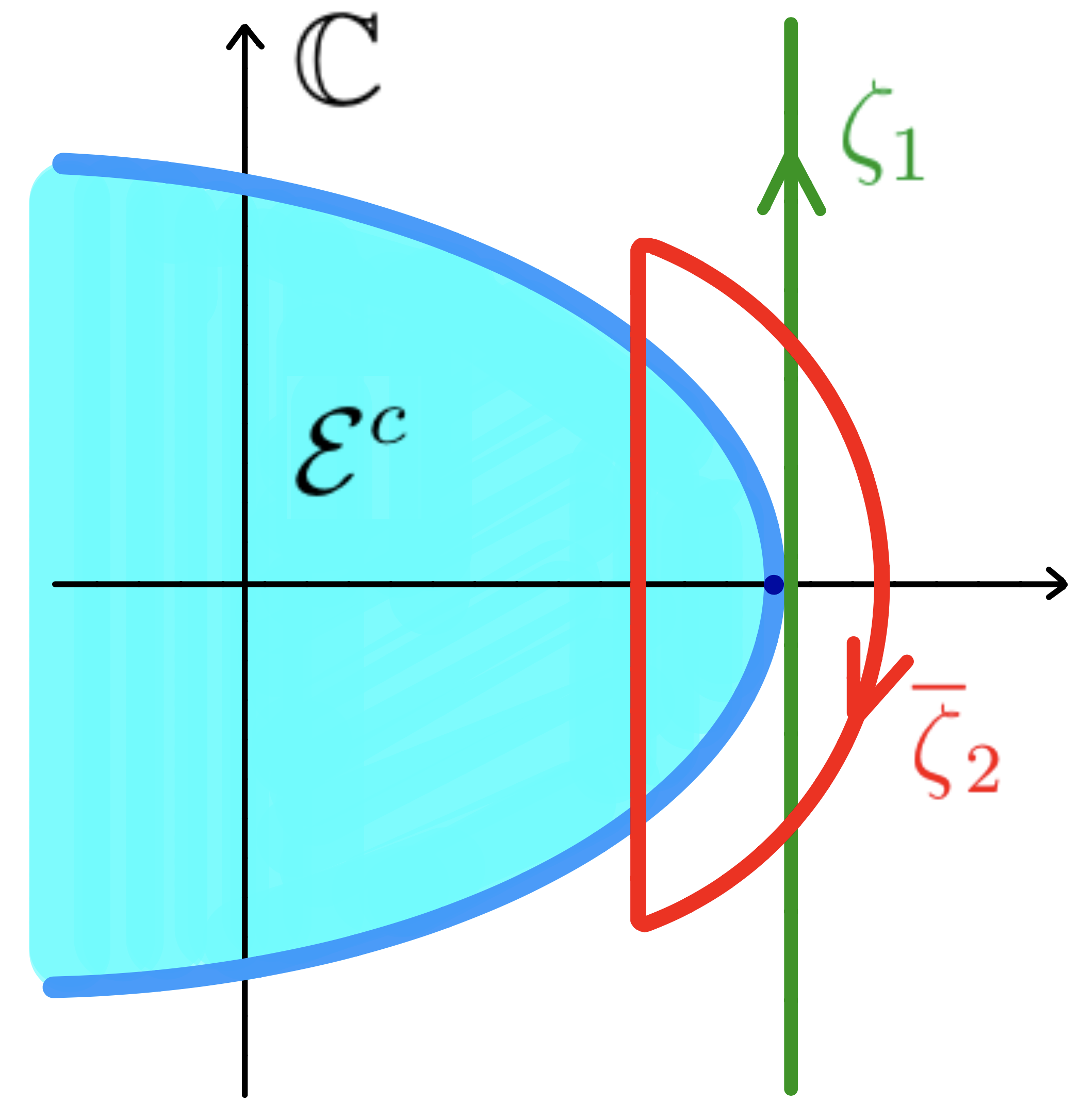}
\caption{Contours of integration}
\label{contour}
\end{figure}

We  set $\delta_1 = 3 \sqrt{\eps'}$ and will close the $\ol{\zeta}_2$-contour to pick up a residue. From Lemma~\ref{lemma:derivativescorr}, $\partial_1 \ol{z}_2(\zeta^*)=-1$. In particular, $\ol{z}_2(\zeta^*+\Delta) = \zeta^* -\Delta + \ord(\abs{\Delta}^2)$ and $\re( \ol{z}_2(\zeta_1)-\zeta^*) =\ord(\eps')$ for any $\zeta_1 \in \Gamma_{\eps, 3\sqrt{\eps'}}$. Thus  we choose $\delta_2\ge C \sqrt{\eps'}$ for sufficiently large $C>0$ and can close the $\ol{\zeta}_2$-contour by adding the line $\tau_\eps[\ee^{\ii \delta_2},\ee^{-\ii \delta_2}]$, once again adding an exponentially small in $t$ error term. Then the inner integral is computed by the residue formula
\[ 
\int_{\Gamma_{\eps, 3\sqrt{\eps'}}} \frac{\dd {\zeta_1}}{2\pi \ii}\ee^{t(g\zeta_1 +g \ol{z}_2(\zeta_1)-2)}\,\frac{\avg{P(\zeta_1,\ol{z}_2(\zeta_1))1}}{\partial_{\ol{\zeta}_2}\lambda(\zeta_1,\ol{z}_2(\zeta_1))}\,.
\]

Expanding $\ol{z}_2$ in the exponent up to second order, once again using that $\partial_1 \ol{z}_2(\zeta^*)=-1$ and then shifting the contour to $\zeta_1= \zeta^* +\ii\Delta$ with $\Delta \in \R$ gives 

\begin{align*}
\int_{\R} \frac{\dd {\Delta}}{2\pi}\ee^{t(2g\zeta^* -g(\partial_1^2\ol{z}_2)\Delta^2/2 -2 )}\,\frac{\avg{v_l}\avg{v_r}}{\ol{\partial}_{2}\lambda}
&+\ord\big((gt)^{-1}\ee^{2t(\zeta^*g-1)}\big)\\
&=  \frac{\avg{v_l}\avg{v_r}\ee^{2t(\zeta^* g-1)}}{\ol{\partial}_{2}\lambda \sqrt{2\pi \partial_1^2\ol{z}_2}}\Big(\frac{1}{\sqrt{gt}}+\ord((gt)^{-1}) \Big)\,,
\end{align*}
where  $\ol{\partial}_{2}\lambda$, and $\partial_1^2\ol{z}_2$ are evaluated at $\zeta^*$, and by Lemma~\ref{lemma:derivativescorr}, the coefficient $A(S,T)$ multiplying the function $\ee^{2t(\zeta^*g-1)}/\sqrt{2\pi gt}$ on the right hand side is $\sim 1$. Furthermore, we used the normalization $\scalar{v_l}{v_r}=1$. 
 We  applied a stationary phase approximation which uses smoothness of the phase and the integrand in $\zeta_1$, where all derivatives are bounded by $N$-independent constants. This effectively controlled smoothness follows from analytic perturbation theory of an isolated non-degenerate eigenvalue. 

This completes the proof of Theorem \ref{thm:heatelliptictype}.
\end{proof}

We conclude by proving the technical lemma used above.

\begin{proof}[Proof of Lemma~\ref{lemma:derivativescorr}]

At $\zeta^*$ we have $\mathfrak{b}=\ol{\mathfrak{b}}$ with $ \mathfrak{b}= \mathfrak{b}(\zeta^*)$ and also $\frak{b}'=\partial_1 \mathfrak{b}(\zeta_1)|_{\zeta_1=\zeta^*} = \ol{\partial_2} \2\ol{\mathfrak{b}(\zeta_2)}|_{\zeta_2=\zeta^*} $.  Dividing \eqref{definition of b} by $\frak{b}$ and then  differentiating with respect to $\zeta$  we find
\bels{B derivatives}{
\mathfrak {b}'=(\frak{D}_{\frak{b}}^{-2}-{T})^{-1}1\,, \qquad \mathfrak {b}'' = 2(\frak{D}_{\frak{b}}^{-2}-{T})^{-1}\big[(\frak{b}')^2/ \frak{b}^{3}\big]\,,
}
which we apply to compute the following derivatives of $L(\zeta_1, \ol{\zeta}_2)$ at $\zeta_1=\ol{\zeta}_2 = \zeta^*$:
\bels{derivatives of L}{
\partial_1{L}=\ol{\partial}_2{L}=-\frak{D}_{\frak{b}'/\frak{b}^{3}}\,, \quad
 \partial_1\ol{\partial}_2{L}=\frak{D}_{(\frak{b}')^2/\frak{b}^{4}} \,,\quad
 \partial_1^2{L}=\ol{\partial}_2^2{L}=\frak{D}_{(2(\frak{b}')^2-\frak{b}\2\frak{b}'')/\frak{b}^4}\,.
 }
We compute the derivatives of $\lambda$ from the  standard first and second order perturbation theory formulas  of an isolated eigenvalue, namely   
\bels{lambda derivatives}{ \begin{gathered}\partial_{a} \lambda = \langle v_l, \partial_{a} {L}   v_r\rangle\,,\\
\partial_{a} \partial_{b} \lambda = \langle  v_l, \partial_{a} \partial_{b} {L}  v_r \rangle + \langle  v_l, \partial_{a} {L} {Q}( \lambda - {L})^{-1} {Q}  \partial_{b} {L}   v_r \rangle + \langle  v_l, \partial_{b} {L} {Q}( \lambda - {L})^{-1} {Q}  \partial_{a} {L}   v_r \rangle\,, \end{gathered}
}
where $\partial_{a},\partial_{b}$ are either $\partial_{1}$ or $\ol{\partial}_{2}$.
The derivatives of $\ol{z}_2$ are then computed by the implicit function theorem applied to the equation $\lambda(\zeta_1, \ol{z}_2(\zeta_1))=0$ to get
\begin{equation} \label{eq:zderiv}
\partial_1\ol{z}_2 = -\frac{\partial_1\lambda}{\ol{\partial}_2\lambda}=-1\,,\qquad \partial_1^2\ol{z}_2 =-\frac{\partial_1^2\lambda+2 (\partial_1\ol{\partial}_2\lambda)\partial\ol{z}_2+(\ol{\partial}_2^2\lambda)(\partial\ol{z}_2)^2}{\ol{\partial}_2\lambda}=2\frac{\partial_1\ol{\partial}_2\lambda-\partial_1^2\lambda}{{\partial}_1\lambda}\,,
\end{equation}
where we used $\partial_1\lambda=\ol{\partial}_2\lambda$  and $\partial_1^2\lambda=\ol{\partial}_2^2\lambda$. Inserting \eqref{derivatives of L} and \eqref{lambda derivatives} yields 
\[
\ol{\partial}_2 \lambda = \avg{v_l v_r x/\abs{\frak{b}}^2}\,, \qquad 
\partial_1^2\ol{z}_2 = 2\frac{\scalar{v_l v_r/\abs{\frak{b}}^2}{(1+F)(1-F)^{-1}x^2}}{\avg{v_l v_r x/\abs{\frak{b}}^2}}\,,
\]
implying the formula \eqref{eq:elltypecoeff} for  the coefficient $A(S,T)$. The norm of the symmetric matrix $F$ from \eqref{def of F with T} is strictly smaller than $1$ since $\norm{F} \le \abs{\varrho} \norm{\frak{D}_{\abs{\frak{b}}}\wt{S}\frak{D}_{\abs{\frak{b}}}} \le \abs{\varrho} \frak{r}(\frak{D}_{\abs{\frak{b}}}^2 S) = \abs{\varrho} $ with $\wt{s}_{ij}= \sqrt{s_{ij}s_{ji}}$ (compare the argument for \eqref{estimate for b2 tilde S}). Together with $\abs{\frak{b}} \sim 1$ and $v_l \sim v_r \sim 1$ from Lemma~\ref{lemm:quanprim} we see that $x \sim 1$ and that $\abs{\ol{\partial}_2 \lambda} \sim \abs{\partial^2\ol{z}_2}\sim 1$, finishing the proof of the lemma.  
\end{proof}

\begin{appendix}

\section{Computation of $\zeta^*$ for Example 3 of Section \ref{sec:examples}}

In our block example~\eqref{Ex3}, 
 the solution $\frak{b}$ to~\eqref{definition of b}
  is constant along the indices corresponding to each block, so it suffices to consider the block-constant solution $\frak{b}(\zeta) = (\frak{b}^1 , \frak{b}^2) \in \C^{N/2} \oplus \C^{N/2}$,
  %  $\frak{b}^1 := \frac{2}{N} \sum_{i=1}^{N/2} \frak{b}_i$ and $\frak{b}^2 := \frac{2}{N} \sum_{i=N/2+1}^{N} \frak{b}_i$ which 
whose coordinates satisfy the equations:
\[ 1  + (\zeta + \frac{\rho}{2} \frak{b}^1(\zeta) )\frak{b}^1(\zeta) = 0 \]
and 
\[ 1 + \zeta \frak{b}^2(\zeta) = 0. \]

Note that the solution to the first equation is not unique, but Proposition \eqref{prp:Solution of vector equation}-\ref{prop:Asymptotics} specifies that $\zeta \frak{b}^1(\zeta) \to -1$ as $|\zeta| \to \infty$, so we choose the solution:
\[ \frak{b}^1(\zeta) = \frac{ -\zeta + \sqrt{\zeta^2 - 4\rho  }}{2 \rho}, \]
where the square root is chosen with a branch cut along the segment $[-2\sqrt{\varrho},2\sqrt{\varrho}]$ so that $\sqrt{\zeta^2-4 \varrho} - \zeta \to 0$ as $|\zeta| \to \infty$.

Then the set $\mathcal{R}$ is all $\zeta$ such that the operator $\frak{D}_{\abs{\frak{b}(\zeta)}^2} S$ has spectral radius less than 1, which is equivalent to the operator
\[  \begin{pmatrix} |\frak{b}^1(\zeta)|^2 & 0 \\ 0 & |\frak{b}^2(\zeta)|^2 \end{pmatrix}  
\begin{pmatrix} 1 &\sigma^2_{12} \\ \sigma^2_{21} & \sigma^2_{22} \end{pmatrix}\]
having spectral radius less than 1.%, where we set $s^{ij} := N \sigma^2_{ij}$. 

%Noting that $1/b_2(\zeta) = -\zeta = \rho b_1(\zeta) + b_1(\zeta)^{-1}$ 
By the continuity of $\frak{b}$, the boundary of $\mathcal{R}$ is determined by  all the $\zeta$'s
 such that 1 the largest eigenvalue of this operator, which furthermore implies
\[ \det\left(   \begin{pmatrix} |\frak{b}^1(\zeta)|^{-2} & 0 \\ 0 & |\frak{b}^2(\zeta)|^{-2} \end{pmatrix}  -
\begin{pmatrix} 1 &\sigma^2_{12} \\ \sigma^2_{21} & \sigma^2_{22} \end{pmatrix} \right) = 0\]
holds.
Rearranging gives:
%The eigenvalues of the this matrix are less than 1 if:
\begin{equation}\label{cub}
% \left(\frac{1}{2} - |\frak{b}^1(\zeta)|^{-2} \right)\left(\frac{\sigma^2_{22}}{2} - |\frak{b}^2(\zeta)|^{-2}\right) = \frac{1}{4} \sigma^2_{12}\sigma^2_{21}. 
 \left(\frac{1}{2} - |\frak{b}^1(\zeta)|^{-2} \right)\left( \sigma^2_{22} - |\frak{b}^2(\zeta)|^{-2}\right) =\sigma^2_{12}\sigma^2_{21}. 
\end{equation}

Noting that $1/\frak{b}^2(\zeta) = -\zeta = \frac{\rho}{2} \frak{b}^1(\zeta) + \frak{b}^1(\zeta)^{-1}$ we can solve for $\frak{b}^1(\zeta)$ and then determine which $\zeta$ correspond to the boundary. 
 
In Section~\ref{sec:ellcorr} we showed that if $\rho \geq 0$ then $\zeta^{*} = \max_{\zeta \in \cal{R}^c}\zeta$ is real. Furthermore from the symmetry of the equation we have that $\frak{b}(\zeta)$ is real and negative for $\zeta >0$.

Using that $|\frak{b}(\zeta^*)|= -\frak{b}(\zeta^*)$, we rearrange the equation~\eqref{cub} to:
\[-\frac{\rho^2}{8} |\frak{b}^1|^6  + (-\frac{\rho}{2}+\frac{\rho^2}{4} + \frac{\sigma^2_{22}}{2} - \sigma^2_{12} \sigma^2_{21} ) |\frak{b}^1|^4 + (\rho -\sigma^2_{22} -\frac{1}{2})|\frak{b}^1|^2   + 1 =0,
\]
%\[-\frac{\rho}{4} |\frak{b}_1|^6  + (\frac{\rho}{2} + \frac{s^{22}}{4} - \frac{s^{12} s^{21}}{4}) |\frak{b}_1|^4 - (  \frac{s^{22}}{2}-\frac{1}{2})|\frak{b}_1|^2   + 1 ==0\]
%\[ -\rho^2 |b|^6  + (\rho^2 + s^{22} - s^{12} s^{21}) |b|^4 + (2 \rho - s^{22}-1)|b|^2   + 1 ==0  \]
which is a cubic polynomial in $|\frak{b}^1|^2$. Then let $\frak{b}^*$ be the negative square root of the root of the polynomial:
\[-\frac{\rho^2}{8} x^3  + (-\frac{\rho}{2}+\frac{\rho^2}{4} + \frac{\sigma^2_{22}}{2} -  \sigma^2_{12} \sigma^2_{21}) x^2 + (\rho - \sigma^2_{22} -\frac{1}{2})x   + 1 =0\]
such that 
\[   \begin{pmatrix} (\frak{b}^*)^2 & 0 \\ 0 & (\frac{ \rho}{2} \frak{b}^* - \frac{1}{\frak{b}^{*}})^{-2} \end{pmatrix}  
\begin{pmatrix} 1 &\sigma^2_{12} \\ \sigma^2_{21} & \sigma^2_{22} \end{pmatrix}\]
has spectral radius 1. Note that
the other roots correspond to 1 being the smallest and not the largest (in modulus) eigenvalue of this operator.

The left most point of the support  of $\cal{R}^c$ is then given by \[\zeta^* = -\frac{ \rho}{2} \frak{b}^* - \frac{1}{\frak{b}^{*}}  .\]

\end{appendix}

\bibliographystyle{abbrv}
\bibliography{2pt}

\end{document}